\documentclass[reqno,10pt]{amsart}

\usepackage{enumitem}                            % For lists of items (works with \autoref)
\usepackage[hidelinks,colorlinks=true]{hyperref} % Hyperlinks
\hypersetup{allcolors=[rgb]{0,0.31,0.62}}
\usepackage[utf8]{inputenc}                      % We probably don't need this
\usepackage{mathtools}                           % For \coloneq
\usepackage{stix2}                               % Nice font (an even nicer font is notomath)
\usepackage{thmtools}                            % In order to use \autoref.
\usepackage{tikz}                                % For drawing diagrams, etc.
\usetikzlibrary{decorations.markings}            % For drawing 'suspension' morphisms
\usetikzlibrary{matrix}                          % For typesetting diagrams

\newcommand{\A}{\category{A}}
\newcommand{\C}{\category{C}}
\newcommand{\category}[1]{\mathcal{#1}}
\newcommand{\T}{\category{T}}
\newcommand{\U}{\category{U}}

\newcommand{\nang}{\mathscr{N}}

\newcommand{\toda}[1]{\langle #1\rangle}
\newcommand{\todai}[2]{\toda{#1}_{[#2]}}
\newcommand{\todacc}[1]{\toda{#1}_{\text{cc}}}
\newcommand{\todafc}[1]{\toda{#1}_{\text{fc}}}
\newcommand{\todaff}[1]{\toda{#1}_{\text{ff}}}

\newcommand{\todaSS}[1]{\toda{#1}_{\text{SS}}}

\newcommand{\massey}[1]{\langle\!\langle #1 \rangle\!\rangle}

\DeclareMathOperator{\add}{add}
\DeclareMathOperator{\derived}{D}
\DeclareMathOperator{\dg}{dg}
\DeclareMathOperator{\gldim}{gldim}

\DeclareMathOperator{\htpy}{K}
\let\mod\relax
\DeclareMathOperator{\mod}{mod} % category of finitely generated modules
\DeclareMathOperator{\Mod}{Mod}
\DeclareMathOperator{\dgmod}{\Mod_{\dg}}
\DeclareMathOperator{\proj}{proj}
\DeclareMathOperator{\pd}{pd}
\DeclareMathOperator{\Span}{span}
\DeclareMathOperator{\TT}{T}
\DeclareMathOperator{\Ext}{Ext}

\newcommand{\op}{\oplus}

\let\phi\varphi

\newcommand{\Z}{\mathbb{Z}}

\newcommand{\field}{\mathbb{k}}

\newcommand{\rep}[2]{\begin{array}{c} #1\\ #2\end{array}}

\newcommand{\Dd} {\mathfrak D}

\newcommand{\tild}{\widetilde}

\newcommand{\ff}{\mathrm{ff}} %finitely generated free
\newcommand{\Id}{\mathrm{Id}} %identity functor

\newcommand{\diagram}[3]{\matrix (#1) [matrix of math nodes,%
  row sep = {#2},%
  column sep = {#3},%
  text height = 1.5ex,%
  text depth=0.25ex]}
\newcommand{\smalldiagram}[3]{\matrix (#1) [matrix of math nodes,%
  row sep = {#2},%
  column sep = {#3},%
  text height = 1.5ex,%
  text depth=0.25ex,%
  font = \small]}

\newcommand{\downequal}[3]{
  ([xshift=0.1em] {#1}-#2.south) edge[-] ([xshift=0.1em] {#1}-#3.north)
  ([xshift=-0.1em] {#1}-#2.south) edge[-] ([xshift=-0.1em] {#1}-#3.north)}

\newcommand{\smat}[1]{\left[\begin{smallmatrix*}[r]#1\end{smallmatrix*}\right]}
\tikzset{
  suspension/.style = {postaction = decorate,
    decoration = {
      markings,
      mark = at position 0.3 with {\draw[-] (0,-0.075) -- (0,0.075);}
    },
  },
}

\tikzset{
  encircled/.style = {
    ->,
    shorten < = 2.5pt,
    shorten > = 2.5pt,
  }
}

\theoremstyle{plain}
\newtheorem{theorem}{Theorem}[section]
\newtheorem{proposition}[theorem]{Proposition}
\newtheorem{corollary}[theorem]{Corollary}
\newtheorem{lemma}[theorem]{Lemma}

\theoremstyle{definition}
\newtheorem{definition}[theorem]{Definition}
\newtheorem{example}[theorem]{Example}
\newtheorem{remark}[theorem]{Remark}

\newcommand{\Def}[1]{\textbf{\boldmath{#1}}} %the word being defined

\makeatletter
\def\namedlabel#1#2{\begingroup
  #2%
  \def\@currentlabel{#2}%
  \phantomsection\label{#1}\endgroup
}
\makeatother

\def\equationautorefname~#1\null{Diagram~(#1)\null}

\title{Toda brackets in $n$-angulated categories}
\date{\today}

\author{Martin Frankland}
\address{Department of Mathematics and Statistics, University of Regina, 3737
  Wascana Parkway Regina, Saskatchewan, S4S 0A2, Canada}
\email{martin.frankland@uregina.ca}
\author{Sebastian H.\ Martensen}
\address{Department of Mathematical Sciences, NTNU, N-7491 Trondheim, Norway}
\email{sebastian.martensen@ntnu.no}
\author{Marius Thaule}
\address{Department of Mathematical Sciences, NTNU, N-7491 Trondheim, Norway}
\email{marius.thaule@ntnu.no}

\begin{document}

%
% Abstract
%
\begin{abstract}
  We introduce Toda brackets for $n$-angulated categories and show that the
  various definitions of Toda brackets coincide. We prove juggling formulas for
  these Toda brackets generalizing the triangulated case. Following that, we
  generalize a theorem due to Heller in the triangulated setting to the setting
  of $n$-angulated categories. We also provide several examples of computing
  Toda brackets for $n$-angulated categories. Finally, for an $n$-angulated
  category sitting in a triangulated category as in the setup of Geiss, Keller
  and Oppermann, we show that Toda brackets in the $n$-angulated sense coincide
  with $n$-fold Toda brackets in the triangulated sense up to an explicit sign.
\end{abstract}



\maketitle

%
% Introduction
%
\section{Introduction}
\label{sec:intro}

Triangulated categories were first introduced, independently, by Puppe (in 1962)
and Verdier (in 1963) in algebraic topology and algebraic geometry,
respectively. Triangulated categories are prominently featured in areas such as
representation theory, commutative algebra, algebraic geometry and algebraic
topology. A generalization of triangulated categories, known as $n$-angulated
categories, was introduced by Geiss, Keller and Oppermann \cite{GKO}, such that
$3$-angulated categories coincide with triangulated categories. Examples of
$n$-angulated categories include certain $(n - 2)$-cluster tilting subcategories
of triangulated categories \cite{GKO}, i.e., algebraic $n$-angulated categories
in the sense of \cite{J} and categories of finitely generated free modules over
commutative local rings with principal maximal ideal squaring to zero
\cite{BerghJT16}.

Toda brackets were originally introduced by Toda (in 1952) to compute homotopy
groups of spheres and play an important role in homotopy theory \cite{Toda62}.
They have been used extensively in computations of stable homotopy groups of
spheres \cite{Isaksen19}. Toda brackets may be defined for an arbitrary
triangulated category. Christensen and Frankland proved in \cite{CF} how the
differential $d_r$ in the Adams spectral sequence constructed from a
triangulated category with a projective or an injective class, cf.\ \cite{C},
may be described using $(r + 1)$-fold Toda brackets. Another application of Toda
brackets and triangulated categories include a theorem due to Heller
\cite[Theorem~13.2]{H} stating that a triangle
\begin{center}
  \begin{tikzpicture}
    \diagram{d}{2em}{2em}{
      X & Y & Z & \Sigma X\\
    };

    \path[->, font = \scriptsize, auto]
    (d-1-1) edge node{$f$} (d-1-2)
    (d-1-2) edge node{$g$} (d-1-3)
    (d-1-3) edge node{$h$} (d-1-4);
  \end{tikzpicture}
\end{center}
in a triangulated category $\T$ is \emph{distinguished} if and only if it is
Yoneda exact and the Toda bracket $\toda{h,g,f}\subseteq \T(\Sigma X, \Sigma X)$
contains the identity morphism $1_{\Sigma X}$. See also \cite[Theorem B.1]{CF}.

\subsection*{Organization and main results}
The main goal for this paper is to introduce Toda brackets for an arbitrary
$n$-angulated category. We should note that none of our proofs rely on the
octahedral axiom for $n$-angulated categories which allows us to work with
pre-$n$-angulated categories.

In \autoref{sec:prelim}, we review the definitions of the iterated cofiber,
iterated fiber and fiber-cofiber Toda brackets in the setting of a triangulated
category. We also briefly review the definition of a pre-$n$-angulated category.
In \autoref{sec:toda_nang}, we define Toda brackets in the setting of
pre-$n$-angulated categories. \autoref{sec:toda} contains one of our main
results, namely that under a mild condition, we may speak of \emph{the} Toda
bracket in pre-$n$-angulated categories, in the sense that the iterated cofiber,
iterated fiber and fiber-cofiber Toda brackets all coincide
(\autoref{thm:main}). In \autoref{sec:intermediate}, we introduce what we call
$[i]$-intermediate Toda brackets in the setting of pre-$n$-angulated categories,
with the $[1]$-intermediate Toda bracket being the same as the iterated cofiber
Toda bracket and the $[n]$-intermediate Toda bracket being the same as the
iterated fiber Toda bracket. For $1 < i < n$, the $[i]$-intermediate Toda
brackets resemble the fiber-cofiber Toda bracket in the triangulated setting. We
prove that under a mild condition, the $[i]$-intermediate Toda brackets coincide
with the previous definitions (\autoref{thm:all_coincide}).

In \autoref{sec:juggling_formulas}, we prove that the juggling formulas for Toda
brackets in the triangulated case carry over to the pre-$n$-angulated case
(\autoref{prop:subadd} and \autoref{prop:juggling_fc}). For the case $n = 3$ we
get a collected reference for $3$-fold Toda brackets in pretriangulated
categories. In \autoref{sec:toda_determines_nang}, we prove that Toda brackets
determine the pre-$n$-angulation (\autoref{prop:NAngleCriterion}), generalizing
a theorem due to Heller \cite{H}. \autoref{sec:examples} contains several
examples of computing Toda brackets for various $n$-angulated categories ranging
from exotic $n$-angulated categories to $n$-angulated categories arising in
connection to cluster tilting categories.

Finally, in \autoref{sec:LongerHigher}, for an $n$-angulated category sitting in
a triangulated category as in the setup of \cite{GKO}, we show that Toda
brackets in the $n$-angulated sense coincide with $n$-fold Toda brackets in the
triangulated sense up to an explicit sign (\autoref{thm:LongerHigher}). We
deduce that Toda brackets in an enhanced $n$-angulated category (in the sense of
\cite{JM}) coincide with Massey products up to a sign (\autoref{cor:Massey}).

\subsection*{Acknowledgments}

We thank Dan Isaksen, Steffen Oppermann, and Laertis Vaso for helpful
discussions. Frankland acknowledges the support of the Natural Sciences and
Engineering Research Council of Canada (NSERC), grant RGPIN-2019-06082.

%
% Preliminaries
%
\section{Preliminaries}
\label{sec:prelim}

\subsection*{Toda brackets in triangulated categories}
\label{subsec:toda_tri}

We recall the iterated cofiber, the iterated fiber and fiber-cofiber Toda
bracket in triangulated categories from \cite[Definition~3.1]{CF}.

\begin{definition}
  \label{def:TodaBracket_tri}
  Let $\T$ be a triangulated category and let
  \begin{center}
    \begin{tikzpicture}
      \diagram{d}{2em}{2em}{
        X_1 & X_2 & X_3 & X_4\\
      };
      
      \path[->, font = \scriptsize, auto]
      (d-1-1) edge node{$f_1$} (d-1-2)
      (d-1-2) edge node{$f_2$} (d-1-3)
      (d-1-3) edge node{$f_3$} (d-1-4);
    \end{tikzpicture}
  \end{center}
  be a diagram in $\T$. We define subsets of $\T(\Sigma X_1, X_4)$ as follows:
  \begin{itemize}
  \item The \Def{iterated cofiber Toda bracket} $\todacc{f_3,f_2,f_1}\subseteq
    \T(\Sigma X_1, X_4)$ consists of all morphisms $\psi \colon \Sigma X_1 \to
    X_4$ that appear in a commutative diagram
    \begin{center}
      \begin{tikzpicture}
        \diagram{d}{2em}{2em}{
          X_1 & X_2 & Y_3 & \Sigma X_1\\
          X_1 & X_2 & X_3 & X_4\\
        };
        
        \path[->, auto, font = \scriptsize]
        (d-1-1) edge node{$f_1$} (d-1-2)
        \downequal{d}{1-1}{2-1}
        (d-1-2) edge node{$y_2$} (d-1-3)
        \downequal{d}{1-2}{2-2}
        (d-1-3) edge node{$y_3$} (d-1-4)
        edge node{$\phi$} (d-2-3)
        (d-1-4) edge node{$\psi$} (d-2-4)
        
        (d-2-1) edge node{$f_1$} (d-2-2)
        (d-2-2) edge node{$f_2$} (d-2-3)
        (d-2-3) edge node{$f_3$} (d-2-4);
      \end{tikzpicture}
    \end{center}
    where the top row is distinguished.
  \item The \Def{iterated fiber Toda bracket} $\todaff{f_3,f_2,f_1}\subseteq
    \T(\Sigma X_1,X_4)$ consists of all morphisms $\Sigma \delta\colon \Sigma
    X_1\to X_4$ where $\delta$ appears in a commutative diagram
    \begin{center}
      \begin{tikzpicture}
        \diagram{d}{2em}{2em}{
          X_1 & X_2 & X_3 & X_4\\
          \Sigma^{-1} X_4 & W_2 & X_3 & X_4\\
        };
        
        \path[->, auto, font = \scriptsize]
        (d-1-1) edge node{$f_1$} (d-1-2)
                edge node{$\delta$} (d-2-1)
        (d-1-2) edge node{$f_2$} (d-1-3)
                edge node{$\gamma$} (d-2-2)
        (d-1-3) edge node{$f_3$} (d-1-4)
        \downequal{d}{1-3}{2-3}
        \downequal{d}{1-4}{2-4}
        
        (d-2-1) edge node{$w_1$} (d-2-2)
        (d-2-2) edge node{$w_2$} (d-2-3)
        (d-2-3) edge node{$f_3$} (d-2-4);
      \end{tikzpicture}
    \end{center}
    where the bottom row is distinguished.
  \item The \Def{fiber-cofiber Toda bracket} $\todafc{f_3,f_2,f_1}\subseteq
    \T(\Sigma X_1, X_4)$ consists of all composites $\Sigma(\beta^2_1 \beta^1_1)
    \colon \Sigma X_1 \to X_4$, where $\beta^1_1$ and $\beta^2_1$ appear in a
    commutative diagram
    \begin{center}
      \begin{tikzpicture}
        \diagram{d}{2em}{2em}{
          X_1 & X_2 && \Sigma X_1\\
          Z^2_1 & X_2 & X_3 & \Sigma Z^2_1\\
          && X_3 & X_4.\\
        };
            
        \path[->, auto, font = \scriptsize]
        (d-1-1) edge node{$f_1$} (d-1-2)
                edge node{$\beta^1_1$} (d-2-1)
        \downequal{d}{1-2}{2-2}
        (d-1-4) edge node{$\Sigma\beta^1_1$} (d-2-4)
        
        (d-2-1) edge node{$z^2_1$} (d-2-2)
        (d-2-2) edge node{$f_2$} (d-2-3)
        (d-2-3) edge node{$z^2_3$} (d-2-4)
        \downequal{d}{2-3}{3-3}
        (d-2-4) edge node{$\Sigma\beta^2_1$} (d-3-4)
        
        (d-3-3) edge node{$f_3$} (d-3-4);
      \end{tikzpicture}
    \end{center}
    where the middle row is distinguished.
  \end{itemize}
\end{definition}

The notation was chosen to be consistent with \autoref{def:TodaNang}.

\begin{remark}
  The iterated cofiber, iterated fiber and fiber-cofiber definitions of Toda
  brackets coincide. See, e.g., \cite[Proposition~3.3]{CF} for a proof.
\end{remark}

%
% Definition of n-angulated categories
%
\subsection*{Pre-\texorpdfstring{$n$}{n}-angulated categories}
\label{subsec:nang}

Let $\C$ be an additive category, $\Sigma \colon \C \to \C$ an automorphism, and
let $n$ be an integer greater than or equal to three.

We will refer to a sequence of objects and morphisms in $\C$ of the form
\begin{center}
  \begin{tikzpicture}
    \diagram{d}{2em}{2em}{
      X_\bullet \coloneq \bigl(X_1 & X_2 & \cdots & X_n & \Sigma X_1\bigr)\\
    };

    \path[->, font = \scriptsize, auto]
    (d-1-1) edge node{$f_1$} (d-1-2)
    (d-1-2) edge node{$f_2$} (d-1-3)
    (d-1-3) edge node{$f_{n - 1}$} (d-1-4)
    (d-1-4) edge node{$f_n$} (d-1-5);
  \end{tikzpicture}
\end{center}
as an \Def{$n$-$\Sigma$-sequence}. The \Def{left and right rotations} of
$X_\bullet$ are the two $n$-$\Sigma$-sequences
\begin{center}
  \begin{tikzpicture}
    \diagram{d}{2em}{3.5em}{
      X_\bullet [1] \coloneq \bigl(X_2 & X_3 & \cdots & \Sigma X_1 & \Sigma X_2\bigr)\\
    };
    
    \path[->, font = \scriptsize, auto]
    (d-1-1) edge node{$f_2$} (d-1-2)
    (d-1-2) edge node{$f_3$} (d-1-3)
    (d-1-3) edge node{$f_n$} (d-1-4)
    (d-1-4) edge node{$(-1)^n \Sigma f_1$} (d-1-5);
  \end{tikzpicture}
\end{center}
and
\begin{center}
  \begin{tikzpicture}
    \diagram{d}{2em}{4em}{
      X_\bullet [-1] \coloneq \bigl(\Sigma^{-1} X_n & X_1 & \cdots & X_{n - 1} &
      X_n\bigr),\\
    };
    
    \path[->, font = \scriptsize, auto]
    (d-1-1) edge node{$(-1)^n \Sigma^{-1} f_n$} (d-1-2)
    (d-1-2) edge node{$f_1$} (d-1-3)
    (d-1-3) edge node{$f_{n - 2}$} (d-1-4)
    (d-1-4) edge node{$f_{n - 1}$} (d-1-5);
    \end{tikzpicture}
\end{center}
respectively. If $X_\bullet$ is rotated to the left $k$ times we write
$X_\bullet [k]$, and similarly, if $X_\bullet$ is rotated to the right $k$ times
we write $X_\bullet [-k]$. A \Def{trivial} $n$-$\Sigma$-sequence is a sequence
of the form
\begin{center}
  \begin{tikzpicture}
    \diagram{d}{2em}{2em}{
      (TX)_\bullet \coloneq \bigl(X & X & 0 & \cdots & 0 & \Sigma X_1\bigr)\\
    };
    
    \path[->, font = \scriptsize, auto]
    (d-1-1) edge node{$1$} (d-1-2)
    (d-1-2) edge (d-1-3)
    (d-1-3) edge (d-1-4)
    (d-1-4) edge (d-1-5)
    (d-1-5) edge (d-1-6);
  \end{tikzpicture}
\end{center}
or any of its rotations.

A \Def{morphism} $\phi\colon X_{\bullet} \to Y_{\bullet}$ of
$n$-$\Sigma$-sequences is a sequence $\phi = (\phi_1,\phi_2,\ldots,\phi_n)$ of
morphisms in $\C$ such that the diagram
\begin{center}
  \begin{tikzpicture}
    \diagram{d}{2em}{2em}{
      X_1 & X_2 & X_3 & \cdots & X_n & \Sigma X_1\\
      Y_1 & Y_2 & Y_3 & \cdots & Y_n & \Sigma Y_1\\
    };
    
    \path[->, auto, font = \scriptsize]
    (d-1-1) edge node {$f_1$} (d-1-2)
            edge node {$\phi_1$} (d-2-1)
    (d-1-2) edge node {$f_2$} (d-1-3)
            edge node {$\phi_2$} (d-2-2)
    (d-1-3) edge node {$f_3$} (d-1-4)
            edge node {$\phi_3$} (d-2-3)
    (d-1-4) edge node {$f_{n - 1}$} (d-1-5)
    (d-1-5) edge node {$f_n$} (d-1-6)
            edge node {$\phi_n$} (d-2-5)
    (d-1-6) edge node {$\Sigma\phi_1$} (d-2-6)
    
    (d-2-1) edge node {$g_1$} (d-2-2)
    (d-2-2) edge node {$g_2$} (d-2-3)
    (d-2-3) edge node {$g_3$} (d-2-4)
    (d-2-4) edge node {$g_{n - 1}$} (d-2-5)
    (d-2-5) edge node {$g_n$} (d-2-6);
  \end{tikzpicture}
\end{center}
commutes. It is an \Def{isomorphism} if $\phi_1,\phi_2,\ldots,\phi_n$ are all
isomorphisms in $\C$, and a \Def{weak isomorphism} if $\phi_i$ and $\phi_{i +
  1}$ are isomorphisms for some $1 \leq i \leq n$ (with $\phi_{n + 1} \coloneq
\Sigma \phi_1$).

\begin{definition}\label{def:NAng}
  Let $\nang$ be a collection of $n$-$\Sigma$-sequences in $\C$. Then the triple
  $(\C, \Sigma, \nang)$ is a \Def{pre-$n$-angulated category} if the following
  three axioms are satisfied:

  \begin{description}[leftmargin = 2\parindent, style = multiline]
  \item[\namedlabel{N1}{(N1)}]
    \begin{description}[style = multiline]
    \item[\namedlabel{N1a}{(a)}] $\nang$ is closed under direct sums, direct
      summands and isomorphisms of $n$-$\Sigma$-sequences.
    \item[\namedlabel{N1b}{(b)}] For all $X \in \C$, the trivial
      $n$-$\Sigma$-sequence $(TX)_\bullet$ belongs to $\nang$.
    \item[\namedlabel{N1c}{(c)}] For each morphism $f \colon X_1 \to X_2$ in
      $\C$, there exists an $n$-$\Sigma$-sequence in $\nang$ whose first
      morphism is $f$.\sloppy
    \end{description}
  \item[\namedlabel{N2}{(N2)}] An $n$-$\Sigma$-sequence belongs to $\nang$ if
    and only if its left rotation belongs to $\nang$.
  \item[\namedlabel{N3}{(N3)}] Given the solid part of the commutative diagram
    \begin{center}
      \begin{tikzpicture}
        \diagram{d}{2em}{2em}{
          X_1 & X_2 & X_3 & \cdots & X_n & \Sigma X_1\\
          Y_1 & Y_2 & Y_3 & \cdots & Y_n & \Sigma Y_1\\
        };
        
        \path[->, auto, font = \scriptsize]
        (d-1-1) edge node {$f_1$} (d-1-2)
                edge node {$\phi_1$} (d-2-1)
        (d-1-2) edge node {$f_2$} (d-1-3)
                edge node {$\phi_2$} (d-2-2)
        (d-1-3) edge node {$f_3$} (d-1-4)
                edge[densely dashed] node {$\phi_3$} (d-2-3)
        (d-1-4) edge node {$f_{n - 1}$} (d-1-5)
        (d-1-5) edge node {$f_n$} (d-1-6)
                edge[densely dashed] node {$\phi_n$} (d-2-5)
        (d-1-6) edge node {$\Sigma\phi_1$} (d-2-6)
        
        (d-2-1) edge node {$g_1$} (d-2-2)
        (d-2-2) edge node {$g_2$} (d-2-3)
        (d-2-3) edge node {$g_3$} (d-2-4)
        (d-2-4) edge node {$g_{n - 1}$} (d-2-5)
        (d-2-5) edge node {$g_n$} (d-2-6);
      \end{tikzpicture}
    \end{center}
    with rows in $\nang$, the dotted morphisms exist and give a morphism of
    $n$-$\Sigma$-sequences.\sloppy
  \end{description}

  We refer to the collection $\nang$ as a \Def{pre-$n$-angulation} of the
  category $\C$ (relative to the automorphism $\Sigma$), and the
  $n$-$\Sigma$-sequences in $\nang$ as \Def{$n$-angles}.
\end{definition}

\begin{remark}
  \label{rem:nang}
  A pre-$n$-angulated category $\C$ is called an \Def{$n$-angulated category} if
  the collection $\nang$ of $n$-angles also satisfies a fourth axiom. In the
  original definition, Geiss, Keller and Oppermann \cite{GKO} used a ``mapping
  cone axiom'' inspired by Neeman's mapping cone axiom in the triangulated case.
  Bergh and Thaule proved in \cite{BT} that this is equivalent to a ``higher
  octahedron axiom'' inspired by Verdier's original octahedron axiom for
  triangulated categories. We refer the reader to \cite{GKO} and \cite{BT} for
  the exact details of the fourth axiom.
\end{remark}

%
% Toda brackets in n-angulated categories
%
\section{Toda brackets in pre-\texorpdfstring{$n$}{n}-angulated categories}
\label{sec:toda_nang}

We now define the iterated cofiber, iterated fiber and fiber-cofiber Toda
brackets for pre-$n$-angulated categories. We fix a pre-$n$-angulated category
$\C$ for the remainder of the paper.
\begin{definition}\label{def:TodaNang}
  Let
  \begin{equation}
    \label{eq:seq}
    \begin{aligned}
      \begin{tikzpicture}
        \diagram{d}{2em}{2em}{
          X_1 & X_2 & X_3 & \cdots & X_n & X_{n+1}\\
        };
        
        \path[->, auto, font = \scriptsize]
        (d-1-1) edge node{$f_1$} (d-1-2)
        (d-1-2) edge node{$f_2$} (d-1-3)
        (d-1-3) edge node{$f_3$} (d-1-4)
        (d-1-4) edge node{$f_{n-1}$} (d-1-5)
        (d-1-5) edge node{$f_n$} (d-1-6);
      \end{tikzpicture}
    \end{aligned}
  \end{equation}
  be a diagram in $\C$. We define subsets of $\C(\Sigma X_1, X_{n + 1})$ as
  follows:

  \begin{itemize}
  \item The \Def{iterated cofiber Toda bracket}
    $\todacc{f_n,\dots,f_2,f_1}\subseteq \C(\Sigma X_1, X_{n + 1})$ consists of
    all morphisms $\psi \colon \Sigma X_1 \to X_{n + 1}$ that appear in a
    commutative diagram
    \begin{equation}
      \label{eq:TodaCC}
      \begin{aligned}
        \begin{tikzpicture}
          \diagram{d}{2em}{2em}{
            X_1 & X_2 & Y_3 & Y_4 & \cdots & Y_n & \Sigma X_1\\
            X_1 & X_2 & X_3 & X_4 & \cdots & X_n & X_{n + 1}\\
          };
          
          \path[->, auto, font = \scriptsize]
          (d-1-1) edge node{$f_1$} (d-1-2)
          \downequal{d}{1-1}{2-1}
          (d-1-2) edge node{$y_2$} (d-1-3)
          \downequal{d}{1-2}{2-2}
          (d-1-3) edge node{$y_3$} (d-1-4)
                  edge node{$\phi_3$} (d-2-3)
          (d-1-4) edge node{$\phi_4$} (d-2-4)
                  edge node{$y_4$} (d-1-5)
          (d-1-5) edge node{$y_{n-1}$} (d-1-6)
          (d-1-6) edge node{$y_n$} (d-1-7)
                  edge node{$\phi_{n}$} (d-2-6)
          (d-1-7) edge node{$\psi$} (d-2-7)
          
          (d-2-1) edge node{$f_1$} (d-2-2)
          (d-2-2) edge node{$f_2$} (d-2-3)
          (d-2-3) edge node{$f_3$} (d-2-4)
          (d-2-4) edge node{$f_4$} (d-2-5)
          (d-2-5) edge node{$f_{n-1}$} (d-2-6)
          (d-2-6) edge node{$f_n$} (d-2-7);
        \end{tikzpicture}
      \end{aligned}
    \end{equation}
    where the top row is an $n$-angle extension of $f_1$.

  \item Dually, the \Def{iterated fiber Toda bracket}
    $\todaff{f_n,\dots,f_2,f_1}\subseteq \C(\Sigma X_1,X_{n + 1})$ consists of
    all morphisms $\Sigma \delta\colon \Sigma X_1\to X_{n + 1}$ where $\delta$
    appears in a commutative diagram
    \begin{center}
      \begin{tikzpicture}
        \smalldiagram{d}{2em}{2em}{
          X_1 & X_2 & X_3 & \cdots & X_{n - 1} & X_n & X_{n + 1}\\
          \Sigma^{-1} X_{n + 1} & W_2 & W_3 & \cdots & W_{n-1} & X_n & X_{n + 1}\\
        };
        
        \path[->, auto, font = \scriptsize]
        (d-1-1) edge node{$f_1$} (d-1-2)
                edge node{$\delta$} (d-2-1)
        (d-1-2) edge node{$f_2$} (d-1-3)
                edge node{$\gamma_2$} (d-2-2)
        (d-1-3) edge node{$f_3$} (d-1-4)
                edge node{$\gamma_3$} (d-2-3)
        (d-1-4) edge node{$f_{n - 2}$} (d-1-5)
        (d-1-5) edge node{$f_{n - 1}$} (d-1-6)
                edge node{$\gamma_{n - 1}$} (d-2-5)
        (d-1-6) edge node{$f_n$} (d-1-7)
        \downequal{d}{1-6}{2-6}
        \downequal{d}{1-7}{2-7}
        
        (d-2-1) edge node{$w_1$} (d-2-2)
        (d-2-2) edge node{$w_2$} (d-2-3)
        (d-2-3) edge node{$w_3$} (d-2-4)
        (d-2-4) edge node{$w_{n - 2}$} (d-2-5)
        (d-2-5) edge node{$w_{n - 1}$} (d-2-6)
        (d-2-6) edge node{$f_n$} (d-2-7); 
      \end{tikzpicture}
    \end{center}
    where the bottom row is an $n$-angle.
    
    \newcommand\sym{Z}
    \newcommand\symc{z}
    
  \item The \Def{fiber-cofiber Toda bracket}
    $\todafc{f_n,\dots,f_2,f_1}\subseteq \C(\Sigma X_1, X_{n + 1})$ consists of
    all composites
    \begin{center}
      \begin{tikzpicture}
        \diagram{d}{1.5em}{1.5em}{
          \Sigma(\beta_1^{n - 1}\cdots \beta_1^2 \beta_1^1) \colon \Sigma X_1 &
          X_{n + 1}\\
        };

        \path[->]
        (d-1-1) edge (d-1-2);
      \end{tikzpicture}
    \end{center}
    where $\beta_1^1, \beta_1^2,\dots,\beta_1^{n - 1}$ appear in a commutative
    diagram of the shape of \autoref{figure:fcbracket}
    \begin{figure}[htbp]
      \centering
      \begin{tikzpicture}
        \smalldiagram{d}{3em}{2.5em}{
          X_1 & X_2 & & & & &\Sigma X_1\\
          \sym^{2}_1 & X_2 & X_3 & \sym^{2}_4 & \cdots & \sym^{2}_{n} & \Sigma
          \sym^{2}_{1}\\
          \sym^{3}_1 & \sym^{3}_2 & X_3 & X_4 & \cdots & \sym^{3}_{n} & \Sigma
          \sym^{3}_1\\
          \sym^{4}_1 & \sym^{4}_2 & \sym^{4}_3 & X_4 & \cdots & \sym^{4}_n & \Sigma
          \sym^{4}_1\\
          \vdots & \vdots & \vdots & \vdots & & \vdots & \vdots \\
          \sym^{{n-1}}_1 & \sym^{{n-1}}_2 &  \sym^{{n-1}}_3 &  \sym^{{n-1}}_4 &
          \cdots & X_n & \Sigma \sym^{{n-1}}_1\\
          & & & & & X_n & X_{n + 1}\\
        };
            
        \path[->, auto, font = \scriptsize]
        (d-1-1) edge node{$f_1$} (d-1-2)
                edge node{$\beta_1^1$} (d-2-1)
        \downequal{d}{1-2}{2-2}
        (d-1-7) edge node{$\Sigma\beta_1^1$} (d-2-7)
        
        (d-2-1) edge node{$\symc^2_1$} (d-2-2)
                edge node{$\beta_1^2$} (d-3-1)
        (d-2-2) edge node{$f_2$} (d-2-3)
                edge node{$\beta_2^2$} (d-3-2)
        (d-2-3) edge node{$\symc_3^2$} (d-2-4)
        \downequal{d}{2-3}{3-3}
        (d-2-4) edge node{$\symc^2_4$} (d-2-5)
                edge node{$\beta_4^2$} (d-3-4)
        (d-2-5) edge node{$\symc^2_{n-1}$} (d-2-6)
        (d-2-6) edge node{$\symc^2_{n}$} (d-2-7)
                edge node{$\beta_n^2$} (d-3-6)
        (d-2-7) edge node{$\Sigma\beta_1^2$} (d-3-7)
        
        (d-3-1) edge node{$\symc^3_1$} (d-3-2)
                edge node{$\beta_1^3$} (d-4-1)
        (d-3-2) edge node{$\symc_2^3$} (d-3-3)
                edge node{$\beta_2^3$} (d-4-2)
        (d-3-3) edge node{$f_3$} (d-3-4)
                edge node{$\beta_3^3$} (d-4-3)
        (d-3-4) edge node{$\symc^3_4$} (d-3-5)
        \downequal{d}{3-4}{4-4}
        (d-3-5) edge node{$\symc^3_{n-1}$} (d-3-6)
        (d-3-6) edge node{$\symc^3_{n}$} (d-3-7)
                edge node{$\beta_{n}^3$} (d-4-6)
        (d-3-7) edge node{$\Sigma\beta_1^3$} (d-4-7)
        
        (d-4-1) edge node{$\symc^4_1$} (d-4-2)
                edge node{$\beta_1^4$} (d-5-1)
        (d-4-2) edge node{$\symc^4_2$} (d-4-3)
                edge node{$\beta_2^4$} (d-5-2)
        (d-4-3) edge node{$\symc^4_3$} (d-4-4)
                edge node{$\beta_3^4$} (d-5-3)
        (d-4-4) edge node{$f_4$} (d-4-5)
                edge node{$\beta_4^4$} (d-5-4)
        (d-4-5) edge node{$\symc^4_{n-1}$} (d-4-6)
        (d-4-6) edge node{$\symc^4_{n}$} (d-4-7)
                edge node{$\beta_{n}^4$} (d-5-6)
        (d-4-7) edge node{$\Sigma\beta_1^4$} (d-5-7)
        
        (d-5-1) edge node{$\beta_{1}^{n-2}$} (d-6-1)
        (d-5-2) edge node{$\beta_{2}^{n-2}$} (d-6-2)
        (d-5-3) edge node{$\beta_{3}^{n-2}$} (d-6-3)
        (d-5-4) edge node{$\beta_{4}^{n-2}$} (d-6-4)
        (d-5-6) edge node{$\beta_{n}^{n-2}$} (d-6-6)
        (d-5-7) edge node{$\Sigma\beta_{1}^{n-2}$} (d-6-7)
        
        (d-6-1) edge node{$\symc^{n-1}_1$} (d-6-2)
        (d-6-2) edge node{$\symc^{n-1}_2$} (d-6-3)
        (d-6-3) edge node{$\symc^{n-1}_3$} (d-6-4)
        (d-6-4) edge node{$\symc^{n-1}_4$} (d-6-5)
        (d-6-5) edge node{$f_{n-1}$} (d-6-6)
        (d-6-6) edge node{$\symc^{n-1}_n$} (d-6-7)
        \downequal{d}{6-6}{7-6}
        (d-6-7) edge node{$\Sigma\beta_1^{n-1}$} (d-7-7)
        
        (d-7-6) edge node{$f_{n}$} (d-7-7);
      \end{tikzpicture}
      \caption{An element of the fiber-cofiber Toda bracket.}
      \label{figure:fcbracket}
    \end{figure}
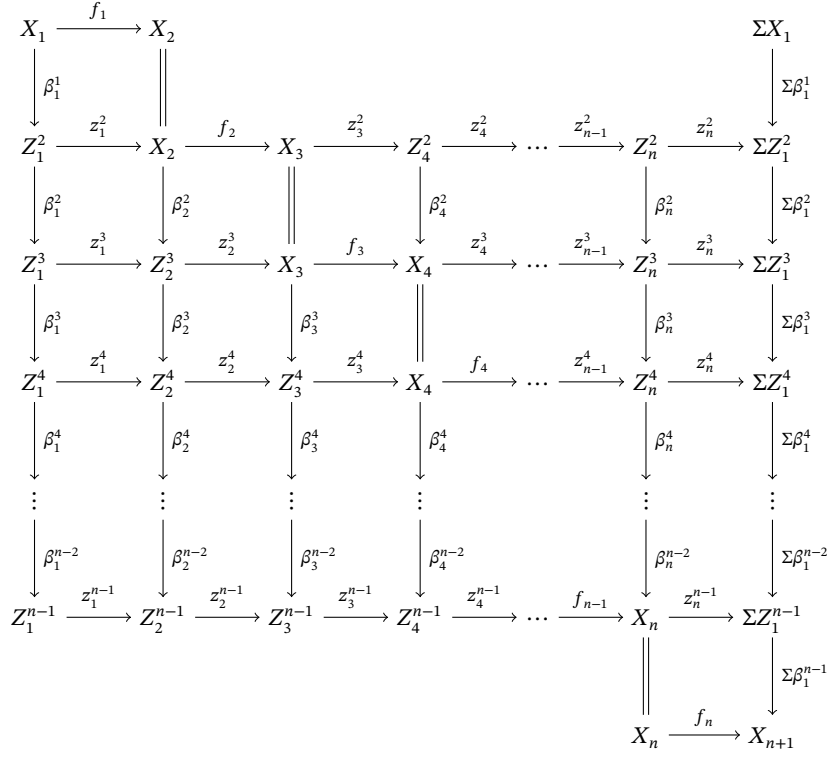
    where row $i$ is an $n$-angle $Z_\bullet^i$ in which $f_i$ occurs as the
    $i^\text{th}$ morphism.
  \end{itemize}
\end{definition}

Note that the vertical morphism in the upper right corner of
\autoref{figure:fcbracket} is there to clearly visualize the factorization of
$\psi$; it is simply the suspension of the morphism in the upper left corner.

\begin{remark}\label{rem:Toda4}
  The iterated fiber and cofiber Toda brackets have a familiar look to the
  triangulated case, however, the diagram defining the fiber-cofiber Toda
  bracket is rather overwhelming to digest. To address this, consider a
  $4$-angulated category, $\C$, and a diagram
  \begin{center}
    \begin{tikzpicture}    
      \diagram{d}{2em}{2em}{
        X_1 & X_2 & X_3 & X_4 & X_5.\\ 
      };
      
      \path[->, font = \scriptsize, auto]
      (d-1-1) edge node{$f_1$} (d-1-2)
      (d-1-2) edge node{$f_2$} (d-1-3)
      (d-1-3) edge node{$f_3$} (d-1-4)
      (d-1-4) edge node{$f_4$} (d-1-5);
    \end{tikzpicture}
  \end{center}
  An element of the fiber-cofiber Toda bracket
  $\psi\in\todafc{f_4,f_3,f_2,f_1}$ is then defined by a diagram
  
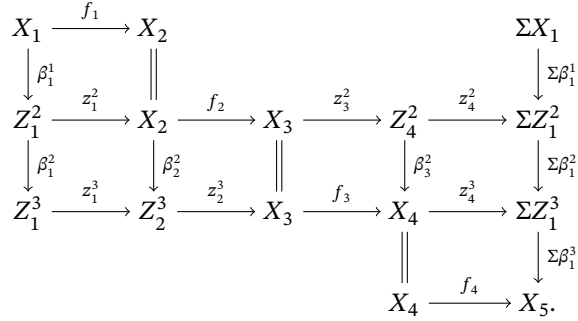
\begin{center}
    \begin{tikzpicture}    
      \diagram{d}{2em}{3em}{
        X_1 & X_2 & & & \Sigma X_1\\
        Z_1^{2} & X_2 & X_3 & Z_4^{2} & \Sigma Z_1^{2}\\
        Z_1^{3} & Z_2^{3} & X_3 & X_4 & \Sigma Z_1^{3}\\
        &&& X_4 & X_5. \\
      };
      
      \path[->, font = \scriptsize, auto]
      (d-1-1) edge node{$f_1$} (d-1-2)
      
      (d-2-1) edge node{$z_1^2$} (d-2-2)
      (d-2-2) edge node{$f_2$} (d-2-3)
      (d-2-3) edge node{$z_3^2$} (d-2-4)
      (d-2-4) edge node{$z_4^2$} (d-2-5)
      
      (d-3-1) edge node{$z_1^3$} (d-3-2)
      (d-3-2) edge node{$z_2^3$} (d-3-3)
      (d-3-3) edge node{$f_3$} (d-3-4)
      (d-3-4) edge node{$z_4^3$} (d-3-5)
      
      (d-4-4) edge node{$f_4$} (d-4-5)
      
      (d-1-1) edge node{$\beta^1_1$} (d-2-1)
      \downequal{d}{1-2}{2-2}
      (d-1-5) edge node{$\Sigma \beta^1_1$} (d-2-5)
      
      (d-2-1) edge node{$\beta^2_1$} (d-3-1)
      (d-2-2) edge node{$\beta^2_2$} (d-3-2)
      \downequal{d}{2-3}{3-3}   
      (d-2-4) edge node{$\beta^2_3$} (d-3-4)
      (d-2-5) edge node{$\Sigma\beta^2_1$} (d-3-5)
      
      \downequal{d}{3-4}{4-4}
      (d-3-5) edge node{$\Sigma \beta^3_1$} (d-4-5);
    \end{tikzpicture}
  \end{center}
  for some choices of $n$-angles containing $f_2$ and $f_3$ such that
  \begin{equation*}
    \psi = \Sigma(\beta_1^3 \beta_1^2 \beta^1_1).
  \end{equation*}
  The reader may now notice that this definition is similar to that of the
  fiber-cofiber Toda bracket in a triangulated category, in that we simply
  factor our element of the bracket through some $4$-angle extensions of the
  ``middle morphisms'' as opposed to the triangulated fiber-cofiber Toda
  bracket, in which we factor the elements of the bracket through the cofiber of
  the ``middle morphism.''

  Given this, one may wonder whether there are a definition of the Toda brackets
  that more closely resemble the fiber-cofiber Toda bracket in the triangulated
  case; i.e., would it be enough to specify only an extension of one of the
  morphisms in the given diagram, as is the case for the iterated fiber and
  cofiber Toda brackets? The answer turns out to be in the affirmative; we
  return to this in \autoref{sec:intermediate} where we introduce
  $[i]$-intermediate Toda brackets.
\end{remark}

\begin{lemma}\label{lem:InvariantIso}
  The three notions of Toda brackets are invariant under isomorphism. More
  precisely, consider a diagram in $\C$
  \begin{center}
    \begin{tikzpicture}
      \diagram{d}{2em}{2em}{
        X_1 & X_2 & X_3 & \cdots & X_n & X_{n+1}\\
        X'_1 & X'_2 & X'_3 & \cdots & X'_n & X'_{n+1}\\
      }; 
      
      \path[->, auto, font = \scriptsize]
      (d-1-1) edge node{$f_1$} (d-1-2)
              edge node{$g_1$} node[left]{$\cong$} (d-2-1) 
      (d-1-2) edge node{$f_2$} (d-1-3)
              edge node{$g_2$} node[left]{$\cong$} (d-2-2) 
      (d-1-3) edge node{$f_3$} (d-1-4)
              edge node{$g_3$} node[left]{$\cong$} (d-2-3) 
      (d-1-4) edge node{$f_{n-1}$} (d-1-5)
      (d-1-5) edge node{$f_n$} (d-1-6)
              edge node{$g_n$} node[left]{$\cong$} (d-2-5) 
      (d-1-6) edge node{$g_{n+1}$} node[left]{$\cong$} (d-2-6)
      
      (d-2-1) edge node{$f'_1$} (d-2-2)
      (d-2-2) edge node{$f'_2$} (d-2-3)
      (d-2-3) edge node{$f'_3$} (d-2-4)
      (d-2-4) edge node{$f'_{n-1}$} (d-2-5)
      (d-2-5) edge node{$f'_n$} (d-2-6);
    \end{tikzpicture}
  \end{center}
  where each $g_i \colon X_i \to X'_i$ is an isomorphism. Then the subset
  $\todacc{f_n, \ldots, f_2, f_1}$ corresponds to $\todacc{f'_n, \ldots, f'_2,
    f'_1}$ via the bijection $\C(\Sigma X_1, X_{n+1}) \cong \C(\Sigma X'_1,
  X'_{n+1})$ induced by $g_1$ and $g_{n+1}$. The same holds for the iterated
  fiber and fiber-cofiber Toda brackets.
\end{lemma}

%
% The Toda bracket
%
\section{The Toda bracket}
\label{sec:toda}

We will show in this section that, under a mild assumption, we may speak of
\emph{the} Toda bracket for a pre-$n$-angulated category, i.e., the three
definitions of Toda brackets introduced in the previous section coincide. It
turns out that explicitly describing an element of the iterated (co)fiber Toda
brackets as an element of the fiber-cofiber Toda bracket is rather difficult,
so, to get around this issue, we will instead show that all three definitions
define the same coset of a certain subgroup of morphisms.

Let us therefore recall the following basic algebraic fact:

\begin{lemma}\label{lemma:cosets}
  Let $G$ be a group, let $H$ be a subgroup of $G$, and let $S$ be a non-empty
  subset of $G$.
  \begin{enumerate}[label={(\arabic*)}, ref={Lemma \thelemma.(\arabic*)}]
  \item \label{lemma-4-1-1} If, for all $s\in S$ and $h\in H$, we have $sh\in S$
    then $S$ is a union of left cosets of $H$; i.e., there is a non-empty subset
    $I\subseteq G$ such that
    \begin{equation*}
      S=\bigcup_{g\in I}gH.
    \end{equation*}
  \item \label{lemma-4-1-2} If, for all $s,s'\in S$, we have $s^{-1}s'\in H$
    then $S\subseteq gH$ for some $g\in G$.
  \item \label{lemma-4-1-3} If $S$ satisfies both (1) and (2) there is some
    $g\in G$ such that
    \begin{equation*}
      S=gH.
    \end{equation*}
  \end{enumerate}
\end{lemma}

Next, let us note that the iterated (co)fiber Toda brackets retain the nice
property of being independent of the chosen $n$-angle extension.

\begin{lemma}\label{lem:AnyExtension}
  Let $f_1, f_2, \ldots, f_n$ be composable morphisms in $\C$ as in
  \autoref{eq:seq}.
  \begin{enumerate}
  \item The iterated cofiber Toda bracket $\todacc{f_n,\dots,f_2,f_1}$ can be
    computed with any chosen $n$-angle extension of $f_1$.
  \item The iterated fiber Toda bracket $\todaff{f_n,\dots,f_2,f_1}$ can be
    computed with any chosen $n$-angle extension of $f_n$.
  \end{enumerate}
\end{lemma}

\begin{proof}
  We prove the first statement; the second is dual. Let
  \begin{center}
    \begin{tikzpicture}    
      \diagram{d}{2em}{2em}{
        X_1 & X_2 & Y_3 & \cdots & Y_n & \Sigma X_n\\ 
      };
      
      \path[->, font = \scriptsize, auto]
      (d-1-1) edge node{$f_1$} (d-1-2)
      (d-1-2) edge node{$y_2$} (d-1-3)
      (d-1-3) edge node{$y_3$} (d-1-4)
      (d-1-4) edge node{$y_{n - 1}$} (d-1-5)
      (d-1-5) edge node{$y_n$} (d-1-6);
    \end{tikzpicture}
  \end{center}
  be an $n$-angle extension of $f_1$ and let $\tild{Y}_{\bullet}$ be another
  such extension. Furthermore, let $\psi \in
  \toda{f_n,\dots,f_2,f_1}_{\text{cc}, Y_{\bullet}}$ be a representative of the
  Toda bracket computed with the \mbox{$n$-angle} extension $Y_{\bullet}$,
  exhibited by a diagram which is the bottom part of
  \begin{center}
    \begin{tikzpicture}
      \diagram{d}{2em}{2em}{
        X_1 & X_2 & \tild{Y}_3 & \cdots & \tild{Y}_n & \Sigma X_1\\
        X_1 & X_2 & Y_3 & \cdots & Y_n & \Sigma X_1\\
        X_1 & X_2 & X_3 & \cdots & X_n & X_{n+1}.\\
      };
      
      \path[->, auto, font = \scriptsize]
      (d-1-1) edge node {$f_1$} (d-1-2)
      \downequal{d}{1-1}{2-1}
      (d-1-2) edge node {$\tild{y}_2$} (d-1-3)
      \downequal{d}{1-2}{2-2}
      (d-1-3) edge node {$\tild{y}_3$} (d-1-4)
              edge[densely dashed] node {$\theta_3$} (d-2-3)
      (d-1-4) edge node {$\tild{y}_{n-1}$} (d-1-5)
      (d-1-5) edge node {$\tild{y}_n$} (d-1-6)
              edge[densely dashed] node {$\theta_n$} (d-2-5)
      \downequal{d}{1-6}{2-6}
      
      (d-2-1) edge node {$f_1$} (d-2-2)
      \downequal{d}{2-1}{3-1}
      (d-2-2) edge node {$y_2$} (d-2-3)
      \downequal{d}{2-2}{3-2}
      (d-2-3) edge node {$y_3$} (d-2-4)
              edge[densely dashed] node {$\phi_3$} (d-3-3)
      (d-2-4) edge node {$y_{n-1}$} (d-2-5)
      (d-2-5) edge node {$y_n$} (d-2-6)
              edge[densely dashed] node {$\phi_n$} (d-3-5)
      (d-2-6) edge[densely dashed] node {$\psi$} (d-3-6)
      
      (d-3-1) edge node {$f_1$} (d-3-2)
      (d-3-2) edge node {$f_2$} (d-3-3)
      (d-3-3) edge node {$f_3$} (d-3-4)
      (d-3-4) edge node {$f_{n-1}$} (d-3-5)
      (d-3-5) edge node {$f_n$} (d-3-6);
    \end{tikzpicture}
  \end{center}
  By the morphism axiom \ref{N3}, there exists a morphism of $n$-angles $\theta
  \colon \tild Y_{\bullet} \to {Y}_{\bullet}$ as in the top part of the diagram.
  Now the composite diagram from the top row to the bottom row shows $\psi \in
  \toda{f_n,\dots,f_2,f_1}_{\text{cc}, \tild{Y}_{\bullet}}$. Since this holds
  for any two $n$-angle extensions of $f_1$, we obtain the equality
  $\toda{f_n,\dots,f_2,f_1}_{\text{cc}, Y_{\bullet}} =
  \todacc{f_n,\dots,f_2,f_1}$.
\end{proof}

\begin{remark}\label{rem:extension_dependence}
  The fact of the preceding lemma may be considered a rather nice ``edge case''
  phenomenon; in fact, the fiber-cofiber Toda bracket is dependent on the choice
  of $n$-angle extension. To see this, suppose that $n=5$ and that we have a
  diagram
  \begin{center}
    \begin{tikzpicture}    
      \diagram{d}{2em}{3em}{
        X_1 & X_2 & X_3 & X_4 & X_5 & X_6\\ 
      };
      
      \path[->, font = \scriptsize, auto]
      (d-1-1) edge node{$f_1$} (d-1-2)
      (d-1-2) edge node{$f_2$} (d-1-3)
      (d-1-3) edge node{$f_3=0$} (d-1-4)
      (d-1-4) edge node{$f_4$} (d-1-5)
      (d-1-5) edge node{$f_5$} (d-1-6);
    \end{tikzpicture}
  \end{center}
  such that $f_{i+1}f_i=0$ for all $1\leq i\leq 4$. It is a consequence of
  \autoref{prop:brackets_are_cosets} and \autoref{prop:fi_is_zero} that in this
  case, in fact,
  \begin{equation*}
    \todafc{f_5,f_4,0,f_2,f_1}=(f_5)_*\C(\Sigma X_1,X_5)+(\Sigma f_1)^*\C(\Sigma
    X_2,X_6)\subseteq \C(\Sigma X_1,X_6)
  \end{equation*}
  and, so, supposing this subgroup is non-trivial, the fiber-cofiber Toda
  bracket will contain an element different from the zero morphism; call it
  $\psi$. However, if we suppose that all elements of the fiber-cofiber Toda
  bracket may be calculated by any fixed choices of $5$-angle extensions of
  $f_2$, $f_3$, and $f_4$, then
  by the diagram
  \begin{center}
    \begin{tikzpicture}
      \smalldiagram{d}{2em}{3em}{
        X_1 & X_2 & & & & \Sigma X_1\\
        Z^{2}_{1} & X_2 & X_3 & Z^{2}_4 & Z^{2}_5 & \Sigma Z^2_1 \\
        0 & X_3 & X_3 & X_4 & X_4  & 0  \\
        Z^{4}_{1} & Z^{4}_{2} & Z^{4}_{3} & X_4 & X_5 & \Sigma Z_1^4 \\
        & & & & X_5 & X_{6}\\
      };
      
      \path[->, auto, font = \scriptsize]
      (d-1-1) edge node{$f_1$} (d-1-2)
              edge node{$\beta_{1}^1$} (d-2-1)
      \downequal{d}{1-2}{2-2}
      (d-1-6) edge node{$\Sigma \beta_{1}^1$} (d-2-6)

      (d-2-1) edge (d-2-2)
              edge (d-3-1)
      (d-2-2) edge node{$f_2$} (d-2-3)
              edge node{$f_2$} (d-3-2)
      (d-2-3) edge (d-2-4)
      \downequal{d}{2-3}{3-3}
      (d-2-4) edge (d-2-5)
              edge node{$\beta_4^2$} (d-3-4)
      (d-2-5) edge (d-2-6)
              edge node{$\beta_5^2$} (d-3-5)
      (d-2-6) edge (d-3-6)

      (d-3-1) edge (d-3-2)
              edge (d-4-1)
      (d-3-2) edge node{$1$} (d-3-3)
              edge node{$\beta_{2}^3$} (d-4-2)
      (d-3-3) edge node{$0$} (d-3-4)
              edge node{$\beta_{3}^3$} (d-4-3)
      (d-3-4) edge node{$1$} (d-3-5)
      \downequal{d}{3-4}{4-4}
      (d-3-5) edge (d-3-6)
              edge node{$f_4$} (d-4-5)
      (d-3-6) edge (d-4-6)

      (d-4-1) edge (d-4-2)
      (d-4-2) edge (d-4-3)
      (d-4-3) edge (d-4-4)
      (d-4-4) edge node{$f_4$} (d-4-5)
      (d-4-5) edge (d-4-6)
      \downequal{d}{4-5}{5-5}
      (d-4-6) edge node{$\Sigma\beta_1^4$} (d-5-6)
      
      (d-5-5) edge node{$f_5$} (d-5-6);
    \end{tikzpicture}
  \end{center}
  since the middle row is the valid 5-angle extension $(TX_3)_\bullet[-1]\oplus
  (TX_4)_\bullet[-3]$. However, this would imply that $\psi = 0$, a
  contradiction.
\end{remark}

Next we note that the fiber-cofiber Toda bracket is always contained in both the
iterated fiber and iterated cofiber Toda brackets.

\begin{lemma}\label{lem:fc_in_cc}
  Let $f_1, f_2, \ldots, f_n$ be composable morphisms in $\C$ as in
  \autoref{eq:seq}. Then
  \begin{equation*}
    \todafc{f_n,\dots,f_2,f_1}\subseteq \todacc{f_n,\dots,f_2,f_1}\cap
    \todaff{f_n,\dots,f_2,f_1}.
  \end{equation*}
\end{lemma}

\begin{proof}
  An element of $\psi\in \todafc{f_n,\dots,f_2,f_1}$ is given by a commutative
  diagram as in \autoref{figure:fcbracket} such that
  \begin{equation*}
    \psi = \Sigma(\beta_1^{n-1} \beta_1^{n-2} \cdots \beta_1^2 \beta_1^1).
  \end{equation*}
  Choose an $n$-angle extension $Z_\bullet^1$ of $f_1$ and some fill-in
  \begin{center}
    \begin{tikzpicture}
      \diagram{d}{2em}{2em}{
        X_1 & X_2 & Z^{1}_3 & Z^{1}_4 & \cdots & Z_{n}^{1} & \Sigma X_1\\
        Z^{2}_{1} & X_2 & X_3 & Z^{2}_4 & \cdots & Z^{2}_{n} & \Sigma Z^{2}_{1}.\\
      };
      
      \path[->, auto, font = \scriptsize]
      (d-1-1) edge node{$f_1$} (d-1-2)
      edge node{$\beta_1^1$} (d-2-1)
      \downequal{d}{1-2}{2-2}
      (d-1-2) edge (d-1-3)
      (d-1-3) edge (d-1-4)
              edge[densely dashed] node{$\beta^1_3$} (d-2-3)
      (d-1-4) edge (d-1-5)
              edge[densely dashed] node{$\beta^1_4$} (d-2-4)
      (d-1-5) edge (d-1-6)
      (d-1-6) edge (d-1-7)
              edge[densely dashed] node{$\beta^1_n$} (d-2-6)
      (d-1-7) edge node{$\Sigma \beta^1_{1}$} (d-2-7)
      
      (d-2-1) edge (d-2-2)
      (d-2-2) edge node{$f_2$} (d-2-3)
      (d-2-3) edge (d-2-4)
      (d-2-4) edge (d-2-5)
      (d-2-5) edge (d-2-6)
      (d-2-6) edge (d-2-7);
    \end{tikzpicture}
  \end{center}
  We define
  \begin{equation*}
    \phi_{i}\coloneq \beta^{i-2}_i \beta^{i-3}_i \cdots \beta^{2}_{i} \beta^1_i
  \end{equation*}
  to obtain a commutative diagram
  \begin{center}
    \begin{tikzpicture}
      \diagram{d}{2em}{2em}{
        X_1 & X_2 & Z^{1}_3 & Z^{1}_4 & \cdots & Z_{n}^{1} & \Sigma X_1\\
        X_1 & X_2 & X_3 & X_4 & \cdots & X_{n} & X_{n+1}.\\
      };
      
      \path[->, auto, font = \scriptsize]
      (d-1-1) edge node{$f_1$} (d-1-2)
      \downequal{d}{1-1}{2-1}
      (d-1-2) edge (d-1-3)
      \downequal{d}{1-2}{2-2}
      (d-1-3) edge (d-1-4)
              edge node{$\phi_3$} (d-2-3)
      (d-1-4) edge (d-1-5)
              edge node{$\phi_4$} (d-2-4)
      (d-1-5) edge (d-1-6)
      (d-1-6) edge (d-1-7)
              edge node{$\phi_{n}$} (d-2-6)
      (d-1-7) edge node{$\psi$} (d-2-7)
      
      (d-2-1) edge node{$f_1$} (d-2-2)
      (d-2-2) edge node{$f_2$} (d-2-3)
      (d-2-3) edge node{$f_3$} (d-2-4)
      (d-2-4) edge node{$f_4$} (d-2-5)
      (d-2-5) edge node{$f_{n-1}$} (d-2-6)
      (d-2-6) edge node{$f_{n}$} (d-2-7);
    \end{tikzpicture}
  \end{center}
  Thus $\psi\in\todacc{f_n,f_{n-1},\dots,f_2,f_1}$. The dual argument shows
  $\psi\in\todaff{f_n,f_{n-1},\dots,f_2,f_1}$.
\end{proof}

\begin{lemma}\label{lem:nonempty_fcbracket}
  Let $f_1, f_2, \ldots, f_n$ be composable morphisms in $\C$ as in
  \autoref{eq:seq}. Then the fiber-cofiber Toda bracket
  $\todafc{f_n,\dots,f_2,f_1}$ is non-empty if and only if $f_{i+1}f_i=0$ for
  $1\leq i\leq n-1$.
\end{lemma}

\begin{proof}
  Suppose $f_{i+1}f_i=0$ for $1\leq i\leq n-1$. We begin by setting up the
  ``staircase diagram'' that make up the spine of the diagram representing an
  element of the fiber-cofiber Toda bracket $\todafc{f_n,\dots,f_2,f_1}$.
  \begin{center}
    \begin{tikzpicture}
      \smalldiagram{d}{1.5em}{1.5em}{
        X_1 & X_2\\
        & X_2 & X_3 \\
        && X_3 & X_4\\
        &&& X_4 & X_5\\
        &&&& \vdots\\
        &&&&& \cdots & X_n\\
        &&&&&& X_n & X_{n + 1}.\\
      };
            
      \path[->, auto, font = \scriptsize]
      (d-1-1) edge node{$f_1$} (d-1-2)
      \downequal{d}{1-2}{2-2}
      (d-2-2) edge node{$f_2$} (d-2-3)
      \downequal{d}{2-3}{3-3}
      (d-3-3) edge node{$f_3$} (d-3-4)
      \downequal{d}{3-4}{4-4}
      (d-4-4) edge node{$f_4$} (d-4-5)
      \downequal{d}{4-5}{5-5}
      (d-6-6) edge node{$f_{n-1}$} (d-6-7)
      \downequal{d}{6-7}{7-7}
      (d-7-7) edge node{$f_{n}$} (d-7-8);
    \end{tikzpicture}
  \end{center}
  At each ``step'' of the staircase we extend the morphism to an $n$-angle.
  Since $f_{i+1}f_i=0$, we know that $f_{i+1}$ factors through the cofiber of
  $f_i$ and a morphism
  \begin{center}
    \begin{tikzpicture}
      \diagram{d}{1.5em}{1.5em}{
        \beta_{i+2}^i\colon Z_{i+2}^{i} & X_{i+2}.\\
      };
      
      \path[->]
      (d-1-1) edge (d-1-2);
    \end{tikzpicture}
  \end{center}
  Note that in the case of $i=n-1$, we label the induced morphism
  $\Sigma\beta^{n-1}_1$. Since we now have two consecutive morphisms forming a
  commutative square between two $n$-angles, we may now apply axioms \ref{N2}
  and \ref{N3} to claim that there exists a morphism of $n$-angles
  $\beta^i_\bullet\colon Z_\bullet^{i}\to Z_\bullet^{i+1}$. Filling in these
  morphisms of $n$-angles in the staircase diagram above, we achieve a diagram
  representing some element of $\todafc{f_n,\dots,f_2,f_1}$, namely the morphism
  $\psi\coloneq\Sigma(\beta_1^{n-1}\beta_1^{n-2}\cdots \beta_{1}^1)$, and thus we
  may conclude that this is a non-empty set.
  
  Now suppose $\todafc{f_n,\dots,f_2,f_1}\neq\emptyset$. Then we have a diagram
  in the shape of \autoref{figure:fcbracket}, and so we may simply read off the
  diagram that
  \begin{equation*}
    f_{i+1}f_i = f_{i+1}z_{i}^{i+1}\beta_{i}^i=0
  \end{equation*}
  for $1\leq i\leq n-2$. For the case of $i=n-1$ we see that
  \begin{equation*}
    f_nf_{n-1} = (\Sigma\beta_1^{n-1})z_n^{n-1}f_{n-1}=0.
  \end{equation*}
  This concludes the proof.
\end{proof}

\begin{corollary}\label{cor:ffcc_nonempty}
  Let $f_1, f_2, \ldots, f_n$ be composable morphisms in $\C$ as in
  \autoref{eq:seq}. If $f_{i+1}f_i=0$ for $1\leq i\leq n-1$ then the iterated
  (co)fiber Toda bracket are non-empty.
\end{corollary}

\begin{proof}
  By \autoref{lem:fc_in_cc} and \autoref{lem:nonempty_fcbracket}, 
  \begin{equation*}
    \emptyset\neq \todafc{f_n,\dots,f_2,f_1}\subseteq \todacc{f_n,\dots,f_2,f_1}\cap
    \todaff{f_n,\dots,f_2,f_1}.
  \end{equation*}
  Thus both $\todacc{f_n,\dots,f_2,f_1}$ and $\todaff{f_n,\dots,f_2,f_1}$ must be non-empty.
\end{proof}

The next lemma will help us prove that the Toda brackets are cosets of the same subgroup.

\begin{lemma}\label{lem:fc_is_union}
  Let $f_1, f_2, \ldots, f_n$ be composable morphisms in $\C$ as in
  \autoref{eq:seq} and let $\psi\in \todafc{f_n,\dots,f_2,f_1}$. For arbitrary
  morphisms $h\colon \Sigma X_1 \to X_n$ and $k\colon \Sigma X_2 \to X_{n + 1}$,
  we have
  \begin{equation*}
    \psi + f_n h + k\Sigma f_1 \in \todafc{f_n,\dots,f_2,f_1}.
  \end{equation*}
\end{lemma}

\begin{proof}
  We only prove $\psi + f_n h\in \todafc{f_n,\dots,f_2,f_1}$; the proof for
  $\psi + k\Sigma f_1$ is dual.
    
  In the case $n=3$, the statement is a standard fact, whose proof only uses the
  axioms of a pretriangulated category; see for instance
  \cite[Lemma~4.6.2]{BarnesR20}. Now assume $n \geq 4$. Consider the diagram
  representing $\psi$ given by \autoref{figure:fcbracket} where
  \begin{equation*}
    \psi = \Sigma(\beta_1^{n-1}\beta_1^{n-2}\cdots\beta_{1}^2\beta_{1}^1).
  \end{equation*}
  From this we can now construct a new diagram, using the given morphism
  $h\colon \Sigma X_1\to X_n$: to each intermediate $n$-angle in the diagram,
  except the very last one, we add on the trivial $n$-angle $(TX_1)_\bullet[1]$
  and exchange most of the vertical morphisms with their obvious counterpart in
  \autoref{fig:trivial}.

  \begin{figure}[htbp]
    \centering
    \begin{tikzpicture}
      \begin{scope}[xshift=-1cm]
        \diagram{d}{3em}{2em}{
          X_1 &[2em] X_2\\
          Z_{1}^2 \oplus X_1 &[2em] X_2 & X_3 & \cdots\\
          && X_3 & \cdots\\
        };
        
        \path[->, font = \scriptsize, auto]
        (d-1-1) edge node{$f_1$} (d-1-2)
        edge[blue] node[blue]{$\left[
            \begin{smallmatrix}
              \beta_{1}^1\\
              1
            \end{smallmatrix}
          \right]$} (d-2-1)
        \downequal{d}{1-2}{2-2}
        
        (d-2-1) edge node{$\left[
            \begin{smallmatrix}
              z_1^2 & 0
            \end{smallmatrix}\right]$} (d-2-2)
        (d-2-2) edge node{$f_2$} (d-2-3)
        (d-2-3) edge node{$z_3^2$} (d-2-4)
        \downequal{d}{2-3}{3-3}
        
        (d-3-3) edge node{$f_3$} (d-3-4);
      \end{scope}
      
      \begin{scope}[xshift=2cm,yshift=-6.75cm]
        \diagram{d}{3em}{2em}{
          &&&[2em] &[2em] \Sigma X_1\\
          \cdots & Z_{n - 2}^2 & Z_{n-1}^2 &[2em] Z_{n}^2 \oplus \Sigma X_1 &[2em]
          \Sigma(Z_{1}^2 \oplus X_1)\\
          \cdots & Z_{n - 2}^3 & Z_{n - 1}^3 &[2em] Z_{n}^3 \oplus \Sigma X_1 &[2em]
          \Sigma(Z_{1}^3 \oplus  X_1)\\
          & \vdots & \vdots &[2em] \vdots &[2em] \vdots\\
          & X_{n - 2} & X_{n - 1} &[2em] Z_n^{n-2} \oplus \Sigma X_1 &[2em]
          \Sigma(Z_1^{n - 2} \oplus  X_1)\\
          && X_{n - 1} &[2em] X_n &[2em] Z_1^{n - 1}\\
          &&&[2em] X_n &[2em] X_{n + 1}\\ 
        };
        
        \path[->, font = \scriptsize, auto]
        (d-1-5) edge[blue] node[blue]{$\left[
            \begin{smallmatrix}
              \Sigma\beta_{1}^1\\
              1
            \end{smallmatrix}\right]$} (d-2-5)
        
        (d-2-1) edge node{$z_{n - 3}^2$} (d-2-2)
        (d-2-2) edge node{$z_{n - 2}^2$} (d-2-3)
        edge node{$\beta_{n - 2}^2$} (d-3-2)
        (d-2-3) edge node{$\left[
            \begin{smallmatrix}
              z_{n - 1}^2\\
              0
            \end{smallmatrix}\right]$} (d-2-4)
        edge node{$\beta_{n - 1}^2$} (d-3-3)
        (d-2-4) edge node{$\left[
            \begin{smallmatrix}
              z_{n}^2 & 0\\
              0 & 1
            \end{smallmatrix}\right]$} (d-2-5)
        edge node{$\left[
            \begin{smallmatrix}
              \beta_{n}^2 & 0\\
              0 & 1
            \end{smallmatrix}\right]$} (d-3-4)
        (d-2-5) edge node{$\left[
            \begin{smallmatrix}
              \Sigma \beta_{1}^2 & 0\\
              0 & 1
            \end{smallmatrix}\right]$} (d-3-5)
        
        (d-3-1) edge node{$z_{n - 3}^3$} (d-3-2)
        (d-3-2) edge node{$z_{n - 2}^3$} (d-3-3)
        edge node{$\beta_{n - 2}^3$} (d-4-2)
        (d-3-3) edge node{$\left[
            \begin{smallmatrix}
              z_{n - 1}^3\\
              0
            \end{smallmatrix}\right]$} (d-3-4)
        edge node{$\beta_{n - 1}^3$} (d-4-3)
        (d-3-4) edge node{$\left[
            \begin{smallmatrix}
              z_{n}^3 & 0\\
              0 & 1
            \end{smallmatrix}\right]$} (d-3-5)
        edge node{$\left[
            \begin{smallmatrix}
              \beta_{n}^3 & 0\\
              0 & 1
            \end{smallmatrix}\right]$} (d-4-4)
        (d-3-5) edge node{$\left[
            \begin{smallmatrix}
              \Sigma\beta_{1}^3 & 0\\
              0 & 1
            \end{smallmatrix}\right]$} (d-4-5)
        
        \downequal{d}{4-2}{5-2}
        (d-4-3) edge node{$\beta_{n-1}^{n - 3}$} (d-5-3)
        (d-4-4) edge node{$\left[
            \begin{smallmatrix}
              \beta_n^{n - 3} & 0\\
              0 & 1
            \end{smallmatrix}\right]$} (d-5-4)
        (d-4-5) edge node{$\left[
            \begin{smallmatrix}
              \Sigma\beta_1^{n - 3} & 0\\
              0 & 1
            \end{smallmatrix}\right]$} (d-5-5)
        
        (d-5-2) edge node{$f_{n - 2}$} (d-5-3)
        (d-5-3) edge node{$\left[
            \begin{smallmatrix}
              z_{n-1}^{n - 2}\\
              0
            \end{smallmatrix}\right]$} (d-5-4)
        \downequal{d}{5-3}{6-3}
        (d-5-4) edge node{$\left[
            \begin{smallmatrix}
              z_n^{n - 2} & 0\\
              0 & 1
            \end{smallmatrix}\right]$} (d-5-5)
        edge[blue] node[blue]{$\left[
            \begin{smallmatrix}
              \beta_n^{n - 2} & h
            \end{smallmatrix}\right]$} (d-6-4)
        (d-5-5) edge[blue] node[blue]{$\left[
            \begin{smallmatrix}
              \Sigma\beta_1^{n - 2} & z_n^{n - 1} h
            \end{smallmatrix}\right]$} (d-6-5)
        
        (d-6-3) edge node{$f_{n - 1}$} (d-6-4)
        (d-6-4) edge node{$z_n^{n - 1}$} (d-6-5)
        \downequal{d}{6-4}{7-4}
        (d-6-5) edge node{$\Sigma\beta_1^{n - 1}$} (d-7-5)
        
        (d-7-4) edge node{$f_n$} (d-7-5);
      \end{scope}
    \end{tikzpicture}
    \caption{Adding trivial summands.}
    \label{fig:trivial}
  \end{figure}
    
  In service of legibility we have left out the left lower corner of the
  staircase, since nothing of further note has been changed there. Emphasized in
  blue are the only ``non-trivial changes'' to the diagram. Commutativity of
  this diagram is easily verified and, so, extending it to the left, we see that
  we have produced a diagram that represents a new element of the fiber-cofiber
  Toda bracket $\todafc{f_n,\dots,f_2,f_1}$; let us denote it by $\overline
  \psi$. Let us calculate said morphism:
  \begin{align*}
    \overline\psi&= (\Sigma \beta_1^{n-1})
                   \left[\begin{matrix}
                       \Sigma\beta_1^{n-2} & z_n^{n-1}h
                     \end{matrix}\right]
                   \left[\begin{matrix}
                       \Sigma\beta_1^{n-3}&0\\
                       0&1
                     \end{matrix}\right]\cdots
                   \left[\begin{matrix}
                       \Sigma\beta_{1}^2 & 0
                       \\0&1
                     \end{matrix}\right]
                   \left[\begin{matrix}
                       \Sigma\beta_{1}^1\\
                       1\end{matrix}\right] \\
                 &=\left[\begin{matrix}
                     \Sigma(\beta_1^{n-1}\beta_1^{n-2}) & (\Sigma\beta_1^{n-1})z^{n-1}_{n}h
                   \end{matrix}\right]
                   \left[\begin{matrix}
                       \Sigma(\beta_1^{n-3}\cdots \beta_{1}^2)&0\\
                       0&1
                     \end{matrix}\right]
                   \left[\begin{matrix}
                       \Sigma \beta^{1}_1
                       \\1
                     \end{matrix}\right]\\
                 &=\Sigma(\beta_1^{n-1}\beta_1^{n-2}\beta_1^{n-3}\cdots\beta_{1}^2\beta_{1}^1)
                   +(\Sigma\beta_1^{n-1})z^{n-1}_nh\\
                 &=\psi+f_nh. \qedhere
  \end{align*}
\end{proof}

\begin{proposition}\label{prop:brackets_are_cosets}
  Let $f_1, f_2, \ldots, f_n$ be composable morphisms in $\C$ as in
  \autoref{eq:seq}. such that $f_{i+1}f_i=0$ for $1\leq i\leq n-1$. Let $G$ be
  the subgroup
  \begin{equation*}
    (f_n)_\ast \C(\Sigma X_1, X_n) + (\Sigma f_1)^\ast \C(\Sigma X_2, X_{n + 1})\subseteq
    \C(\Sigma X_1, X_{n + 1}).
  \end{equation*}
  Then the sets $\todacc{f_n,\dots,f_2,f_1}$, $\todaff{f_n,\dots,f_2,f_1}$ and
  $\todafc{f_n,\dots,f_2,f_1}$ are all cosets of $G$.
\end{proposition}

\begin{proof}
  \textbf{Iterated (co)fiber.} Since the iterated fiber and iterated cofiber
  Toda brackets are dual, we prove the statement for the latter. First, note
  that by \autoref{cor:ffcc_nonempty}, the iterated cofiber Toda bracket
  $\todacc{f_n,\dots,f_2,f_1}$ is non-empty. Let $\psi$ and $\tilde{\psi}$ be
  two elements of $\todacc{f_n,\dots,f_2,f_1}$. By \autoref{lem:AnyExtension} we
  may suppose $\psi$ and $\tild\psi$ are computed with the same $n$-angle
  extension of $f_1$, so let us denote this extension by $Y_\bullet$ exhibited
  as the top row in the commutative diagram
  \begin{center}
    \begin{tikzpicture}
      \diagram{d}{2em}{2em}{
        X_1 & X_2 & Y_3 & Y_4 & \cdots & Y_n & \Sigma X_{1}\\
        X_1 & X_2 & X_3 & X_4 & \cdots & X_n & X_{n+1}.\\
      };
      
      \path[->, auto, font = \scriptsize]
      (d-1-1) edge node{$f_1$} (d-1-2)
      (d-1-2) edge node{$y_2$} (d-1-3)
      (d-1-3) edge node{$y_3$} (d-1-4)
      (d-1-4) edge node{$y_4$} (d-1-5)
      (d-1-5) edge node{$y_{n-1}$} (d-1-6)
      (d-1-6) edge node{$y_n$} (d-1-7)
      
      (d-2-1) edge node{$f_1$} (d-2-2)
      (d-2-2) edge node{$f_2$} (d-2-3)
      (d-2-3) edge node{$f_3$} (d-2-4)
      (d-2-4) edge node{$f_4$} (d-2-5)
      (d-2-5) edge node{$f_{n-1}$} (d-2-6)
      (d-2-6) edge node{$f_n$} (d-2-7)
      
      \downequal{d}{1-1}{2-1}
      \downequal{d}{1-2}{2-2}
      ([xshift=-.2em] d-1-3.south) edge node[swap]{$\varphi_3$} ([xshift=-.2em] d-2-3.north)
      ([xshift=.2em] d-1-3.south) edge node{$\tild\varphi_3$} ([xshift=.2em] d-2-3.north)
      ([xshift=-.2em] d-1-4.south) edge node[swap]{$\varphi_4$} ([xshift=-.2em] d-2-4.north)
      ([xshift=.2em] d-1-4.south) edge node{$\tild\varphi_4$} ([xshift=.2em] d-2-4.north)
      ([xshift=-.2em] d-1-6.south) edge node[swap]{$\varphi_{n}$} ([xshift=-.2em] d-2-6.north)
      ([xshift=.2em] d-1-6.south) edge node{$\tild\varphi_{n}$} ([xshift=.2em] d-2-6.north)
      ([xshift=-.2em] d-1-7.south) edge node[swap]{$\psi$} ([xshift=-.2em] d-2-7.north)
      ([xshift=.2em] d-1-7.south) edge node{$\tild\psi$} ([xshift=.2em] d-2-7.north);
    \end{tikzpicture}
  \end{center}
  Then we see that
  \begin{equation*}
    (\phi_3-\tild\phi_3)y_2=0
  \end{equation*}
  so there is a morphism $h_3\colon Y_4\to X_3$ such that
  $h_3y_3=\varphi_3-\tild\varphi_3$. Therefore
  \begin{equation*}
    (\varphi_4-\tild\varphi_4)y_3=f_3(\varphi_3-\tild\varphi_3)=f_3h_3y_3
  \end{equation*}
  and so
  \begin{equation*}
    (\varphi_4-\tild\varphi_4-f_3h_3)y_3=0
  \end{equation*}
  which means that there is a morphism $h_4\colon Y_5\to X_4$ such that
  $h_4y_4=\varphi_4-\tild\varphi_4-f_3h_3$. Continuing this process, we see that
  \begin{equation*}
    (\varphi_5-\tild\varphi_5)y_4=f_4(\varphi_4-\tild\varphi_4)=f_4(\varphi_4 -
    \tild\varphi_4-f_3h_3)=f_4h_4y_4
  \end{equation*}
  since $f_4f_3=0$, so, again,
  \begin{equation*}
    (\varphi_5-\tild\varphi_5-f_4h_4)y_4=0.
  \end{equation*}
  Suppose now that for some $k\leq n$ we have found morphisms $h_i\colon
  Y_{i+1}\to X_i$ for all $4\leq i\leq k$ such that
  \begin{equation*}
    h_iy_i=\varphi_{i}-\tild\varphi_{i}-f_{i-1}h_{i-1}.
  \end{equation*}
  Then
  \begin{equation*}
    (\varphi_{k+1}-\tild\varphi_{k+1})y_k=f_k(\varphi_{k}-\tild\varphi_{k})=f_k(\varphi_{k}
    -\tild\varphi_{k}-f_{k-1}h_{k-1})=f_kh_ky_k
  \end{equation*}
  so
  \begin{equation*}
    (\varphi_{k+1}-\tild\varphi_{k+1}-f_kh_k)y_k=0
  \end{equation*}
  and therefore there is a morphism $h_{k+1}\colon Y_{k+2}\to X_{k+1}$ such that
  \begin{equation*}
    h_{k+1}y_{k+1}=\varphi_{k+1}-\tild\varphi_{k+1}-f_kh_k.
  \end{equation*}
  In the case $k=n$, this argument yields a morphism $h_{n+1}\colon \Sigma
  X_2\to Y_{n+1}$ satisfying the equation
  \begin{equation*}
    h_{n+1}\Sigma f_1=\psi-\tild\psi-f_nh_n
  \end{equation*}
  where $h_n\in\C(\Sigma X_1,X_n)$ and so
  \begin{equation*}
    \psi-\tild\psi=f_nh_n+h_{n+1}\Sigma f_1\in (f_n)_*\C(\Sigma X_1,X_{n})+(\Sigma
    f_1)^*\C(\Sigma X_2,X_{n+1}).
  \end{equation*}
  This proves that the set $\todacc{f_n,\dots,f_2,f_1}$ satisfies the condition
  of \ref{lemma-4-1-2}. To prove that it also satisfies the condition of
  \ref{lemma-4-1-1} we now take the $\psi$ from earlier and consider two
  arbitrary morphisms
  \begin{equation*}
    h_n\colon \Sigma X_1\to X_n \quad\text{and}\quad h_{n+1}\colon \Sigma X_2\to X_{n+1}.
  \end{equation*}
  We claim that the diagram
  \begin{center}
    \begin{tikzpicture}
      \diagram{d}{2.5em}{2em}{
        X_1 & X_2 & Y_3 & Y_4 & \cdots &Y_{n-1}& Y_n & \Sigma X_{1}\\
        X_1 & X_2 & X_3 & X_4 & \cdots &X_{n-1}& X_n & X_{n+1}.\\
      };
      
      \path[->, auto, font = \scriptsize]
      (d-1-1) edge node{$f_1$} (d-1-2)
      (d-1-2) edge node{$y_2$} (d-1-3)
      (d-1-3) edge node{$y_3$} (d-1-4)
      (d-1-4) edge node{$y_4$} (d-1-5)
      (d-1-5) edge node{$y_{n-2}$} (d-1-6)
      (d-1-6) edge node{$y_{n-1}$} (d-1-7)
      (d-1-7) edge node{$y_n$} (d-1-8)
      
      (d-2-1) edge node{$f_1$} (d-2-2)
      (d-2-2) edge node{$f_2$} (d-2-3)
      (d-2-3) edge node{$f_3$} (d-2-4)
      (d-2-4) edge node{$f_4$} (d-2-5)
      (d-2-5) edge node{$f_{n-2}$} (d-2-6)
      (d-2-6) edge node{$f_{n-1}$} (d-2-7)
      (d-2-7) edge node{$f_n$} (d-2-8)
      
      \downequal{d}{1-1}{2-1}
      \downequal{d}{1-2}{2-2}
      (d-1-3) edge node{$\varphi_3$} (d-2-3)
      (d-1-4) edge node{$\varphi_4$} (d-2-4)
      (d-1-6) edge node{$\varphi_{n-1}$} (d-2-6)
      (d-1-7) edge node[xshift=-20pt,inner sep=1pt,fill=white]{$\varphi_{n}+h_ny_n$}
      (d-2-7)
      (d-1-8) edge node[xshift=-20pt,inner sep=1pt,fill=white]{$\psi+f_nh_n+h_{n+1}
        \Sigma f_1$} (d-2-8);
    \end{tikzpicture}
  \end{center}
  commutes and therefore $\psi+f_nh_n+h_{n+1}\Sigma
  f_1\in\todacc{f_n,\dots,f_2,f_1}$. To see that the diagram commutes, simply
  note
  \begin{equation*}
    (\varphi_{n}+h_ny_n)y_{n-1}=\varphi_{n}y_{n-1}=f_{n-1}\varphi_{n-1}
  \end{equation*}
  and
  \begin{equation*}
    (\psi+f_nh_n+h_{n+1}\Sigma f_1)y_n=(\psi+f_nh_n)y_n=f_n(\varphi_{n}+h_ny_n).
  \end{equation*}
  By \ref{lemma-4-1-3}, this proves the statement for the iterated cofiber Toda
  bracket.

  \textbf{Fiber-cofiber.} Note first that the fiber-cofiber Toda bracket
  $\todafc{f_n,\dots,f_2,f_1}$ is non-empty by \autoref{lem:nonempty_fcbracket}.
  By \autoref{lem:fc_is_union} we know that for any
  $\psi\in\todafc{f_n,\dots,f_2,f_1}$ and any two morphisms, $h\colon \Sigma
  X_1\to X_n$ and $k\colon \Sigma X_2\to X_{n+1}$ we have
  \begin{equation*}
    \psi + f_nh + k\Sigma f_1\in\todafc{f_n,\dots,f_2,f_1}.
  \end{equation*}
  Thus we conclude that $\todafc{f_n,\dots,f_2,f_1}$ is a union of cosets of $G$
  by \ref{lemma-4-1-1}. However, as we know by \autoref{lem:fc_in_cc} that
  \begin{equation*}
    \todafc{f_n,\dots,f_2,f_1}\subseteq\todacc{f_n,\dots,f_2,f_1}
  \end{equation*}
  and that $\todacc{f_n,\dots,f_2,f_1}$ is a coset of $G$, we conclude that
  $\todafc{f_n,\dots,f_2,f_1}$ is the same coset.
\end{proof}

This motivates recycling the terminology of triangulated categories; namely, we
will call the subgroup
\begin{equation*}
  (f_n)_*\C(\Sigma X_1,X_{n})+(\Sigma f_1)^*\C(\Sigma X_2,X_{n+1})
\end{equation*}
the \Def{indeterminacy}. The main theorem now follows as a simple corollary.

\begin{theorem}\label{thm:main}
  Let $\C$ be a pre-$n$-angulated category and consider a diagram of the shape
  \begin{center}
    \begin{tikzpicture}
      \diagram{d}{2em}{2em}{
        X_1 & X_2 & X_3 & \cdots & X_n & X_{n+1}\\
      };
      
      \path[->, auto, font = \scriptsize]
      (d-1-1) edge node{$f_1$} (d-1-2)
      (d-1-2) edge node{$f_2$} (d-1-3)
      (d-1-3) edge node{$f_3$} (d-1-4)
      (d-1-4) edge node{$f_{n-1}$} (d-1-5)
      (d-1-5) edge node{$f_n$} (d-1-6);
    \end{tikzpicture}
  \end{center}
  such that $f_{i+1}f_i=0$ for $1\leq i\leq n-1$. Then
  \begin{equation*}
    \todaff{f_n,\dots,f_2,f_1}=\todafc{f_n,\dots,f_2,f_1}=\todacc{f_n,\dots,f_2,f_1}.
  \end{equation*}
\end{theorem}

\begin{proof}
  By \autoref{prop:brackets_are_cosets}, we know that all three Toda brackets
  are cosets. Since the fiber-cofiber Toda bracket $\todafc{f_n,\dots,f_2,f_1}$
  is contained in the remaining two, we may conclude that all three are equal.
\end{proof}

\begin{remark}
  In the case $n=3$, the proofs provided in \cite[Section~4.6]{Meier12} or
  \cite[Proposition~3.3]{CF} use the octahedral axiom, whereas the proof
  presented in this section holds in any pretriangulated category. For $n=4$, we
  know a direct proof of the inclusion $\todacc{f_4,f_3,f_2,f_1} \subseteq
  \todafc{f_4,f_3,f_2,f_1}$ relying on the higher octahedral axiom~(N4)
  mentioned in \autoref{rem:nang} in the spirit of \cite[Proposition~3.3]{CF}.
  However, using the indeterminacy yields a simpler proof.
\end{remark}

We may now talk about \emph{the} Toda bracket, denoted by
$\toda{f_n,\dots,f_2,f_1}$, whenever we are presented a diagram of $n$
consecutive morphisms $f_i$ where the $2$-fold composites are zero, that is,
$f_{i+1} f_i = 0$ for $1 \leq i \leq n-1$. We think of the fiber-cofiber Toda
bracket as \emph{the} Toda bracket, morally speaking. However, we may still use
the notation of the fiber-cofiber, iterated fiber or iterated cofiber Toda
brackets whenever we wish to specify which kind of diagram we want to represent
an element of the bracket by.

\begin{example}\label{eq:ff_cc_not_equal}
  It turns out that we really do require that the morphisms compose to zero for
  \autoref{thm:main} to hold. In the triangulated case, this is not so: if two
  of the three morphisms in this scenario do not compose to zero, all three Toda
  brackets will be empty. In the $n$-angulated case for $n> 3$ if two morphisms
  do not compose to zero, the iterated (co)fiber Toda bracket may be non-empty,
  while the fiber-cofiber Toda bracket is empty.

  To illustrate this point, consider the following simple diagram in a
  $4$-angulated category:
  \begin{center}
    \begin{tikzpicture}
      \diagram{d}{2em}{2em}{
        X & X & \Sigma X & \Sigma X & \Sigma X\\
      };
      
      \path[->, auto, font = \scriptsize]
      (d-1-1) edge node{$1$} (d-1-2)
      (d-1-2) edge node{$0$} (d-1-3)
      (d-1-3) edge node{$1$} (d-1-4)
      (d-1-4) edge node{$1$} (d-1-5);
    \end{tikzpicture}
  \end{center}
  for some object $X$. Then, from the diagram
  \begin{center}
    \begin{tikzpicture}
      \diagram{d}{2em}{2em}{
        X & X & 0 & 0 & \Sigma X\\
        X & X & \Sigma X & \Sigma X & \Sigma X,\\
      };
      
      \path[->, auto, font = \scriptsize]
      (d-1-1) edge node{$1$} (d-1-2)
      (d-1-2) edge (d-1-3)
      (d-1-3) edge (d-1-4)
      (d-1-4) edge (d-1-5)
      
      (d-2-1) edge node{$1$} (d-2-2)
      (d-2-2) edge node{$0$} (d-2-3)
      (d-2-3) edge node{$1$} (d-2-4)
      (d-2-4) edge node{$1$} (d-2-5)
      
      \downequal{d}{1-1}{2-1}
      \downequal{d}{1-2}{2-2}
      (d-1-3) edge (d-2-3)
      (d-1-4) edge (d-2-4)
      (d-1-5) edge node{$\psi$} (d-2-5);
    \end{tikzpicture}
  \end{center}
  we may immediately discern that
  \begin{equation*}
    \todacc{1,1,0,1}=\C(\Sigma X,\Sigma X)\neq\emptyset
  \end{equation*}
  since there are no restrictions imposed on $\psi$. Note also that the iterated
  cofiber and iterated fiber Toda brackets do not agree, since
  \begin{equation*}
    \todaff{1,1,0,1}=\emptyset.
  \end{equation*}
  This example easily generalizes to all $n$-angulated categories for $n\geq 4$;
  in general we see that
  \begin{equation*}
    \todacc{1,1,1,\dots,1,0,1}=\C(\Sigma X,\Sigma X)\neq\emptyset.
  \end{equation*}
  For $n\geq 5$ it is also possible to construct examples where both the
  iterated fiber and iterated cofiber Toda brackets are both non-empty
  simultaneously.
\end{example}

%
% Intermediate Toda brackets
%
\section{Intermediate Toda brackets}
\label{sec:intermediate}

Let $\C$ be a pre-$n$-angulated category and let
\begin{center}
  \begin{tikzpicture}
    \diagram{d}{2em}{2em}{
      X_1 & X_2 & X_3 & \cdots & X_n & X_{n + 1}\\
    };
    
    \path[->, font = \scriptsize, auto]
    (d-1-1) edge node{$f_1$} (d-1-2)
    (d-1-2) edge node{$f_2$} (d-1-3)
    (d-1-3) edge node{$f_3$} (d-1-4)
    (d-1-4) edge node{$f_{n - 1}$} (d-1-5)
    (d-1-5) edge node{$f_n$} (d-1-6);
  \end{tikzpicture}
\end{center}
be a diagram in $\C$. We define new subsets of $\C(\Sigma X_1, X_{n + 1})$ as
follows:

For each $1\leq i\leq n$ we have the \Def{$[i]$-intermediate Toda bracket}
$\todai{f_n,\dots,f_2,f_1}{i}\subseteq\C(\Sigma X_1,X_{n+1})$ consisting of
morphisms $\psi=\Sigma(\beta_{1}\alpha_1)\colon \Sigma X_1\to X_{n+1}$ such that
$\alpha_1$ and $\beta_{1}$ occur in a commutative diagram
\begin{center}
  \begin{tikzpicture}
    % Part 1
    \begin{scope}
      \smalldiagram{d}{2em}{2.5em}{
        X_1 & X_2 & \cdots\\
        Z_{1} & Z_{2} & \cdots\\
      };
      
      \path[->, font = \scriptsize, auto]
      (d-1-1) edge node{$f_1$} (d-1-2)
              edge node{$\alpha_1$} (d-2-1)
      (d-1-2) edge node{$f_2$} (d-1-3)
              edge node{$\alpha_2$} (d-2-2)
      
      (d-2-1) edge node{$z_{1}$} (d-2-2)
      (d-2-2) edge node{$z_{2}$} (d-2-3);
    \end{scope}
    
    % Part 2
    \begin{scope}[xshift = 3.5cm, yshift = -3cm]
      \smalldiagram{d}{2em}{2.5em}{
        \cdots & X_{i - 1} & X_i\\
        \cdots & Z_{i-1} & X_i & X_{i + 1} & Z_{i+2} & \cdots\\
        &&& X_{i + 1} & X_{i + 2} & \cdots\\
      };
      
      \path[->, font = \scriptsize, auto]
      (d-1-1) edge node{$f_{i - 2}$} (d-1-2)
      (d-1-2) edge node{$f_{i - 1}$} (d-1-3)
              edge node{$\alpha_{i - 1}$} (d-2-2)
      \downequal{d}{1-3}{2-3}
      
      (d-2-1) edge node{$z_{i - 2}$} (d-2-2)
      (d-2-2) edge node{$z_{i - 1}$} (d-2-3)
      (d-2-3) edge node{$f_i$} (d-2-4)
      (d-2-4) edge node{$z_{i+1}$} (d-2-5)
      \downequal{d}{2-4}{3-4}
      (d-2-5) edge node{$z_{i+2}$} (d-2-6)
              edge node{$\beta_{i+2}$} (d-3-5)
      
      (d-3-4) edge node{$f_{i + 1}$} (d-3-5)
      (d-3-5) edge node{$f_{i + 2}$} (d-3-6);
    \end{scope}
    
    % Part 3
    \begin{scope}[xshift = 7cm, yshift = -6cm]
      \smalldiagram{d}{2em}{2.5em}{
        \cdots & Z_{n} & \Sigma Z_{1}\\
        \cdots & X_n & X_{n + 1}\\
      };
      
      \path[->, font = \scriptsize, auto]
      (d-1-1) edge node{$z_{n - 1}$} (d-1-2)
      (d-1-2) edge node{$z_{n}$} (d-1-3)
              edge node{$\beta_{n}$} (d-2-2)
      (d-1-3) edge node{$\Sigma\beta_{1}$} (d-2-3)
      
      (d-2-1) edge node{$f_{n - 1}$} (d-2-2)
      (d-2-2) edge node{$f_n$} (d-2-3);
    \end{scope}
  \end{tikzpicture}
\end{center}
where the middle row is some $n$-angle $Z_\bullet$ in which $f_i$ occurs as the
$i^\text{th}$ morphism.

These intermediate Toda brackets do, indeed, produce the same set as the ones we
have seen so far which we will prove in this section.

\begin{proposition}\label{prop:i_subsets}
  Let $f_1, f_2, \ldots, f_n$ be composable morphisms in $\C$ as in
  \autoref{eq:seq}. Then for each $1\leq i\leq n$ we have the following set
  inclusions:
  \begin{enumerate}
  \item $\todafc{f_n,\dots,f_2,f_1}\subseteq \todai{f_n,\dots,f_2,f_1}{i}$.
  \item If $f_{j+1}f_j=0$ for $1\leq j \leq i-3$ then 
    \begin{equation*}
      \todai{f_n,\dots,f_2,f_1}{i}\subseteq \todacc{f_n,\dots,f_2,f_1}.
    \end{equation*}
    If $i<4$ the inclusion holds without any assumption on composition of
    consecutive morphisms.
  \item If $f_{j+1}f_j=0$ for $i+2\leq j \leq n-1$ then 
    \begin{equation*}
      \todai{f_n,\dots,f_2,f_1}{i}\subseteq \todaff{f_n,\dots,f_2,f_1}.
    \end{equation*}
    If $i>n-3$ the inclusion holds without any assumption on composition of
    consecutive morphisms.
  \end{enumerate}
\end{proposition}  

\begin{proof}
  Let $\psi \in \todafc{f_n,\dots,f_2,f_1}$ be represented by a diagram as in
  \autoref{figure:fcbracket} such that
  \begin{equation*}
    \psi = \Sigma(\beta_1^{n - 1} \cdots \beta_1^2 \beta_1^1).
  \end{equation*} 
  We compose all the vertical morphisms before row $i$ and all the vertical
  morphisms after row $i$ and discard the rest in order to obtain the diagram in
  question. Thus we get a factorization of $\psi$ as an element of
  $\todai{f_n,\dots,f_2,f_1}{i}$ for any $i$, proving \textbf{(1)}.

  We now prove \textbf{(2)}; \textbf{(3)} is dual. First, note that if
  $\todai{f_n,\dots,f_2,f_1}{i}=\emptyset,$ the statement is trivial. Suppose
  now that we are given an element
  $\psi=\Sigma(\beta_{1}\alpha_1)\in\todai{f_n,\dots,f_2,f_1}{i}$ for $i\geq 4$ and
  let $Y_\bullet$ denote some $n$-angle extension of $f_1$ and suppose
  $Z_\bullet$ is the $n$-angle in which $f_i$ occurs as the $i^\text{th}$
  morphism used in defining $\psi$. The existence of $\psi$ implies that
  $f_if_{i-1}=0$ and $f_{i-1}f_{i-2}=0$, so combining this with the criterion of
  the statement, we know that we may construct the morphisms $\varphi_k$ for
  $3\leq k\leq i$ in the following diagram by repeatedly using the weak cokernel
  property.

  Then we may construct the solid part of the diagram
  \begin{center}
    \begin{tikzpicture}
      % Part 1
      \begin{scope}
        \smalldiagram{d}{2em}{2em}{
          X_1 & X_2 & Y_3 & \cdots\\
          X_1 & X_2 & X_3 & \cdots\\
          Z_{1} & Z_{2} & Z_{3} & \cdots\\
        };
        
        \path[->, font = \scriptsize, auto]
        (d-1-1) edge node{$f_1$} (d-1-2)
        \downequal{d}{1-1}{2-1}
        (d-1-2) edge (d-1-3)
        \downequal{d}{1-2}{2-2}
        (d-1-3) edge (d-1-4)
        edge node{$\phi_3$} (d-2-3)
        
        (d-2-1) edge node{$f_1$} (d-2-2)
                edge node{$\alpha_1$} (d-3-1)
        (d-2-2) edge node{$f_2$} (d-2-3)
                edge node{$\alpha_2$} (d-3-2)
        (d-2-3) edge node{$f_3$} (d-2-4)
                edge node{$\alpha_3$} (d-3-3)
        
        (d-3-1) edge (d-3-2)
        (d-3-2) edge (d-3-3)
        (d-3-3) edge (d-3-4);
      \end{scope}
      
      % Part 2
      \begin{scope}[xshift = 4cm, yshift = -4.5cm]
        \smalldiagram{d}{2em}{2em}{
          \cdots & Y_{i - 1} & Y_{i} & Y_{i + 1} & Y_{i+2} & \cdots & Y_{n} & \Sigma X_1\\
          \cdots & X_{i - 1} & X_i\\
          \cdots & Z_{i-1} & X_i & X_{i + 1} & Z_{i+2} & \cdots & Z_{n} & \Sigma Z_{1}\\
          &&& X_{i + 1} & X_{i + 2} & \cdots & X_n & X_{n + 1}.\\
        };
        
        \path[->, font = \scriptsize, auto]
        (d-1-1) edge (d-1-2)
        (d-1-2) edge (d-1-3)
                edge node{$\phi_{i - 1}$} (d-2-2)
        (d-1-3) edge (d-1-4)
                edge node{$\phi_i$} (d-2-3)
        (d-1-4) edge (d-1-5)
                edge[densely dashed] node{$\phi_{i + 1}'$} (d-3-4)
        (d-1-5) edge (d-1-6)
                edge[densely dashed] node{$\phi_{i + 2}'$} (d-3-5)
        (d-1-6) edge (d-1-7)
        (d-1-7) edge (d-1-8)
                edge[densely dashed] node{$\phi_n'$} (d-3-7)
        (d-1-8) edge node{$\Sigma \alpha_1$}(d-3-8)
        
        (d-2-1) edge node{$f_{i - 2}$} (d-2-2)
        (d-2-2) edge node{$f_{i - 1}$} (d-2-3)
                edge node{$\alpha_{i - 1}$} (d-3-2)
        \downequal{d}{2-3}{3-3}
        
        (d-3-1) edge (d-3-2)
        (d-3-2) edge (d-3-3)
        (d-3-3) edge node{$f_i$} (d-3-4)
        (d-3-4) edge (d-3-5)
        \downequal{d}{3-4}{4-4}
        (d-3-5) edge (d-3-6)
                edge node{$\beta_{i+2}$} (d-4-5)
        (d-3-6) edge (d-3-7)
        (d-3-7) edge (d-3-8)
                edge node{$\beta_{n}$} (d-4-7)
        (d-3-8) edge node{$\Sigma\beta_{1}$} (d-4-8)
        
        (d-4-4) edge node{$f_{i + 1}$} (d-4-5)
        (d-4-5) edge node{$f_{i + 2}$} (d-4-6)
        (d-4-6) edge node{$f_{n - 1}$} (d-4-7)
        (d-4-7) edge node{$f_n$} (d-4-8);
      \end{scope}
    \end{tikzpicture}
  \end{center}
  The tuple of morphisms
  \begin{equation*}
    (\alpha_1,\alpha_2,\alpha_3\varphi_3,\dots,\alpha_{i-1}\varphi_{i-1},\varphi_i)
  \end{equation*}
  is a partial morphism of $n$-angles and, so, applying
  \cite[Lemma~2.1]{Klapproth} we may complete it to a morphism of $n$-angles
  \begin{equation*}
    (\alpha_1,\alpha_2,\alpha_3\varphi_3,\dots,\alpha_{i-1}\varphi_{i-1},\varphi_i,
    \varphi_{i+1}',\dots,\varphi_n')
  \end{equation*}
  as illustrated by the dashed morphisms in the diagram above. Finally, we can
  express $\psi$ as an element of the iterated cofiber Toda bracket by the
  diagram
  \begin{center}
    \begin{tikzpicture}
      % Part 1
      \begin{scope}
        \smalldiagram{d}{2em}{2em}{
          X_1 & X_2 & Y_3 & \cdots\\
          X_1 & X_2 & X_3 & \cdots\\
        };
        
        \path[->, font = \scriptsize, auto]
        (d-1-1) edge node{$f_1$} (d-1-2)
        \downequal{d}{1-1}{2-1}
        (d-1-2) edge (d-1-3)
        \downequal{d}{1-2}{2-2}
        (d-1-3) edge (d-1-4)
        edge node{$\phi_3$} (d-2-3)
        
        (d-2-1) edge node{$f_1$} (d-2-2)
        (d-2-2) edge node{$f_2$} (d-2-3)
        (d-2-3) edge node{$f_3$} (d-2-4);
      \end{scope}
      
      % Part 2
      \begin{scope}[xshift = 4cm, yshift = -2cm]
        \smalldiagram{d}{2em}{2em}{
          \cdots & Y_{i - 1} & Y_{i} & Y_{i + 1} & Y_{i+2} & \cdots & Y_{n} & \Sigma X_1\\
          \cdots & X_{i - 1} & X_i & X_{i + 1} & X_{i + 2} & \cdots & X_n & X_{n + 1}.\\
        };
        
        \path[->, font = \scriptsize, auto]
        (d-1-1) edge (d-1-2)
        (d-1-2) edge (d-1-3)
                edge node{$\phi_{i - 1}$} (d-2-2)
        (d-1-3) edge (d-1-4)
                edge node{$\phi_i$} (d-2-3)
        (d-1-4) edge (d-1-5)
                edge node{$\phi_{i + 1}'$} (d-2-4)
        (d-1-5) edge (d-1-6)
                edge node{$\beta_{i + 2}\phi_{i + 2}'$} (d-2-5)
        (d-1-6) edge (d-1-7)
        (d-1-7) edge (d-1-8)
                edge node{$\beta_n \phi_n'$} (d-2-7)
        (d-1-8) edge node{$\Sigma(\beta_1\alpha_1)$} (d-2-8)
        
        (d-2-1) edge node{$f_{i - 2}$} (d-2-2)
        (d-2-2) edge node{$f_{i - 1}$} (d-2-3)
        (d-2-3) edge node{$f_i$} (d-2-4)
        (d-2-4) edge node{$f_{i + 1}$} (d-2-5)
        (d-2-5) edge node{$f_{i + 2}$} (d-2-6)
        (d-2-6) edge node{$f_{n - 1}$} (d-2-7)
        (d-2-7) edge node{$f_n$} (d-2-8);
      \end{scope}
    \end{tikzpicture}
  \end{center}
  This concludes the proof for all $i\geq 4$. The case $i=3$ follows the same
  argument, however, existence of an element in $\todai{f_n,\dots,f_2,f_1}{3}$
  guarantees $f_2f_1=0$, so we get $\varphi_3$ for free without any extra
  assumptions. The case $i=2$ does not rely on any version of $f_{k+1}f_k=0$,
  since we may skip directly ahead and apply \ref{N3}. The case $i=1$ is
  trivial.
\end{proof}

\begin{theorem}\label{thm:all_coincide}
  Let $\C$ be a pre-$n$-angulated category and consider a diagram of the shape
  \begin{center}
    \begin{tikzpicture}
      \diagram{d}{2em}{2em}{
        X_1 & X_2 & X_3 & \cdots & X_n & X_{n+1},\\
      };
            
      \path[->, auto, font = \scriptsize]
      (d-1-1) edge node{$f_1$} (d-1-2)
      (d-1-2) edge node{$f_2$} (d-1-3)
      (d-1-3) edge node{$f_3$} (d-1-4)
      (d-1-4) edge node{$f_{n-1}$} (d-1-5)
      (d-1-5) edge node{$f_n$} (d-1-6);
    \end{tikzpicture}
  \end{center}
  such that $f_{i+1}f_i=0$ for each $1\leq i\leq n-1$. Then
  \begin{equation*}
    \todai{f_n,\dots,f_2,f_1}j=\todai{f_n,\dots,f_2,f_1}k
  \end{equation*}
  for any $1\leq j,k\leq n$. In particular, this set coincides with the
  fiber-cofiber bracket and the iterated (co)fiber bracket.
\end{theorem}

\begin{proof}
  If the given condition is satisfied, it follows by \autoref{prop:i_subsets} that
  \begin{equation*}
    \todafc{f_n,\dots,f_2,f_1}\subseteq \todai{f_n,\dots,f_2,f_1}{i}\subseteq
    \todacc{f_n,\dots,f_2,f_1}\cap\todaff{f_n,\dots,f_2,f_1}
  \end{equation*}
  for all $1\leq i\leq n$ and, so, by \autoref{thm:main} the statement follows.
\end{proof}

\begin{remark}
  There is some redundancy in these new Toda brackets; we see that the
  definitions line up in the cases of
  \begin{equation*}
    \todai{f_n,\dots,f_2,f_1}{1}=\todacc{f_n,\dots,f_2,f_1} \quad\text{and}\quad
    \todai{f_n,\dots,f_2,f_1}{n}=\todaff{f_n,\dots,f_2,f_1}
  \end{equation*}
  and in triangulated case we further have that
  \begin{equation*}
    \todai{f_3,f_2,f_1}{2}=\todafc{f_3,f_2,f_1}.
  \end{equation*}
  Hence, these intermediate definitions more clearly resemble the fiber-cofiber
  Toda bracket of triangulated categories. However, they come with both the
  drawbacks of the iterated (co)fiber and the fiber-cofiber Toda brackets; i.e.,
  when $i\neq1,n$ they are dependent of the choice of $n$-angle extension of
  $f_i$ and, in most cases, we can have a non-empty $[i]$-intermediate Toda
  bracket even if there are consecutive morphisms that do not compose to
  zero.

  To see the dependence of extensions, we may recycle the example of
  \autoref{rem:extension_dependence}: Suppose that $n=5$ and that we have a
  diagram
  \begin{center}
    \begin{tikzpicture}    
      \diagram{d}{2em}{3em}{
        X_1 & X_2 & X_3 & X_4 & X_5 & X_6\\ 
      };
      
      \path[->, font = \scriptsize, auto]
      (d-1-1) edge node{$f_1$} (d-1-2)
      (d-1-2) edge node{$f_2$} (d-1-3)
      (d-1-3) edge node{$f_3=0$} (d-1-4)
      (d-1-4) edge node{$f_4$} (d-1-5)
      (d-1-5) edge node{$f_5$} (d-1-6);
    \end{tikzpicture}
  \end{center}
  such that $f_{i+1}f_i=0$ for all $1\leq i\leq 4$. Then, by
  \autoref{thm:all_coincide}, we know that
  \begin{equation*}
    \todai{f_5,f_4,0,f_2,f_1}3=\todafc{f_5,f_4,0,f_2,f_1}.
  \end{equation*}
  Since this fiber-cofiber Toda bracket is equal to the indeterminacy subgroup,
  as we have previously argued in \autoref{rem:extension_dependence}, it follows
  that if the indeterminacy is non-trivial and we suppose that the
  $[3]$-intermediate Toda bracket is independent of the choice of extension of
  $f_3$, we will reach a contradiction. This is because we could choose the
  extension $(TX_3)_\bullet[-1]\oplus (TX_4)[-3]$ and then the following diagram
  shows that all elements of the $[3]$-intermediate Toda bracket are trivial:
  \begin{center}
    \begin{tikzpicture}    
      \diagram{d}{2em}{3em}{
        X_1 & X_2 & X_3 & & &\\
        0 & X_3 & X_3 & X_4 & X_4 & 0\\
        & & & X_4 & X_5 & X_6. \\
      };
      
      \path[->, font = \scriptsize, auto]
      (d-1-1) edge node{$f_1$} (d-1-2)
      (d-1-2) edge node{$f_2$} (d-1-3)
      (d-2-1) edge node{} (d-2-2)
      (d-2-2) edge node{$1$} (d-2-3)
      (d-2-3) edge node{$0$} (d-2-4)
      (d-2-4) edge node{$1$} (d-2-5)
      (d-2-5) edge node{} (d-2-6)
      
      (d-3-4) edge node{$f_4$} (d-3-5)
      (d-3-5) edge node{$f_5$} (d-3-6)
      
      (d-1-1) edge node{} (d-2-1)
      (d-1-2) edge node{$f_2$} (d-2-2)
      \downequal{d}{1-3}{2-3}
      
      \downequal{d}{2-4}{3-4}
      (d-2-5) edge node{$f_4$} (d-3-5)
      (d-2-6) edge node{} (d-3-6);
    \end{tikzpicture}
  \end{center}

  To see that the $[i]$-intermediate Toda brackets may be non-empty even when
  there are two consecutive morphisms that do not compose to zero, simply
  consider the example in a pre-$5$-angulated category:
  \begin{equation*}
    \todai{1,1,0,1,0}2=\C(\Sigma X,\Sigma X).
  \end{equation*}
  This equality is illustrated by the diagram
  \begin{center}
    \begin{tikzpicture}    
      \diagram{d}{2em}{3em}{
        X & X\\ 
        X & X & X & 0 & \Sigma X & \Sigma X\\
        && X & \Sigma X & \Sigma X & \Sigma X.\\
      };
      
      \path[->, font = \scriptsize, auto]
      (d-1-1) edge node{$0$} (d-1-2)
      
      (d-2-1) edge node{$0$} (d-2-2)
      (d-2-2) edge node{$1$} (d-2-3)
      (d-2-3) edge (d-2-4)
      (d-2-4) edge (d-2-5)
      (d-2-5) edge node{$1$} (d-2-6)
      
      (d-3-3) edge node{$0$} (d-3-4)
      (d-3-4) edge node{$1$} (d-3-5)
      (d-3-5) edge node{$1$} (d-3-6)
      
      (d-1-1) edge node{$1$} (d-2-1)
      \downequal{d}{1-2}{2-2}
      
      \downequal{d}{2-3}{3-3}
      (d-2-4) edge (d-3-4)
      (d-2-5) edge node{$\psi$} (d-3-5)
      (d-2-6) edge node{$\psi$} (d-3-6);
    \end{tikzpicture}
  \end{center}
  
  Notably, however, if $n$ is small enough we may conjure a proposition similar
  to \autoref{lem:nonempty_fcbracket}: the existence of an element of the
  $[i]$-intermediate Toda bracket implies the four identities $f_{j+1}f_j=0$ for
  $i-2\leq j\leq i+1$ and, so, if $n=5$ and there exists an element of the
  $[3]$-intermediate bracket, it follows that all the morphisms in the diagram
  must compose to zero. Similarly if $n=4$ and there is an element of either the
  $[2]$-intermediate or $[3]$-intermediate Toda bracket the morphisms must
  compose trivially. The converse, i.e., that if all $2$-fold composites are
  zero then all the various brackets must be non-empty, is a consequence of
  \autoref{thm:all_coincide} and \autoref{lem:nonempty_fcbracket}.
\end{remark}

%
% The juggling formulas
%
\section{The juggling formulas}
\label{sec:juggling_formulas}

Toda brackets satisfy various formulas involving composition and addition of
morphisms. For Massey products, the algebraic analogue of Toda brackets,
references on juggling formulas include \cite{May69}, \cite[§A1.4]{Ravenel04},
and \cite[§2.2]{Isaksen19}. For Toda brackets, references are more scattered;
see \cite[Lemma~3.5]{CF} for a small sample. In this section, we prove various
juggling formulas for Toda brackets in pre-$n$-angulated categories. In
particular, this provides a collected reference for $3$-fold Toda brackets in
pretriangulated categories.

\begin{proposition}\label{prop:fi_is_zero}
  Let $f_1, f_2, \ldots, f_n$ be composable morphisms in $\C$ as in
  \autoref{eq:seq} such that there is some $1\leq i\leq n$ satisfying $f_i = 0$.
  If $\todafc{f_n,\dots,f_2,f_1}$ is non-empty, then $0\in
  \todafc{f_n,\dots,f_2,f_1}$. The same statement holds for
  $\todacc{f_n,\dots,f_2,f_1}$ and $\todaff{f_n,\dots,f_2,f_1}$.
\end{proposition}

\begin{proof}
  Assuming $\todafc{f_n,\dots,f_2,f_1} \neq \emptyset$, we get by
  \autoref{lem:nonempty_fcbracket} and \autoref{thm:main} that all the different
  variants of the Toda brackets coincide and, so, it suffices to prove the
  statement for the iterated cofiber Toda bracket. The iterated fiber Toda
  bracket is dual.

  Suppose first that $f_1=0$. Then we may construct the commutative diagram
  \begin{center}
    \begin{tikzpicture}
      \diagram{d}{2.5em}{2em}{
        X_1 & X_2 & X_2 & 0 & \cdots &0& \Sigma X_1 & \Sigma X_{1}\\
        X_1 & X_2 & X_3 & X_4 & \cdots &X_{n-1}& X_n & X_{n+1}\\
      };
      
      \path[->, auto, font = \scriptsize]
      (d-1-1) edge node{$0$} (d-1-2)
      (d-1-2) edge node{$1$} (d-1-3)
      (d-1-3) edge (d-1-4)
      (d-1-4) edge (d-1-5)
      (d-1-5) edge (d-1-6)
      (d-1-6) edge (d-1-7)
      (d-1-7) edge node{$1$} (d-1-8)
      
      (d-2-1) edge node{$0$} (d-2-2)
      (d-2-2) edge node{$f_2$} (d-2-3)
      (d-2-3) edge node{$f_3$} (d-2-4)
      (d-2-4) edge node{$f_4$} (d-2-5)
      (d-2-5) edge node{$f_{n-2}$} (d-2-6)
      (d-2-6) edge node{$f_{n-1}$} (d-2-7)
      (d-2-7) edge node{$f_n$} (d-2-8)
      
      \downequal{d}{1-1}{2-1}
      \downequal{d}{1-2}{2-2}
      (d-1-3) edge node{$f_2$} (d-2-3)
      (d-1-4) edge (d-2-4)
      (d-1-6) edge (d-2-6)
      (d-1-7) edge node{$f_{n-1}$} (d-2-7)
      (d-1-8) edge node{$0$} (d-2-8);
    \end{tikzpicture}
  \end{center}
  which exhibits $0\in\todacc{f_n,\dots,f_2,f_1}$. Now suppose $f_i=0$ for some
  $2\leq i\leq n$. We know the Toda bracket $\todacc{f_n,\dots,f_2,f_1}$ is
  non-empty, so simply take some element and represent it by the diagram
  \begin{center}
    \begin{tikzpicture}
      \diagram{d}{2.5em}{2em}{
        X_1 & X_2 & Y_3 & \cdots &  Y_i & Y_{i+1} & \cdots & Y_n & \Sigma X_{1}\\
        X_1 & X_2 & X_3 & \cdots & X_i & X_{i+1} & \cdots & X_n & X_{n+1}.\\
      };
      
      \path[->, auto, font = \scriptsize]
      (d-1-1) edge node{$f_1$} (d-1-2)
      (d-1-2) edge node{$y_2$} (d-1-3)
      (d-1-3) edge node{$y_3$} (d-1-4)
      (d-1-4) edge node{$y_{i-1}$} (d-1-5)
      (d-1-5) edge node{$y_i$} (d-1-6)
      (d-1-6) edge node{$y_{i+1}$} (d-1-7)
      (d-1-7) edge node{$y_{n-1}$} (d-1-8)
      (d-1-8) edge node{$y_{n}$} (d-1-9)
      
      (d-2-1) edge node{$f_1$} (d-2-2)
      (d-2-2) edge node{$f_2$} (d-2-3)
      (d-2-3) edge node{$f_3$} (d-2-4)
      (d-2-4) edge node{$f_{i-1}$} (d-2-5)
      (d-2-5) edge node{$0$} (d-2-6)
      (d-2-6) edge node{$f_{i+1}$} (d-2-7)
      (d-2-7) edge node{$f_{n-1}$} (d-2-8)
      (d-2-8) edge node{$f_{n}$} (d-2-9)
      
      \downequal{d}{1-1}{2-1}
      \downequal{d}{1-2}{2-2}
      (d-1-3) edge node{$\varphi_3$} (d-2-3)
      (d-1-5) edge node{$\varphi_{i}$} (d-2-5)
      (d-1-6) edge node{$\varphi_{i+1}$} (d-2-6)
      (d-1-8) edge node{$\varphi_{n}$} (d-2-8)
      (d-1-9) edge node{$\psi$} (d-2-9);
    \end{tikzpicture}
  \end{center}
  However, if we were to exchange $\varphi_j$ and $\psi$ with 0 for all $j\geq
  i+1$, the resulting diagram would still commute and it would exhibit $0$ as an
  element of the Toda bracket.
\end{proof}

The assumption in \autoref{prop:fi_is_zero} cannot be dropped in general. For
instance, in \autoref{eq:ff_cc_not_equal}, we have $f_2 = 0$, the iterated
cofiber Toda bracket is non-empty while the iterated fiber Toda bracket is
empty.

\begin{lemma}\label{lem:MinusSigns}
  Let $f_1, f_2, \ldots, f_n$ be composable morphisms in $\C$ as in
  \autoref{eq:seq}. Then for any two indices $1 \leq j < k \leq n$, we obtain an
  isomorphic diagram
  \begin{center}
    \begin{tikzpicture}
      \diagram{d}{2em}{2em}{
        X_1 & X_2 & X_3 & \cdots & X_n & X_{n + 1}\\
      };
      
      \path[->, font = \scriptsize, auto]
      (d-1-1) edge node{$f_1'$} (d-1-2)
      (d-1-2) edge node{$f_2'$} (d-1-3)
      (d-1-3) edge node{$f_3'$} (d-1-4)
      (d-1-4) edge node{$f_{n - 1}'$} (d-1-5)
      (d-1-5) edge node{$f_n'$} (d-1-6);
    \end{tikzpicture}
  \end{center}
  by putting a minus sign in front of $f_j$ and $f_k$:
  \begin{equation*}
    f'_i = \begin{cases}
      -f_i &\text{if } i=j,k \\
      f_i &\text{otherwise.} \\
    \end{cases}
  \end{equation*}
\end{lemma}

\begin{proposition}[Subadditivity]\label{prop:subadd}
  Let $\C$ be a pre-$n$-angulated category and consider the diagram
  \begin{center}
    \begin{tikzpicture}
      \diagram{d}{2em}{2em}{
        X_1 & X_2 & X_3 & \cdots & X_n & X_{n+1},\\
      };
      
      \path[->, auto, font = \scriptsize]
      (d-1-1) edge node{$f_1$} (d-1-2)
      (d-1-2) edge node{$f_2$} (d-1-3)
      (d-1-3) edge node{$f_3$} (d-1-4)
      (d-1-4) edge node{$f_{n-1}$} (d-1-5)
      (d-1-5) edge node{$f_n$} (d-1-6);
    \end{tikzpicture}
  \end{center}
  and morphisms $f_j'$ that share source and target with the morphisms $f_j$ for
  $1\leq j\leq n$, and suppose that the identities
  \begin{equation*}
    f_{i+1}f_i=0,\quad f_{i+1}'f_i=0\quad\text{and}\quad f_{i+1}f_i'=0
  \end{equation*}
  are satisfied for $1\leq i\leq n-1$. Then the following statements hold:
  \begin{enumerate}[label=(\arabic*), ref={Proposition
      \theproposition.(\arabic*)}]
  \item $\toda{f_n+f_n',f_{n-1},\dots,f_2,f_1}\subseteq
    \toda{f_n,f_{n-1},\dots,f_2,f_1} + \toda{f_n',f_{n-1},\dots,f_2,f_1}$.
  \item $\toda{f_n,\dots,f_2,f_1+f_1'}\subseteq \toda{f_n,\dots,f_2,f_1} +
    \toda{f_n,f_{n-1},\dots,f_2,f_1'}$.
  \item For $2\leq i\leq n-1$, we have the identities
    \begin{equation*}
      \toda{f_n,\dots,f_i+f_i',\dots,f_2,f_1} = \toda{f_n,\dots,f_i,\dots,f_2,f_1} +
      \toda{f_n,\dots,f_i',\dots,f_2,f_1}.
    \end{equation*}
  \item \label{item:Minus} For $1\leq i\leq n$, we have
    \begin{equation*}
      \toda{f_n, \ldots, -f_i, \ldots, f_1} = -\toda{f_n, \ldots, f_i, \ldots, f_1}. 
    \end{equation*}
  \end{enumerate}
\end{proposition}

\begin{proof}
  We claim that for all $1\leq i\leq n$ it is the case that
  \begin{equation*}
    \toda{f_n,\dots,f_i+f_i',\dots,f_2,f_1}\subseteq
    \toda{f_n,\dots,f_i,\dots,f_2,f_1}+\toda{f_n,\dots,f_i',\dots,f_2,f_1},
  \end{equation*}
  which would prove \textbf{(1)}, \textbf{(2)}, and half of \textbf{(3)}. To
  prove the claim, consider an element $\psi\in
  \todacc{f_n,\dots,f_i+f_i',\dots,f_2,f_1}$ represented by the diagram
  \begin{center}
    \begin{tikzpicture}
      \diagram{d}{2.5em}{2em}{
        X_1 & X_2 & Y_3 & \cdots &  Y_i &[0.5em] Y_{i+1} & \cdots & Y_n & \Sigma X_{1}\\
        X_1 & X_2 & X_3 & \cdots & X_i & X_{i+1} & \cdots & X_n & X_{n+1}.\\
      };
      
      \path[->, auto, font = \scriptsize]
      (d-1-1) edge node{$f_1$} (d-1-2)
      (d-1-2) edge node{$y_2$} (d-1-3)
      (d-1-3) edge node{$y_3$} (d-1-4)
      (d-1-4) edge node{$y_{i-1}$} (d-1-5)
      (d-1-5) edge node{$y_i$} (d-1-6)
      (d-1-6) edge node{$y_{i+1}$} (d-1-7)
      (d-1-7) edge node{$y_{n-1}$} (d-1-8)
      (d-1-8) edge node{$y_{n}$} (d-1-9)
      
      (d-2-1) edge node{$f_1$} (d-2-2)
      (d-2-2) edge node{$f_2$} (d-2-3)
      (d-2-3) edge node{$f_3$} (d-2-4)
      (d-2-4) edge node{$f_{i-1}$} (d-2-5)
      (d-2-5) edge node{$f_i+f_i'$} (d-2-6)
      (d-2-6) edge node{$f_{i+1}$} (d-2-7)
      (d-2-7) edge node{$f_{n-1}$} (d-2-8)
      (d-2-8) edge node{$f_{n}$} (d-2-9)
      
      \downequal{d}{1-1}{2-1}
      \downequal{d}{1-2}{2-2}
      (d-1-3) edge node{$\varphi_3$} (d-2-3)
      (d-1-5) edge node{$\varphi_{i}$} (d-2-5)
      (d-1-6) edge node{$\varphi_{i+1}$} (d-2-6)
      (d-1-8) edge node{$\varphi_{n}$} (d-2-8)
      (d-1-9) edge node{$\psi$} (d-2-9);
    \end{tikzpicture}
  \end{center}
  We take $\varphi_1\coloneq 1_{X_1}$, $\varphi_2\coloneq 1_{X_2}$, and
  $y_{-1}\coloneq(-1)^n\Sigma^{-1}y_n$. By assumption we have
  \begin{equation*}
    f_i\varphi_i y_{i-1}=f_if_{i-1}\varphi_{i-1}=0\quad\text{and}\quad
    f_i'\varphi_iy_{i-1}=f_i'f_{i-1}\varphi_{i-1}=0,
  \end{equation*}
  so there must be morphisms $\varphi_{i+1}^{(1)},\varphi_{i+1}^{(2)}\colon
  Y_{i+1}\to X_{i+1}$ such that $\varphi_{i+1}^{(1)}y_i=f_i\varphi_i$ and
  $\varphi_{i+1}^{(2)}y_i=f_i'\varphi_i$. We then see that
  \begin{equation*}
    (\varphi_{i+1}-\varphi_{i+1}^{(1)}-\varphi_{i+1}^{(2)})y_i =(f_i+f_i')\varphi_i -
    f_i\varphi_i-f_i'\varphi_i=0
  \end{equation*}
  so there is some $h_{i+1}\colon Y_{i+2}\to X_{i+1}$ such that
  \begin{equation*}
    \varphi_{i+1}=\varphi_{i+1}^{(1)}+\varphi_{i+1}^{(2)}+h_{i+1}y_{i+1}.
  \end{equation*}
  We then take $\tilde\varphi_{i+1}^{(2)}\coloneq
  \varphi_{i+1}^{(2)}+h_{i+1}y_{i+1}$; note that we still have the relation
  \begin{equation*}
    \tilde\varphi^{(2)}_{i+1}y_i=(\varphi_{i+1}^{(2)}+h_{i+1}y_{i+1})y_i =
    \varphi_{i+1}^{(2)}y_i = f_i'\varphi_i.
  \end{equation*}
  Now suppose we have commutative diagrams 
  \begin{center}
    \begin{tikzpicture}
      \smalldiagram{d}{2.5em}{2em}{
        X_1 & X_2 & Y_3 & \cdots &  Y_i & Y_{i+1} & \cdots & Y_j & Y_{j+1} & \cdots\\
        X_1 & X_2 & X_3 & \cdots & X_i & X_{i+1} & \cdots & X_j & X_{j+1} & \cdots\\
      };
      
      \path[->, auto, font = \scriptsize]
      (d-1-1) edge node{$f_1$} (d-1-2)
      (d-1-2) edge node{$y_2$} (d-1-3)
      (d-1-3) edge node{$y_3$} (d-1-4)
      (d-1-4) edge node{$y_{i-1}$} (d-1-5)
      (d-1-5) edge node{$y_i$} (d-1-6)
      (d-1-6) edge node{$y_{i+1}$} (d-1-7)
      (d-1-7) edge node{$y_{j-1}$} (d-1-8)
      (d-1-8) edge node{$y_{j}$} (d-1-9)
      (d-1-9) edge node{$y_{j+1}$} (d-1-10)
      
      (d-2-1) edge node{$f_1$} (d-2-2)
      (d-2-2) edge node{$f_2$} (d-2-3)
      (d-2-3) edge node{$f_3$} (d-2-4)
      (d-2-4) edge node{$f_{i-1}$} (d-2-5)
      ([yshift=-.2em] d-2-5.east) edge node[swap]{$f_i'$} ([yshift=-.2em] d-2-6.west)
      ([yshift=.2em] d-2-5.east) edge node{$f_i$} ([yshift=.2em] d-2-6.west)
      (d-2-6) edge node{$f_{i+1}$} (d-2-7)
      (d-2-7) edge node{$f_{j-1}$} (d-2-8)
      (d-2-8) edge node{$f_{j}$} (d-2-9)
      (d-2-9) edge node{$f_{j+1}$} (d-2-10)
      
      \downequal{d}{1-1}{2-1}
      \downequal{d}{1-2}{2-2}
      (d-1-3) edge node{$\varphi_3$} (d-2-3)
      (d-1-5) edge node{$\varphi_{i}$} (d-2-5)
      ([xshift=-.2em] d-1-6.south) edge node[swap]{$\varphi_{i+1}^{(1)}$}
      ([xshift=-.2em] d-2-6.north)
      ([xshift=.2em] d-1-6.south) edge node{$\tild\varphi_{i+1}^{(2)}$}
      ([xshift=.2em] d-2-6.north)
      ([xshift=-.2em] d-1-8.south) edge node[swap]{$\varphi_{j}^{(1)}$}
      ([xshift=-.2em] d-2-8.north)
      ([xshift=.2em] d-1-8.south) edge node{$\tild\varphi_{j}^{(2)}$}
      ([xshift=.2em] d-2-8.north);
    \end{tikzpicture}
  \end{center}
  for some $j\geq i+1$ such that $\varphi_j = \varphi^{(1)}_j +
  \tilde\varphi^{(2)}_j$. Then we may build upon each of the summands to get
  morphisms $\varphi_{j+1}^{(1)},\varphi_{j+1}^{(2)}\colon Y_{j+1}\to X_{j+1}$
  such that $\varphi_{j+1}^{(1)}y_j=f_j\varphi^{(1)}_j$ and
  $\varphi_{j+1}^{(2)}y_j=f_j\tilde\varphi^{(2)}_j$, since
  \begin{equation*}
    f_j\varphi_{j}^{(1)}y_{j-1}=f_jf_{j-1}\varphi_{j-1}^{(1)}=0\quad\text{and}\quad
    f_j\tilde\varphi_{j}^{(2)}y_{j-1}=f_jf_{j-1}\tilde\varphi_{j-1}^{(2)}=0.
  \end{equation*}
  Then we see that 
  \begin{align*}
    (\varphi_{j+1}-\varphi_{j+1}^{(1)}-\varphi^{(2)}_{j+1})y_j
    =f_j(\varphi_j-\varphi_j^{(1)}-\tilde\varphi^{(2)}_{j})=0
  \end{align*}
  so, like earlier, there must be some morphism $h_{j+1}\colon Y_{j+1}\to
  X_{j+1}$ such that
  \begin{equation*}
    \varphi_{j+1}=\varphi_{j+1}^{(1)}+\varphi_{j+1}^{(2)}+h_{j+1}y_{j+1}
  \end{equation*}
  and, so, we take
  $\tilde\varphi_{j+1}^{(2)}\coloneq\varphi^{(2)}_{j+1}+h_{j+1}y_{j+1}$ and get
  the identity
  \begin{equation*}
    \varphi_{j+1}=\varphi_{j+1}^{(1)}+\tilde\varphi_{j+1}^{(2)}.
  \end{equation*}
  Thus, by setting $j=n$, it follows by induction that there are morphisms
  \begin{equation*}
    \psi^{(1)}\in \todacc{f_n,\dots,f_i,\dots,f_2,f_1}\quad\text{and}\quad
    \tilde\psi^{(2)}\in \todacc{f_n,\dots,f_i',\dots,f_2,f_1}
  \end{equation*}
  such that
  \begin{equation*}
    \psi=\psi^{(1)}+\tilde\psi^{(2)}
  \end{equation*}
  and the claim follows.
  
  \textbf{(3)} Suppose now that $2\leq i \leq n-1$ and consider two elements
  \begin{equation*}
    \psi\in\todacc{f_n,\dots,f_i,\dots,f_2,f_1}\quad\text{and}\quad
    \psi'\in\todacc{f_n,\dots,f_i',\dots,f_2,f_1}.
  \end{equation*}
  We may represent these by iterated cofiber diagrams that share the same top
  row and with vertical morphisms given by, respectively, $\varphi_j$ and
  $\varphi_j'$ for $3\leq j\leq n$. We may now use the morphisms $\varphi_k$ for
  $3\leq k\leq i$ to generate a new element
  \begin{equation*}
    \mu\in\todacc{f_n,\dots,f_i',\dots,f_2,f_1}
  \end{equation*}
  by repeatedly applying the weak cokernel property. Thus we may represent $\mu$
  in the iterated cofiber Toda bracket with vertical morphism $\nu_k$ where
  $\nu_k=\varphi_k$ for $3\leq k\leq i$:
  \begin{center}
    \begin{tikzpicture}
      \diagram{d}{2.5em}{2em}{
        X_1 & X_2 & Y_3 & \cdots &  Y_i & Y_{i+1} & \cdots & Y_n & \Sigma X_{1}\\
        X_1 & X_2 & X_3 & \cdots & X_i & X_{i+1} & \cdots & X_n & X_{n+1}.\\
      };
      
      \path[->, auto, font = \scriptsize]
      (d-1-1) edge node{$f_1$} (d-1-2)
      (d-1-2) edge node{$y_2$} (d-1-3)
      (d-1-3) edge node{$y_3$} (d-1-4)
      (d-1-4) edge node{$y_{i-1}$} (d-1-5)
      (d-1-5) edge node{$y_i$} (d-1-6)
      (d-1-6) edge node{$y_{i+1}$} (d-1-7)
      (d-1-7) edge node{$y_{n-1}$} (d-1-8)
      (d-1-8) edge node{$y_{n}$} (d-1-9)
      
      (d-2-1) edge node{$f_1$} (d-2-2)
      (d-2-2) edge node{$f_2$} (d-2-3)
      (d-2-3) edge node{$f_3$} (d-2-4)
      (d-2-4) edge node{$f_{i-1}$} (d-2-5)
      (d-2-5) edge node{$f_i'$} (d-2-6)
      (d-2-6) edge node{$f_{i+1}$} (d-2-7)
      (d-2-7) edge node{$f_{n-1}$} (d-2-8)
      (d-2-8) edge node{$f_{n}$} (d-2-9)
      
      \downequal{d}{1-1}{2-1}
      \downequal{d}{1-2}{2-2}
      (d-1-3) edge node{$\varphi_3$} (d-2-3)
      (d-1-5) edge node{$\varphi_{i}$} (d-2-5)
      (d-1-6) edge node{$\nu_{i+1}$} (d-2-6)
      (d-1-8) edge node{$\nu_{n}$} (d-2-8)
      (d-1-9) edge node{$\mu$} (d-2-9);
    \end{tikzpicture}
  \end{center}
  Note also that this yields a new element of
  $\todacc{f_n,\dots,f_i+f_i',\dots,f_2,f_1}$ from $\psi$, namely $\psi+\mu$, by
  the commutative diagram
  \begin{center}
    \begin{tikzpicture}
      \diagram{d}{2.5em}{2em}{
        X_1 & X_2 & Y_3 & \cdots & Y_i &[0.5em] Y_{i+1} & \cdots & Y_n & \Sigma X_{1}\\
        X_1 & X_2 & X_3 & \cdots & X_i & X_{i+1} & \cdots & X_n & X_{n+1}.\\
      };
      
      \path[->, auto, font = \scriptsize]
      (d-1-1) edge node{$f_1$} (d-1-2)
      (d-1-2) edge node{$y_2$} (d-1-3)
      (d-1-3) edge node{$y_3$} (d-1-4)
      (d-1-4) edge node{$y_{i-1}$} (d-1-5)
      (d-1-5) edge node{$y_i$} (d-1-6)
      (d-1-6) edge node{$y_{i+1}$} (d-1-7)
      (d-1-7) edge node{$y_{n-1}$} (d-1-8)
      (d-1-8) edge node{$y_{n}$} (d-1-9)
      
      (d-2-1) edge node{$f_1$} (d-2-2)
      (d-2-2) edge node{$f_2$} (d-2-3)
      (d-2-3) edge node{$f_3$} (d-2-4)
      (d-2-4) edge node{$f_{i-1}$} (d-2-5)
      (d-2-5) edge node{$f_i+f_i'$} (d-2-6)
      (d-2-6) edge node{$f_{i+1}$} (d-2-7)
      (d-2-7) edge node{$f_{n-1}$} (d-2-8)
      (d-2-8) edge node{$f_{n}$} (d-2-9)
      
      \downequal{d}{1-1}{2-1}
      \downequal{d}{1-2}{2-2}
      (d-1-3) edge node{$\varphi_3$} (d-2-3)
      (d-1-5) edge node{$\varphi_{i}$} (d-2-5)
      (d-1-6) edge node{$\varphi_{i+1}+\nu_{i+1}$} (d-2-6)
      (d-1-8) edge node{$\varphi_{n}+\nu_{n}$} (d-2-8)
      (d-1-9) edge node{$\psi+\mu$} (d-2-9);
    \end{tikzpicture}
  \end{center}
  Since $\psi'$ and $\mu$ lie in the same Toda bracket, suppose $h_n\colon
  \Sigma X_1\to X_n$ and $h_{n+1}\colon \Sigma X_2\to X_{n+1}$ are morphisms
  such that
  \begin{equation*}
    \psi'-\mu=f_nh_n+h_{n+1}\Sigma f_1.
  \end{equation*}
  Then
  \begin{equation*}
    \psi+\psi'-(\psi+\mu)=f_nh_n+h_{n+1}\Sigma f_1,
  \end{equation*}
  and thus we see that $\psi+\psi'$ differs from an element of the Toda bracket
  $\todacc{f_n,\dots,f_i+f_i',\dots,f_2,f_1}$ by the indeterminacy. Hence, % We
  may thus conclude that
  \begin{equation*}
    \todacc{f_n,\dots,f_i+f_i',\dots,f_2,f_1}=\todacc{f_n,\dots,f_i,\dots,f_2,f_1} +
    \todacc{f_n,\dots,f_i',\dots,f_2,f_1}
  \end{equation*}
  for $2\leq i\leq n-1$.

  \textbf{(4)} By \autoref{lem:InvariantIso} and \autoref{lem:MinusSigns}, we
  may assume that the minus sign is in front of the last morphism: $i=n$.
  Elements of $\todacc{f_n, \ldots, f_1}$ are the morphisms $\psi \colon \Sigma
  X_1 \to X_{n+1}$ appearing in \autoref{eq:TodaCC}.
  % a commutative diagram~\eqref{eq:TodaCC}.
  But that diagram commutes if and only if the following diagram commutes:
  \begin{center}
    \begin{tikzpicture}
      \diagram{d}{2em}{2em}{
        X_1 & X_2 & Y_3 & Y_4 & \cdots & Y_n & \Sigma X_1\\
        X_1 & X_2 & X_3 & X_4 & \cdots & X_n & X_{n + 1}.\\
      };
      
      \path[->, auto, font = \scriptsize]
      (d-1-1) edge node{$f_1$} (d-1-2)
      \downequal{d}{1-1}{2-1}
      (d-1-2) edge (d-1-3)
      \downequal{d}{1-2}{2-2}
      (d-1-3) edge (d-1-4)
              edge node{$\phi_3$} (d-2-3)
      (d-1-4) edge node{$\phi_4$} (d-2-4)
              edge (d-1-5)
      (d-1-5) edge (d-1-6)
      (d-1-6) edge (d-1-7)
              edge node{$\phi_{n}$} (d-2-6)
      (d-1-7) edge node{$-\psi$} (d-2-7)
      
      (d-2-1) edge node{$f_1$} (d-2-2)
      (d-2-2) edge node{$f_2$} (d-2-3)
      (d-2-3) edge node{$f_3$} (d-2-4)
      (d-2-4) edge node{$f_4$} (d-2-5)
      (d-2-5) edge node{$f_{n-1}$} (d-2-6)
      (d-2-6) edge node{$-f_n$} (d-2-7);
    \end{tikzpicture}
  \end{center}    
  By \autoref{lem:AnyExtension}, we may use the same $n$-angle extension of
  $f_1$ in the top row to compute $\todacc{-f_n, \ldots, f_2, f_1}$, which
  proves the claim.
\end{proof}

\begin{remark}
  First, note that the proof of \ref{item:Minus} did not use the condition
  $f_{i+1} f_i = 0$.
    
  Second, for a fixed position $i$, we note that in proving
  \autoref{prop:subadd}, all we actually needed to prove the relevant inclusions
  for the $f_i$ and $f_i'$ were the equations
  \begin{equation}\label{eq:relations}
    f_{j+1}f_j=0,\quad f_{j+1}'f_j=0,\quad\text{and}\quad f_{j+1}f_j'=0
  \end{equation}
  for $i-1\leq j\leq n-1$. This entails that the inclusion
  \begin{equation*}
    \todacc{f_n,\dots,f_i+f_i',\dots,f_2,f_1}\subseteq
    \todacc{f_n,\dots,f_i,\dots,f_2,f_1} + \todacc{f_n,\dots,f_i',\dots,f_2,f_1}
  \end{equation*}
  holds under this looser restriction and, dually, 
  \begin{equation*}
    \todaff{f_n,\dots,f_i+f_i',\dots,f_2,f_1}\subseteq
    \todaff{f_n,\dots,f_i,\dots,f_2,f_1} + \todaff{f_n,\dots,f_i',\dots,f_2,f_1}
  \end{equation*}
  if the equations \eqref{eq:relations} are satisfied for $1\leq j\leq i$. In
  the case of the fiber-cofiber Toda bracket, if all three relevant brackets are
  non-empty, we are guaranteed the equations \eqref{eq:relations} and we may
  apply that they coincide with the iterated (co)fiber Toda brackets.
\end{remark}

\begin{proposition}\label{prop:submult}
  Let $\C$ be a pre-$n$-angulated category and consider the diagram
  \begin{center}
    \begin{tikzpicture}
      \diagram{d}{2em}{2em}{
        X_1 & X_2 & X_3 & \cdots & X_n & X_{n+1} & X_{n+2}.\\
      };
      
      \path[->, auto, font = \scriptsize]
      (d-1-1) edge node{$f_1$} (d-1-2)
      (d-1-2) edge node{$f_2$} (d-1-3)
      (d-1-3) edge node{$f_3$} (d-1-4)
      (d-1-4) edge node{$f_{n-1}$} (d-1-5)
      (d-1-5) edge node{$f_n$} (d-1-6)
      (d-1-6) edge node{$f_{n+1}$} (d-1-7);
    \end{tikzpicture}
  \end{center}
  Then the following statements hold:
  \begin{enumerate}
  \item $f_{n+1}\todacc{f_n,\dots,f_2,f_1}\subseteq
    \todacc{f_{n+1}f_n,\dots,f_2,f_1}$.
  \item $\todacc{f_{n+1},\dots,f_3,f_2}\Sigma f_1\subseteq
    \todacc{f_{n+1},\dots,f_3,f_2f_1}$.
  \end{enumerate}
  (The same statements hold for the fiber-cofiber Toda bracket and the iterated
  fiber Toda bracket.)
  \begin{enumerate}[resume, label=(\arabic*), ref={Proposition
      \theproposition.(\arabic*)}]
  \item \label{prop-6-5-3} $\todacc{f_{n+1},\dots,f_3,f_2f_1}\subseteq
    \todacc{f_{n+1},\dots,f_3f_2,f_1}$.
  \item \label{prop-6-5-4} For $3\leq i\leq n$ we have
    \begin{equation*}
      \todacc{f_{n+1},\dots,f_{i+1}f_i,f_{i-1},\dots,f_2,f_1}\subseteq
      \todacc{f_{n+1},\dots,f_{i+1},f_if_{i-1},\dots,f_2,f_1}.
    \end{equation*}
  \end{enumerate}
  Dually, we have
  \begin{enumerate}[label=(\arabic*'), ref={Proposition
      \theproposition.(\arabic*')}]
    \setcounter{enumi}{2}
  \item \label{prop-6-5-3-prime} $\todaff{f_{n+1}f_n,f_{n-1},\dots,f_1}\subseteq
    \todaff{f_{n+1},f_nf_{n-1}\dots,f_1}$.
  \item \label{prop-6-5-4-prime} For $3\leq i\leq n$ we have
    \begin{equation*}
      \todaff{f_{n+1},\dots,f_{i+1},f_if_{i-1},\dots,f_2,f_1}\subseteq
      \todaff{f_{n+1},\dots,f_{i+1}f_i,f_{i-1},\dots,f_2,f_1}.
    \end{equation*}
  \end{enumerate}
\end{proposition}

\begin{proof}
  We prove the statements \textbf{(1)} and \textbf{(2)} for the iterated cofiber
  Toda bracket; the iterated fiber Toda bracket is dual. In the case of the
  fiber-cofiber Toda bracket, if the left side of the inclusion is non-empty, so
  is the right, and therefore we may apply \autoref{thm:main} and reduce to the
  case of the iterated (co)fiber case.

  \textbf{(1)} Given a morphism $\psi\in\todacc{f_{n},\dots,f_2,f_1}$
  represented by a diagram with vertical morphisms given by $\varphi_i$ for
  $3\leq i\leq n$, we may exhibit
  $f_{n+1}\psi\in\todacc{f_{n+1}f_n,\dots,f_2,f_1}$ by the diagram
  \begin{center}
    \begin{tikzpicture}    
      \diagram{d}{2em}{3em}{
        X_1 & X_2 & Y_3 & \cdots & Y_n & \Sigma X_1\\
        X_1 & X_2 & X_3 & \cdots & X_{n} & X_{n+2} \\ 
      };
      
      \path[->, font = \scriptsize, auto]
      (d-1-1) edge node{$f_1$} (d-1-2)
      (d-1-2) edge node{$y_2$} (d-1-3)
      (d-1-3) edge node{$y_3$} (d-1-4)
      (d-1-4) edge node{$y_{n-1}$} (d-1-5)
      (d-1-5) edge node{$y_n$} (d-1-6)
      
      (d-2-1) edge node{$f_1$} (d-2-2)
      (d-2-2) edge node{$f_2$} (d-2-3)
      (d-2-3) edge node{$f_3$} (d-2-4)
      (d-2-4) edge node{$f_{n-1}$} (d-2-5)
      (d-2-5) edge node{$f_{n+1}f_n$} (d-2-6)
      
      \downequal{d}{1-1}{2-1}
      \downequal{d}{1-2}{2-2}
      (d-1-3) edge node{$\varphi_3$} (d-2-3)
      (d-1-5) edge node{$\varphi_n$} (d-2-5)
      (d-1-6) edge node{$f_{n+1}\psi$} (d-2-6);
    \end{tikzpicture}
  \end{center}
  since $\psi y_n=f_n\varphi_n$. %For (2), 
  
  \textbf{(2)} Consider an element $\psi\in\todacc{f_{n+1},\dots,f_3,f_2}$
  represented by a diagram with vertical morphisms given by $\varphi_i$ for
  $3\leq i\leq n$. We are given an extension $Z_\bullet$ of $f_2$ and we choose
  an extension $W_\bullet$ for $f_2f_1$. Then we may apply the axiom \ref{N3} to
  complete the partial morphism $(f_1,1_{X_3})\colon W_\bullet\to Z_\bullet$,
  and thus we may build the diagram
  \begin{center}
    \begin{tikzpicture}    
      \diagram{d}{2em}{3em}{
        X_1 & & X_3 & W_3 & \cdots & W_n & \Sigma X_1\\
        & X_2 & X_3 & Z_3 & \cdots & Z_n & \Sigma X_2\\
        X_1 & X_2 & X_3 &  X_4 & \cdots & X_{n+1} & X_{n+2}\\
      };
      
      \path[->, font = \scriptsize, auto]
      (d-1-1) edge node{$f_2f_1$} (d-1-3)
      (d-1-3) edge node{$w_2$} (d-1-4)
      (d-1-4) edge node{$w_3$} (d-1-5)
      (d-1-5) edge node{$w_{n-1}$} (d-1-6)
      (d-1-6) edge node{$w_n$} (d-1-7)
      
      (d-2-2) edge node{$f_2$} (d-2-3)
      (d-2-3) edge node{$z_2$} (d-2-4)
      (d-2-4) edge node{$z_3$} (d-2-5)
      (d-2-5) edge node{$z_{n-1}$} (d-2-6)
      (d-2-6) edge node{$z_n$} (d-2-7)
      
      (d-3-1) edge node{$f_1$} (d-3-2)
      (d-3-2) edge node{$f_2$} (d-3-3)
      (d-3-3) edge node{$f_3$} (d-3-4)
      (d-3-4) edge node{$f_4$} (d-3-5)
      (d-3-5) edge node{$f_n$} (d-3-6)
      (d-3-6) edge node{$f_{n+1}$} (d-3-7)
      
      \downequal{d}{1-1}{3-1}
      (d-1-1) edge node{$f_1$} (d-2-2)
      \downequal{d}{1-3}{2-3}
      (d-1-4) edge[densely dashed] node{$\theta_3$} (d-2-4)
      (d-1-6) edge[densely dashed] node{$\theta_n$} (d-2-6)
      (d-1-7) edge node{$\Sigma f_1$} (d-2-7)
      \downequal{d}{2-2}{3-2}
      \downequal{d}{2-3}{3-3}
      (d-2-4) edge node{$\varphi_3$} (d-3-4)
      (d-2-6) edge node{$\varphi_n$} (d-3-6)
      (d-2-7) edge node{$\psi$} (d-3-7);
    \end{tikzpicture}
  \end{center}
  which, when composed vertically, shows $\psi\Sigma
  f_1\in\todacc{f_{n+1},\dots,f_3,f_2f_1}$.

  \textbf{(3)} Suppose we have extensions $V_\bullet$ of $f_2f_1$ and
  $U_\bullet$ of $f_1$. Given an element
  $\psi\in\todacc{f_{n+1},\cdots,f_3,f_2f_1}$ with vertical morphisms
  $\varphi_k$ for $3\leq k\leq n$, we may express it as an element of
  $\todacc{f_{n+1},\dots,f_3f_2,f_1}$ by applying \ref{N3} to the partial
  morphism
  \begin{center}
    \begin{tikzpicture}
      \diagram{d}{1.5em}{1.5em}{
        (1_{X_1},1_{X_2})\colon U_\bullet & V_\bullet\\
      };
      
      \path[->]
      (d-1-1) edge (d-1-2);
    \end{tikzpicture}
  \end{center}
  and composing through the following diagram vertically:
  \begin{center}
    \begin{tikzpicture}    
      \diagram{d}{2em}{3em}{
        X_1 & X_2 & & U_3 &\cdots & U_n & \Sigma X_1\\
        X_1 & X_2 & X_3 & V_3 &\cdots & V_{n} & \Sigma X_1\\
        X_1 & X_2 & X_3 & X_4 & \cdots & X_{n+1} & X_{n+2}.\\
      };
      
      \path[->, font = \scriptsize, auto]
      (d-1-1) edge node{$f_1$} (d-1-2)
      (d-1-2) edge node{$u_2$} (d-1-4)
      (d-1-4) edge node{$u_3$} (d-1-5)
      (d-1-5) edge node{$u_{n-1}$} (d-1-6)
      (d-1-6) edge node{$u_n$} (d-1-7)
      
      (d-2-1) edge node{$f_1$} (d-2-2)
      (d-2-2) edge node{$f_2$} (d-2-3)
      (d-2-3) edge node{$v_2$} (d-2-4)
      (d-2-4) edge node{$v_3$} (d-2-5)
      (d-2-5) edge node{$v_{n-1}$} (d-2-6)
      (d-2-6) edge node{$v_n$} (d-2-7)
      
      (d-3-1) edge node{$f_1$} (d-3-2)
      (d-3-2) edge node{$f_2$} (d-3-3)
      (d-3-3) edge node{$f_3$} (d-3-4)
      (d-3-4) edge node{$f_4$} (d-3-5)
      (d-3-5) edge node{$f_n$} (d-3-6)
      (d-3-6) edge node{$f_{n+1}$} (d-3-7)
      
      \downequal{d}{1-1}{2-1}
      (d-1-4) edge[densely dashed] node{$\theta_3$} (d-2-4)
      (d-1-6) edge[densely dashed] node{$\theta_n$} (d-2-6)
      \downequal{d}{1-7}{2-7}
      
      \downequal{d}{2-1}{3-1}
      \downequal{d}{2-3}{3-3}
      (d-2-4) edge node{$\varphi_3$} (d-3-4)
      (d-2-6) edge node{$\varphi_{n}$} (d-3-6)
      (d-2-7) edge node{$\psi$} (d-3-7);
      
      \clip (d-2-3)[xshift=-20pt, yshift=-2pt] circle (37pt);
      \draw[line width=2pt, color=white] (d-2-1) circle (61pt);
      \draw (d-2-1) circle (60pt);
      \draw (d-2-1) circle (62pt);
    \end{tikzpicture}
  \end{center}
    
  \textbf{(4)} Let $3\leq i\leq n$ and consider the commutative diagram
  \begin{center}
    \begin{tikzpicture}
      \diagram{d}{2.5em}{2em}{
        X_1 & X_2 &  \cdots & Y_{i-1} &  Y_i &  & Y_{i+1} & \cdots & \Sigma X_{1}\\
        X_1 & X_2 & \cdots & X_{i-1} & X_i & X_{i+1} & X_{i+2} & \cdots & X_{n+1}.\\
      };
      
      \path[->, auto, font = \scriptsize]
      (d-1-1) edge node{$f_1$} (d-1-2)
      (d-1-2) edge node{$y_2$} (d-1-3)
      (d-1-3) edge node{$y_{i-2}$} (d-1-4)
      (d-1-4) edge node{$y_{i-1}$} (d-1-5)
      (d-1-5) edge node{$y_{i}$} (d-1-7)
      (d-1-7) edge node{$y_{i+1}$} (d-1-8)
      (d-1-8) edge node{$y_{n}$} (d-1-9)
      
      (d-2-1) edge node{$f_1$} (d-2-2)
      (d-2-2) edge node{$f_2$} (d-2-3)
      (d-2-3) edge node{$f_{i-2}$} (d-2-4)
      (d-2-4) edge node{$f_{i-1}$} (d-2-5)
      (d-2-5) edge node{$f_{i}$} (d-2-6)
      (d-2-6) edge node{$f_{i+1}$} (d-2-7)
      (d-2-7) edge node{$f_{i+2}$} (d-2-8)
      (d-2-8) edge node{$f_{n}$} (d-2-9)
      
      \downequal{d}{1-1}{2-1}
      \downequal{d}{1-2}{2-2}
      (d-1-4) edge node{$\varphi_{i-1}$} (d-2-4)
      (d-1-5) edge[densely dashed] node{$\varphi_{i}$} (d-2-5)
      (d-1-7) edge node{$\varphi_{i+1}$} (d-2-7)
      (d-1-9) edge node{$\psi$} (d-2-9)
      (d-1-5) edge[densely dotted] node{$f_i\varphi_i$} (d-2-6);
    \end{tikzpicture}
  \end{center}
  The diagram without the dotted morphism represents an element 
  \begin{equation*}
    \psi\in\todacc{f_{n+1},\dots,f_{i+1}f_i,f_{i-1},\dots,f_2,f_1}
  \end{equation*}
  while the diagram without the dashed morphism represents the same morphism but
  as an element of $\todacc{f_{n+1},\dots,f_{i+1},f_if_{i-1},\dots,f_2,f_1}$ and
  the claim follows.
\end{proof}

\begin{proposition}[Submultiplicativity]\label{prop:juggling_fc}
  Let $\C$ be a pre-$n$-angulated category and consider the diagram
  \begin{center}
    \begin{tikzpicture}
      \diagram{d}{2em}{2em}{
        X_1 & X_2 & X_3 & \cdots & X_n & X_{n+1} & X_{n+2}.\\
      };
      
      \path[->, auto, font = \scriptsize]
      (d-1-1) edge node{$f_1$} (d-1-2)
      (d-1-2) edge node{$f_2$} (d-1-3)
      (d-1-3) edge node{$f_3$} (d-1-4)
      (d-1-4) edge node{$f_{n-1}$} (d-1-5)
      (d-1-5) edge node{$f_n$} (d-1-6)
      (d-1-6) edge node{$f_{n+1}$} (d-1-7);
    \end{tikzpicture}
  \end{center}
  such that $f_{i+1}f_i=0$ for $1\leq i\leq n$. Then the following statements
  hold:
  \begin{enumerate}[label=(\arabic*), ref={Proposition
      \theproposition.(\arabic*)}]
  \item $f_{n+1}\toda{f_n,\dots,f_2,f_1}\subseteq
    \toda{f_{n+1}f_n,\dots,f_2,f_1}$.
  \item $\toda{f_{n+1},\dots,f_3,f_2}\Sigma f_1\subseteq
    \toda{f_{n+1},\dots,f_3,f_2f_1}$.
  \item $\toda{f_{n+1},\dots,f_3,f_2f_1}\subseteq
    \toda{f_{n+1},\dots,f_3f_2,f_1}$.
  \item $\toda{f_{n+1}f_n,f_{n-1},\dots,f_2,f_1}\subseteq
    \toda{f_{n+1},f_nf_{n-1},\dots,f_2,f_1}$.
  \item \label{prop-6-6-5} For all $3\leq i\leq n-1$ we have the identities
    \begin{equation*}
      \toda{f_{n+1},\dots,f_{i+1}f_i,f_{i-1},\dots,f_2,f_1} =
      \toda{f_{n+1},\dots,f_{i+1},f_if_{i-1},\dots,f_2,f_1}.
    \end{equation*}
  \item
    $f_{n+1}\toda{f_{n},\dots,f_2,f_1}=\toda{f_{n+1},\dots,f_3,f_2}(-1)^n\Sigma
    f_1.$
  \end{enumerate}
  Moreover, if the condition of $f_{i + 1} f_i = 0$ is not satisfied, the
  statements (1)--(5) all hold for the fiber-cofiber Toda bracket. Statement (6)
  holds for the fiber-cofiber Toda bracket if the two brackets in question are
  non-empty.
\end{proposition}

\begin{proof}
  We will argue that the statements hold for the fiber-cofiber Toda bracket
  regardless of the condition $f_{i+1}f_i=0$, since this will imply the
  statements for the Toda bracket when the condition is satisfied.

  \textbf{(1)} and \textbf{(2)} were proven in \autoref{prop:submult} for all
  three Toda brackets. We include them here for the sake of
  completion.
  % \smallskip

  In the case of the fiber-cofiber Toda bracket, \textbf{(3)} and \textbf{(4)}
  are a consequence of \ref{prop-6-5-3} and \ref{prop-6-5-3-prime}: If either
  $f_{i+1}f_i\neq 0$ for some $i$ or $f_3f_2f_1\neq 0$ (resp.\
  $f_{n+1}f_nf_{n-1}\neq 0$) then both sides of the inclusion are empty and the
  statement holds. If not, then the Toda brackets coincide and we may appeal to
  \autoref{prop:submult}.\sloppy

  \textbf{(5)} follows from \autoref{thm:main}, \ref{prop-6-5-4} and
  \ref{prop-6-5-4-prime}. Again, if either $f_{j+1}f_j\neq 0$ for some $j$ or if
  $f_{i+1}f_if_{i-1}\neq 0$ then both sides are empty and the equality holds.

  To see that \textbf{(6)} holds, first note that if both the brackets in
  question are empty, the equality is trivial. If they are both non-empty,
  consider the commutative diagram in \autoref{fig:prop-6-6-(6)}.
  
  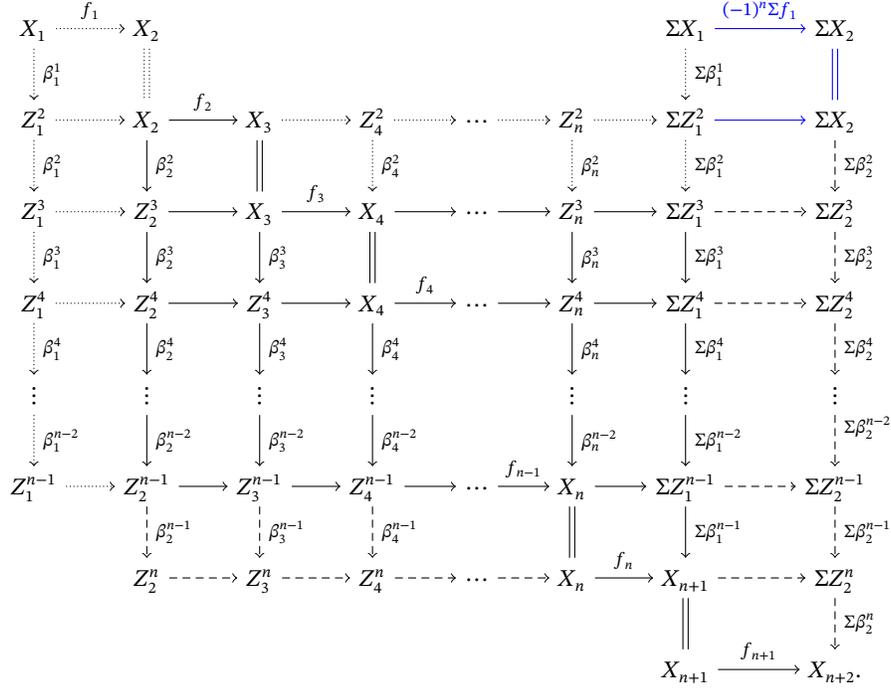
\begin{figure}[htbp]
    \centering
    \begin{tikzpicture}
      \smalldiagram{d}{2em}{2em}{
        X_1 &[-0.1em] X_2 &[-0.1em] &[-0.1em] & & &\Sigma X_1 &[0.75em] \Sigma X_2\\
        Z^{2}_{1} & X_2 & X_3 & Z^{2}_4 & \cdots & Z^{2}_{n} & \Sigma Z^{2}_{1} & \Sigma
        X_2\\
        Z^{3}_{1} & Z^{3}_{2} & X_3 & X_4 & \cdots & Z^{3}_{n} & \Sigma Z^{3}_{1} & \Sigma
        Z^3_{2}\\
        Z^{4}_{1} & Z^{4}_{2} & Z^{4}_{3} & X_4 & \cdots & Z^{4}_{n} & \Sigma Z^{4}_{1} &
        \Sigma Z^4_{2}\\
        \vdots & \vdots & \vdots & \vdots & & \vdots & \vdots & \vdots \\
        Z^{{n-1}}_1 & Z^{{n-1}}_2 & Z^{{n-1}}_3 & Z^{{n-1}}_4 & \cdots & X_n & \Sigma
        Z^{{n-1}}_1 & \Sigma Z^{n-1}_2\\
        & Z^n_2 & Z^{{n}}_3 & Z^{{n}}_4 & \cdots & X_n & X_{n + 1} & \Sigma Z^n_2\\
        &&&&&& X_{n+1} & X_{n+2}.\\
      };
            
      \path[->, auto, font = \scriptsize]
      (d-1-1) edge[densely dotted] node{$f_1$} (d-1-2)
              edge[densely dotted] node{$\beta_{1}^1$} (d-2-1)
      ([xshift=-0.1em] d-1-2.south) edge[-,densely dotted] ([xshift=-0.1em]d-2-2.north)
      ([xshift=0.1em] d-1-2.south) edge[-,densely dotted] ([xshift=0.1em]d-2-2.north)
      (d-1-7) edge[blue] node{$(-1)^n\Sigma f_1$} (d-1-8)
              edge[densely dotted] node{$\Sigma\beta_{1}^1$} (d-2-7)
      ([xshift=-0.1em] d-1-8.south) edge[-,blue] ([xshift=-0.1em]d-2-8.north)
      ([xshift=0.1em] d-1-8.south) edge[-,blue] ([xshift=0.1em]d-2-8.north)
      
      (d-2-1) edge[densely dotted] (d-2-2)
              edge[densely dotted] node{$\beta_{1}^2$} (d-3-1)
      (d-2-2) edge node{$f_2$} (d-2-3)
              edge node{$\beta_{2}^2$} (d-3-2)
      (d-2-3) edge[densely dotted] (d-2-4)
      \downequal{d}{2-3}{3-3}
      (d-2-4) edge[densely dotted] (d-2-5)
              edge[densely dotted] node{$\beta_4^2$} (d-3-4)
      (d-2-5) edge[densely dotted] (d-2-6)
      (d-2-6) edge[densely dotted] (d-2-7)
              edge[densely dotted] node{$\beta_{n}^2$} (d-3-6)
      (d-2-7) edge[blue] (d-2-8)
              edge[densely dotted] node{$\Sigma \beta_{1}^2$} (d-3-7)
      (d-2-8) edge[densely dashed] node{$\Sigma\beta^2_{2}$} (d-3-8)
      
      (d-3-1) edge[densely dotted] (d-3-2)
              edge[densely dotted] node{$\beta_{1}^3$} (d-4-1)
      (d-3-2) edge (d-3-3)
              edge node{$\beta_{2}^3$} (d-4-2)
      (d-3-3) edge node{$f_3$} (d-3-4)
              edge node{$\beta_{3}^3$} (d-4-3)
      (d-3-4) edge (d-3-5)
      \downequal{d}{3-4}{4-4}
      (d-3-5) edge (d-3-6)
      (d-3-6) edge (d-3-7)
              edge node{$\beta_{n}^3$} (d-4-6)
      (d-3-7) edge node{$\Sigma\beta_{1}^3$} (d-4-7)
      (d-3-7) edge[densely dashed] node{} (d-3-8)
      (d-3-8) edge[densely dashed] node{$\Sigma\beta^3_{2}$} (d-4-8)
      
      (d-4-1) edge[densely dotted] (d-4-2)
              edge[densely dotted] node{$\beta_{1}^4$} (d-5-1)
      (d-4-2) edge (d-4-3)
              edge node{$\beta_{2}^4$} (d-5-2)
      (d-4-3) edge (d-4-4)
              edge node{$\beta_{3}^4$} (d-5-3)
      (d-4-4) edge node{$f_4$} (d-4-5)
              edge node{$\beta_{4}^4$} (d-5-4)
      (d-4-5) edge (d-4-6)
      (d-4-6) edge (d-4-7)
              edge node{$\beta_{n}^4$} (d-5-6)
      (d-4-7) edge node{$\Sigma\beta_{1}^4$} (d-5-7)
              edge[densely dashed] (d-4-8)
      (d-4-8) edge[densely dashed] node{$\Sigma\beta^4_{2}$} (d-5-8)
      
      (d-5-1) edge[densely dotted] node{$\beta_{1}^{n-2}$} (d-6-1)
      (d-5-2) edge node{$\beta_{2}^{n-2}$} (d-6-2)
      (d-5-3) edge node{$\beta_{3}^{n-2}$} (d-6-3)
      (d-5-4) edge node{$\beta_{4}^{n-2}$} (d-6-4)
      (d-5-6) edge node{$\beta_{n}^{n-2}$} (d-6-6)
      (d-5-7) edge node{$\Sigma\beta_{1}^{n-2}$} (d-6-7)
      (d-5-8) edge[densely dashed] node[near start]{$\Sigma\beta_{2}^{n-2}$} (d-6-8)
      
      (d-6-1) edge[densely dotted] (d-6-2)
      (d-6-2) edge (d-6-3)
              edge[densely dashed] node{$\beta^{n-1}_2$} (d-7-2)
      (d-6-3) edge (d-6-4)
              edge[densely dashed] node{$\beta^{n-1}_3$} (d-7-3)
      (d-6-4) edge (d-6-5)
              edge[densely dashed] node{$\beta^{n-1}_4$} (d-7-4)
      (d-6-5) edge node{$f_{n-1}$} (d-6-6)
      (d-6-6) edge (d-6-7)
      \downequal{d}{6-6}{7-6}
      (d-6-7) edge[densely dashed] (d-6-8)
              edge node{$\Sigma\beta_1^{n-1}$} (d-7-7)
      (d-6-8) edge[densely dashed] node{$\Sigma\beta^{n-1}_2$} (d-7-8)
      
      (d-7-2) edge[densely dashed] (d-7-3)
      (d-7-3) edge[densely dashed] (d-7-4)
      (d-7-4) edge[densely dashed] (d-7-5)
      (d-7-5) edge[densely dashed] (d-7-6)
      (d-7-6) edge node{$f_{n}$} (d-7-7)
      (d-7-7) edge[densely dashed] (d-7-8)
      \downequal{d}{7-7}{8-7}
      (d-7-8) edge[densely dashed] node{$\Sigma\beta_2^{n}$} (d-8-8)
      
      (d-8-7) edge node{$f_{n+1}$} (d-8-8);
    \end{tikzpicture}
    \caption{$f_{n+1}\todafc{f_{n},\dots,f_2,f_1}=\todafc{f_{n+1},\dots,f_3,f_2}(-1)^n
      \Sigma f_1$}
    \label{fig:prop-6-6-(6)}
  \end{figure}
    
  Suppose we have an element $\psi\in\todafc{f_n,\dots,f_2,f_1}$ given by the
  solid part and the dotted part of the diagram in \autoref{fig:prop-6-6-(6)};
  i.e.,
  \begin{equation*}
    \psi=\Sigma(\beta_1^{n-1}\cdots \beta^{2}_{1}\beta^1_{1}).
  \end{equation*}
  Then we see that the solid part of the diagram together with the dashed part
  make up a diagram representing an element
  $\psi'\in\todafc{f_{n+1},\dots,f_3,f_2}$, i.e.,
  \begin{equation*}
    \psi'=\Sigma(\beta^n_2\cdots\beta^3_{2}\beta^2_{2})
  \end{equation*} 
  and, including the blue part of the diagram in the upper right-hand corner, we
  may read off of the diagram that
  \begin{equation*}
    f_{n+1}\psi=\psi'(-1)^{n+1}\Sigma f_1\in\todafc{f_{n+1},\dots,f_3,f_2}(-1)^n\Sigma
    f_1.
  \end{equation*}
  Similarly, given an element represented by the solid and dashed parts of the
  diagram in \autoref{fig:prop-6-6-(6)}, i.e., an element
  $\psi'\in\todafc{f_{n+1},\dots,f_3,f_2}$, we may expand the diagram to include
  the dotted part such that the membership relation
  \begin{equation*}
    \psi'(-1)^n\Sigma f_1\in f_{n+1}\todafc{f_n,\dots,f_2,f_1}
  \end{equation*}
  is exhibited by the diagram.
\end{proof}

%
% Heller's theorem
%
\section{Toda brackets determine the \texorpdfstring{$n$}{n}-angulation}
\label{sec:toda_determines_nang}

In this section we generalize a classical fact about $3$-fold Toda brackets in
triangulated categories due to Heller \cite[Theorem~13.2]{H} to our setting; see
also \cite[Theorem~B.1]{CF} and the references therein.

\begin{definition}
  \label{def:Yoneda}
  Let $\C$ be a pre-$n$-angulated category. An $n$-$\Sigma$-sequence
  \begin{center}
    \begin{tikzpicture}
      \diagram{d}{2em}{2em}{
        X_1 & X_2 & X_3 & \cdots & X_n & \Sigma X_{1}\\
      };
      
      \path[->, font = \scriptsize, auto]
      (d-1-1) edge node{$f_1$} (d-1-2)
      (d-1-2) edge node{$f_2$} (d-1-3)
      (d-1-3) edge node{$f_3$} (d-1-4)
      (d-1-4) edge node{$f_{n - 1}$} (d-1-5)
      (d-1-5) edge node{$f_n$} (d-1-6);
    \end{tikzpicture}
  \end{center}
  in $\C$ is \Def{Yoneda exact} if the sequence of abelian groups
  \begin{center}
    \begin{tikzpicture}
      \diagram{d}{2em}{3em}{
        \C(A, X_1) & \C(A, X_2) & \cdots & \C(A, \Sigma X_1) & \C(A, \Sigma X_2) \\
      };
      
      \path[->, font = \scriptsize, auto]
      (d-1-1) edge node{$(f_1)_\ast$} (d-1-2)
      (d-1-2) edge node{$(f_2)_\ast$} (d-1-3)
      (d-1-3) edge node{$(f_n)_\ast$} (d-1-4)
      (d-1-4) edge node{$(\Sigma f_1)_\ast$} (d-1-5);
    \end{tikzpicture}
  \end{center}
  is exact for every object $A$ of $\C$, or equivalently, the long sequence of
  abelian groups
  \begin{center}
    \begin{tikzpicture}
      \smalldiagram{d}{2em}{3em}{
        \cdots & \C(A, \Sigma^k X_1) & \C(A, \Sigma^k X_2) & \cdots & \C(A,
        \Sigma^{k + 1} X_1) & \cdots\\
      };

      \path[->, font = \scriptsize, auto]
      (d-1-1) edge (d-1-2)
      (d-1-2) edge node{$(\Sigma^k f_1)_\ast$} (d-1-3)
      (d-1-3) edge node{$(\Sigma^k f_2)_\ast$} (d-1-4)
      (d-1-4) edge node{$(\Sigma^k f_n)_\ast$} (d-1-5)
      (d-1-5) edge (d-1-6);
    \end{tikzpicture}
  \end{center}
  is exact for every object $A$ of $\C$.
\end{definition}

Note that what we call \Def{Yoneda exact} was called \Def{exact} in
\cite[Definition~2.1]{GKO}. We now state and prove the generalized version of
Heller's theorem.

\begin{proposition}\label{prop:NAngleCriterion}
  Let $\C$ be a pre-$n$-angulated category. An $n$-$\Sigma$-sequence
  \begin{center}
    \begin{tikzpicture}
      \diagram{d}{2em}{2em}{
        X_1 & X_2 & X_3 & \cdots & X_n & \Sigma X_{1}\\
      };
      
      \path[->, font = \scriptsize, auto]
      (d-1-1) edge node{$f_1$} (d-1-2)
      (d-1-2) edge node{$f_2$} (d-1-3)
      (d-1-3) edge node{$f_3$} (d-1-4)
      (d-1-4) edge node{$f_{n - 1}$} (d-1-5)
      (d-1-5) edge node{$f_n$} (d-1-6);
    \end{tikzpicture}
  \end{center}
  in $\C$ is an $n$-angle if and only if the following two conditions hold.
  \begin{enumerate}
  \item The $n$-$\Sigma$-sequence $X_{\bullet}$ is Yoneda exact. 
  \item The Toda bracket $\todacc{f_n, \ldots, f_2, f_1} \subseteq \C(\Sigma X_1,
    \Sigma X_1)$ contains the identity morphism $1_{\Sigma X_1}$.
  \end{enumerate}
\end{proposition}

\begin{proof}
  ($\implies$) An $n$-angle is Yoneda exact \cite[Proposition~2.5]{GKO}. For the
  second condition, consider the commutative diagram
  \begin{center}
    \begin{tikzpicture}    
      \diagram{d}{2em}{3em}{
        X_1 & X_2 & X_3 &\cdots & X_{n} & \Sigma X_1\\ 
        X_1 & X_2 & X_3 &\cdots & X_{n} & \Sigma X_1.\\ 
      };
      
      \path[->, font = \scriptsize, auto]
      (d-1-1) edge node{$f_1$} (d-1-2)
      (d-1-2) edge node{$f_2$} (d-1-3)
      (d-1-3) edge node{$f_3$} (d-1-4)
      (d-1-4) edge node{$f_{n-1}$} (d-1-5)
      (d-1-5) edge node{$f_n$} (d-1-6)
      
      (d-2-1) edge node{$f_1$} (d-2-2)
      (d-2-2) edge node{$f_2$} (d-2-3)
      (d-2-3) edge node{$f_3$} (d-2-4)
      (d-2-4) edge node{$f_{n-1}$} (d-2-5)
      (d-2-5) edge node{$f_n$} (d-2-6)
      
      \downequal{d}{1-1}{2-1}
      \downequal{d}{1-2}{2-2}
      (d-1-3) edge node{$1_{X_3}$} (d-2-3)
      (d-1-5) edge node{$1_{X_n}$} (d-2-5)
      (d-1-6) edge node{$1_{\Sigma X_1}$} (d-2-6);
    \end{tikzpicture}
  \end{center}
  Since the top row is an $n$-angle by assumption, this diagram exhibits the
  membership $1_{\Sigma X_1} \in \todacc{f_n, \ldots, f_2, f_1}$.
  
  ($\impliedby$) Assume $1_{\Sigma X_1} \in \todacc{f_n, \ldots, f_2, f_1}$ and
  that the $n$-$\Sigma$-sequence is Yoneda exact. Then we have some $n$-angle
  extension, $Y_\bullet$, of $f_1$ and a diagram
  \begin{center}
    \begin{tikzpicture}    
      \diagram{d}{2em}{3em}{
        X_1 & X_2 & Y_3 &\cdots & Y_{n} & \Sigma X_1\\ 
        X_1 & X_2 & X_3 &\cdots & X_{n} & \Sigma X_1\\ 
      };
      
      \path[->, font = \scriptsize, auto]
      (d-1-1) edge node{$f_1$} (d-1-2)
      (d-1-2) edge node{$y_2$} (d-1-3)
      (d-1-3) edge node{$y_3$} (d-1-4)
      (d-1-4) edge node{$y_{n-1}$} (d-1-5)
      (d-1-5) edge node{$y_n$} (d-1-6)
      
      (d-2-1) edge node{$f_1$} (d-2-2)
      (d-2-2) edge node{$f_2$} (d-2-3)
      (d-2-3) edge node{$f_3$} (d-2-4)
      (d-2-4) edge node{$f_{n-1}$} (d-2-5)
      (d-2-5) edge node{$f_n$} (d-2-6)
      
      \downequal{d}{1-1}{2-1}
      \downequal{d}{1-2}{2-2}
      (d-1-3) edge node{$\varphi_3$} (d-2-3)
      (d-1-5) edge node{$\varphi_n$} (d-2-5)
      (d-1-6) edge node{$1_{\Sigma X_1}$} (d-2-6);
    \end{tikzpicture}
  \end{center}
  which is a weak isomorphism between $n$-$\Sigma$-sequences. Since the top row is
  an $n$-angle, by \cite[Lemma~2.4]{GKO}, so is the bottom row.
\end{proof}

%
% Examples
%
\section{Examples}
\label{sec:examples}

In this section, we compute examples of Toda brackets in $4$-angulated
categories. We first look at an exotic example, then examples from quiver
representation theory.

\subsection*{Exotic \texorpdfstring{$n$}{n}-angulated categories}

As a warm-up example, take the exotic $4$-angulated category $\C = \mod^{\ff}
\Z/p^2$ of finitely generated free modules over $R = \Z/p^2$ constructed in
\cite[Theorem~3.7]{BerghJT16}. The automorphism is the identity functor $\Sigma =
\Id$. Here $p \in \Z$ can be any prime number, since $n=4$ is even.

\begin{example}\label{ex:Exotic4angle}
  Let us compute the Toda bracket of the diagram
  \begin{center}
    \begin{tikzpicture}
      \diagram{d}{2em}{2em}{
        R & R & R & R & R\\
      };
      
      \path[->, auto, font = \scriptsize]
      (d-1-1) edge node{$p$} (d-1-2)
      (d-1-2) edge node{$p$} (d-1-3)
      (d-1-3) edge node{$p$} (d-1-4)
      (d-1-4) edge node{$p$} (d-1-5);
    \end{tikzpicture}
  \end{center}
  as a subset of $\C(\Sigma R, R) = \C(R,R) \cong \Z/p^2$. First we compute the
  iterated cofiber Toda bracket. By \autoref{lem:AnyExtension}, we may choose
  any $4$-angle extension of $f_1 = p \colon R \to R$ as top row in the
  following diagram, so we choose the easiest one:
  \begin{equation*}
    \begin{aligned}
      \begin{tikzpicture}
        \diagram{d}{2em}{2em}{
          R & R & R & R & R\\
          R & R & R & R & R.\\
        };
        
        \path[->, auto, font = \scriptsize]
        (d-1-1) edge node{$p$} (d-1-2)
        \downequal{d}{1-1}{2-1}
        (d-1-2) edge node{$p$} (d-1-3)
        \downequal{d}{1-2}{2-2}
        (d-1-3) edge node{$p$} (d-1-4)
                edge[densely dashed] node{$\phi_3$} (d-2-3)
        (d-1-4) edge[densely dashed]  node{$\phi_4$} (d-2-4)
                edge node{$p$} (d-1-5)
        (d-1-5) edge[densely dashed]  node{$\psi$} (d-2-5)
        
        (d-2-1) edge node{$p$} (d-2-2)
        (d-2-2) edge node{$p$} (d-2-3)
        (d-2-3) edge node{$p$} (d-2-4)
        (d-2-4) edge node{$p$} (d-2-5);
      \end{tikzpicture}
    \end{aligned}
  \end{equation*}
  The possible choices of extensions are $\phi_3 \equiv 1 \mod p$, \, $\phi_4
  \equiv 1 \mod p$, \, and $\psi \equiv 1 \mod p$, which yields
  \begin{equation*}
    \todacc{p,p,p,p} = \{ 1 + cp \mid c \in \Z \} = 1 + (p) \subset \Z/p^2.
  \end{equation*}
  Now let us compute the fiber-cofiber Toda bracket $\todafc{p,p,p,p}$ directly,
  checking \autoref{thm:main} in this case. Any $4$-angle extension of $p \colon
  R \to R$ is isomorphic to
  \begin{center}
    \begin{tikzpicture}
      \diagram{d}{2em}{2em}{
        R & R & R \op M & R \op M & R\\
      };
      
      \path[->, auto, font = \scriptsize]
      (d-1-1) edge node{$p$} (d-1-2)
      (d-1-2) edge node{$\smat{p \\ 0}$} (d-1-3)
      (d-1-3) edge node{$\smat{p & 0 \\ 0 & 1}$} (d-1-4)
      (d-1-4) edge node{$\smat{p & 0}$} (d-1-5);
    \end{tikzpicture}
  \end{center}
  for some finitely generated free $R$-module $M$, by
  \cite[Remark~3.1(5)]{BerghJT16}. Consider all possible dashed arrows in a
  diagram
  \begin{center}
    \begin{tikzpicture}    
      \diagram{d}{2em}{2.75em}{
        R & R & & & R & & \\
        R \op M & R & R & R \op M & R \op M & R & R \\
        & & R & R & R \op N & R \op N & R \\
        & & & R & R. \\
      };
      
      \path[->, font = \scriptsize, auto]
      (d-1-1) edge node{$f_1 = p$} (d-1-2)
      
      (d-2-1) edge node{$\smat{p & 0}$} (d-2-2)
      (d-2-2) edge node{$f_2=p$} (d-2-3)
      (d-2-3) edge node{$\smat{p \\ 0}$} (d-2-4)
      (d-2-4) edge node{$\smat{p & 0 \\ 0 & 1}$} (d-2-5)
      (d-2-5) edge node{$\smat{p & 0}$} (d-2-6)
      (d-2-6) edge node{$\Sigma f_2 = p$} (d-2-7)
      
      (d-3-3) edge node{$f_3=p$} (d-3-4)
      (d-3-4) edge node{$\smat{p \\ 0}$} (d-3-5)
      (d-3-5) edge node{$\smat{p & 0 \\ 0 & 1}$} (d-3-6)
      (d-3-6) edge node{$\smat{p & 0}$} (d-3-7)
      
      (d-4-4) edge node{$f_4 = p$} (d-4-5)
      
      (d-1-1) edge[densely dashed] node{$\Sigma^{-1}\beta^1_3$} (d-2-1)
      \downequal{d}{1-2}{2-2}
      (d-1-5) edge[densely dashed] node{$\beta^1_3$} (d-2-5)
      
      \downequal{d}{2-3}{3-3}   
      (d-2-4) edge[densely dashed] node{$\beta^2_1$} (d-3-4)
      (d-2-5) edge[densely dashed] node{$\beta^2_2$} (d-3-5)
      (d-2-6) edge[densely dashed] node{$\beta^2_3$} (d-3-6)
      \downequal{d}{2-7}{3-7}
      
      \downequal{d}{3-4}{4-4}
      (d-3-5) edge[densely dashed] node{$\beta^3_1$} (d-4-5);
    \end{tikzpicture}
  \end{center}
  A straightforward calculation yields that the possible composites $\beta^3_1
  \beta^2_2 \beta^1_3$ are
  \begin{equation*}
    \todafc{p,p,p,p} = \{ 1 + cp \mid c \in \Z \} = 1 + (p) \subset \Z/p^2.
  \end{equation*}
  In light of \autoref{lem:fc_in_cc}, the more delicate point was the inclusion
  \begin{equation*}
    \todacc{p,p,p,p} \subseteq \todafc{p,p,p,p}.
  \end{equation*}
  Given any element $\psi \in \todacc{p,p,p,p} = 1 +(p)$, this is achieved with
  the choices $M=N=0$, \, $\beta^2_2 = \beta^3_1 = 1$, and $\beta^1_3 = \psi$.
\end{example}

\begin{remark}
  The Toda bracket computation $\todacc{p,p,p,p}$ agrees with
  \autoref{prop:NAngleCriterion}. Since the morphisms $f_i=p$ form a $4$-angle,
  the bracket must contain $1_{\Sigma X_1} = 1 \in \Z/p^2$. Moreover, by
  \autoref{prop:brackets_are_cosets}, it must be a coset of the indeterminacy
  subgroup
  \begin{equation*}
    (f_4)_*\C(\Sigma X_1,X_{4}) + (\Sigma f_1)^*\C(\Sigma X_2,X_{5}) = (p)_* \C(R,R) +
    (p)^* \C(R,R) = (p) \subset \Z/p^2
  \end{equation*}
  and thus $\todacc{p,p,p,p} = 1 + (p)$.
\end{remark}

\begin{example}
  Assume that $p$ is odd and let us compute the Toda bracket of the diagram
  \begin{equation}\label{eq:Minus4angle}
    \begin{aligned}
      \begin{tikzpicture}
        \diagram{d}{2em}{2em}{
          R & R & R & R & R.\\
        };
        
        \path[->, auto, font = \scriptsize]
        (d-1-1) edge node{$p$} (d-1-2)
        (d-1-2) edge node{$p$} (d-1-3)
        (d-1-3) edge node{$p$} (d-1-4)
        (d-1-4) edge node{$-p$} (d-1-5);
      \end{tikzpicture}
    \end{aligned}
  \end{equation}
  Using \ref{item:Minus}, we obtain
  \begin{equation*}
    \todacc{-p,p,p,p} = -\todacc{p,p,p,p} = \{ -1 + cp \mid c \in \Z \} = -1 +
    (p) \subset \Z/p^2.
  \end{equation*}
  In particular, the $4$-$\Sigma$-sequence~\eqref{eq:Minus4angle} is not a
  $4$-angle, by \autoref{prop:NAngleCriterion}.
\end{example}

%
% Quiver representation theory
%
\subsection*{Quiver representation theory}
The prototypical example of an $n$-angulated category is that of
\cite[Theorem~1]{GKO}: if $\T$ is a triangulated category with an $(n -
2)$-cluster tilting subcategory $\C$ closed under $\Sigma^{n - 2}$ where
$\Sigma$ denotes the suspension in $\T$, then $(\C, \Sigma^{n - 2}, \nang)$ is
an $n$-angulated category where $\nang$ is the class of all $n$-$\Sigma^{n -
  2}$-sequences
\begin{center}
  \begin{tikzpicture}
    \diagram{d}{2em}{2em}{
      X_1 & X_2 & \cdots & X_n & \Sigma^{n - 2} X_1\\
    };
    
    \path[->, font = \scriptsize, auto]
    (d-1-1) edge node{$f_1$} (d-1-2)
    (d-1-2) edge node{$f_2$} (d-1-3)
    (d-1-3) edge node{$f_{n - 1}$} (d-1-4)
    (d-1-4) edge node{$f_n$} (d-1-5);
  \end{tikzpicture}
\end{center}
in $\C$ such that there exists a diagram in $\T$
\begin{center}
  \begin{tikzpicture}
    \diagram{d}{2em}{1.5em}{
      & X_2 && X_3 && \cdots &&& X_{n - 1} &\\
      X_1 && X_{2.5} && X_{3.5} & \cdots && X_{n - 1.5} && X_n\\
    };
    
    \path[->, font = \scriptsize, auto]
    (d-1-2) edge node{$f_2$} (d-1-4)
            edge (d-2-3)
    (d-1-4) edge node{$f_3$} (d-1-6)
            edge (d-2-5)
    (d-1-6) edge node{$f_{n - 2}$} (d-1-9)
    (d-1-9) edge node{$f_{n - 1}$} (d-2-10)
    
    (d-2-1) edge node{$f_1$} (d-1-2)
    (d-2-3) edge (d-1-4)
            edge[suspension,->] (d-2-1)
    (d-2-5) edge[suspension,->] (d-2-3)
    (d-2-6) edge[suspension,->] (d-2-5)
    (d-2-8) edge (d-1-9)
           edge[suspension,->] (d-2-6)
    (d-2-10) edge[suspension,->] (d-2-8);
  \end{tikzpicture}
\end{center}
with $X_i\in \T$ for all $i\notin \Z$ such that all triangles with base a
degree-shifting morphism are triangles in $\T$, and $f_n$ is the composition of
the bottom row.

\begin{remark}\label{rem:U_description}
  Given a finite dimensional algebra $\Lambda$ of global dimension at most
  \mbox{$n-2$}, we may apply \cite[Lemma~2.13 and Theorem~2.16]{IO2} to generate
  an $(n-2)$-cluster tilting subcategory $\U$ of the bounded derived category of
  finitely generated $\Lambda$-modules, $\derived^b(\mod\Lambda)$, which has the
  $n$-angulated structure described above following \cite[Example~6.1]{GKO}.
  Moreover, if $\Lambda$ has an $(n - 2)$-cluster tilting module $M$,
  \cite[Theorem~2.21]{IO2} provides a nice description of $\U$:
  \begin{equation*}\label{eq:U_description}
    \U = \add\{M[(n - 2)i]\mid i\in\Z\}
  \end{equation*}
  where $M[(n - 2)i]$ means the chain complex consisting of $M$ concentrated in
  degree $(n - 2)i$. By abuse of notation, we will write $M=M[0]$ when working
  with chain complexes. The next two examples utilize this construction with
  respect to $\Lambda$ a path algebra of a quiver with relations. For background
  on representations of quivers with relations, also known as bound quivers, see
  \cite[§II.2]{AssemSS06}.
\end{remark}

%
% Warm-up example (1 --> 2 --> 3 --> 4 divided out by J^3)
%
\begin{example}\label{ex:Q1}
  Let $Q_1$ denote the quiver with vertices and arrows as depicted in the solid
  part of \autoref{fig:Gamma1}. Note that the underlying undirected graph of
  $Q_1$ is the Dynkin $A_4$ diagram. From $Q_1$ we can form the path algebra
  $\field Q_1$ where $\field$ is a field. Let $J$ denote the $2$-sided ideal in
  $\field Q_1$ generated by all paths of length $1$, i.e., all arrows. The ideal
  $J$ is often referred to as the \Def{arrow ideal}. Next, we consider the
  quotient $\Gamma_1 = \field Q_1/J^3$ where $J^3$ denotes the $2$-sided ideal
  generated by all paths of length $3$, or $cba$, in our case. We can visualize
  this as shown in \autoref{fig:Gamma1}.
    
  \begin{figure}[htbp]
    \centering
    \begin{tikzpicture}
      \diagram{d}{2em}{2em}{
        1 & 2 & 3 & 4\\
      };
      
      \path[->, font = \scriptsize, auto]
      (d-1-1) edge node{$a$} (d-1-2)
              edge[out=320,in=220,densely dashed,->] (d-1-4)
      (d-1-2) edge node{$b$} (d-1-3)
      (d-1-3) edge node{$c$} (d-1-4);
    \end{tikzpicture}
    \caption{The quiver $Q_1$ with relation $cba = 0$ yielding $\Gamma_1 =
      \field Q_1/J^3$.}
    \label{fig:Gamma1}
  \end{figure}
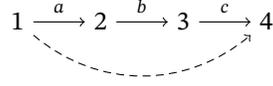

  Next, we denote by $\A_1 \coloneq \mod \Gamma_1$ the category of all finitely
  generated right $\Gamma_1$-modules. We can visualize $\A_1$ as shown in
  \autoref{fig:modGamma1}. If $M_1$ denotes the direct sum of the encircled
  modules, then $M_1$ is a $2$-cluster tilting module in the sense of
  \cite{Iyama}:
  \begin{align*}
    \add M_1 &= \{ X\in \A_1 \mid \Ext_{\Gamma_1}^1 (X,M_1) = 0\}\\
             &= \{ X\in \A_1 \mid \Ext_{\Gamma_1}^1 (M_1, X) = 0\}.
  \end{align*}
  We can see this by a straightforward computation or by \cite[Theorem~3]{V}.
    
  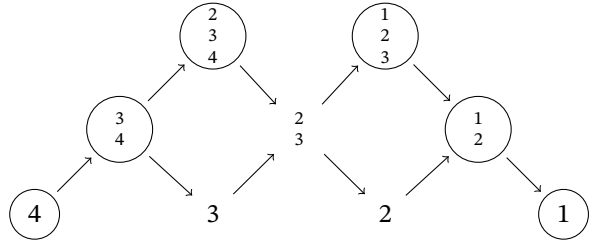
\begin{figure}[htbp]
    \centering
    \begin{tikzpicture}
      \diagram{d}{1em}{1em}{
        && |[draw, circle]|{\scriptsize \rep{2}{3\\ 4}} &&
        |[draw, circle]|{\scriptsize \rep{1}{2\\ 3}} &&\\
        & |[draw, circle]|{\scriptsize \rep{3}{4}} && {\scriptsize \rep{2}{3}} &&
        |[draw, circle]|{\scriptsize \rep{1}{2}} &\\
        |[draw, circle]|4 && 3 && 2 && |[draw, circle]|1\\
      };
      
      \path[encircled, font = \scriptsize, auto]
      (d-1-3) edge (d-2-4)
      (d-1-5) edge (d-2-6)
      
      (d-2-2) edge (d-1-3)
              edge (d-3-3)
      (d-2-4) edge (d-1-5)
              edge (d-3-5)
      (d-2-6) edge (d-3-7)
      
      (d-3-1) edge (d-2-2)
      (d-3-3) edge (d-2-4)
      (d-3-5) edge (d-2-6);
    \end{tikzpicture}
    \caption{The module category $\A_1$ of $\Gamma_1$.}
    \label{fig:modGamma1}
  \end{figure}
    
  Furthermore, in this depiction of $\A_1$, we follow the convention in
  representation theory and write
  \begin{center}
    \begin{tikzpicture}
      \begin{scope}
        \diagram{d}{1.25em}{1em}{
          1 = \bigl(\field & 0 & 0 & 0\bigr)\phantom{.}\\
          2 = \bigl(0 & \field & 0 & 0\bigr)\phantom{.}\\
          3 = \bigl(0 & 0 & \field & 0\bigr)\phantom{.}\\
          4 = \bigl(0 & 0 & 0 & \field\bigr)\phantom{.}\\
        };
        
        \path[->, font = \scriptsize, auto]
        (d-1-1) edge (d-1-2)
        (d-1-2) edge (d-1-3)
        (d-1-3) edge (d-1-4)
        
        (d-2-1) edge (d-2-2)
        (d-2-2) edge (d-2-3)
        (d-2-3) edge (d-2-4)
        
        (d-3-1) edge (d-3-2)
        (d-3-2) edge (d-3-3)
        (d-3-3) edge (d-3-4)
        
        (d-4-1) edge (d-4-2)
        (d-4-2) edge (d-4-3)
        (d-4-3) edge (d-4-4);
      \end{scope}
      
      \begin{scope}[xshift=4cm]
        \diagram{d}{1.25em}{1em}{
          {\scriptsize \rep{1}{2}}\hspace*{-1.5mm} = \bigl(\field & \field &
          0 & 0\bigr)\phantom{.}\\
          {\scriptsize \rep{2}{3}}\hspace*{-1.5mm} = \bigl(0 & \field & \field
          & 0\bigr)\\
          {\scriptsize \rep{3}{4}}\hspace*{-1.5mm} = \bigl(0 & 0 & \field
          & \field\bigr)\\
          \phantom{0}\\
        };
        
        \path[->, font = \scriptsize, auto]
        (d-1-1) edge node{$1$} (d-1-2)
        (d-1-2) edge (d-1-3)
        (d-1-3) edge (d-1-4)
        
        (d-2-1) edge (d-2-2)
        (d-2-2) edge node{$1$} (d-2-3)
        (d-2-3) edge (d-2-4)
        
        (d-3-1) edge (d-3-2)
        (d-3-2) edge (d-3-3)
        (d-3-3) edge node{$1$} (d-3-4);
      \end{scope}
      
      \begin{scope}[xshift=8cm]
        \diagram{d}{1.25em}{1em}{
          {\scriptsize \rep{1}{2\\ 3}}\hspace*{-1.5mm} = \bigl(\field & \field & \field
          & 0\bigr)\\
          {\scriptsize \rep{2}{3\\ 4}}\hspace*{-1.5mm} = \bigl(0 & \field & \field
          & \field\bigr).\\
          \phantom{0}\\
          \phantom{0}\\
        };
        
        \path[->, font = \scriptsize, auto]
        (d-1-1) edge node{$1$} (d-1-2)
        (d-1-2) edge node{$1$} (d-1-3)
        (d-1-3) edge (d-1-4)
        
        (d-2-1) edge (d-2-2)
        (d-2-2) edge node{$1$} (d-2-3)
        (d-2-3) edge node{$1$} (d-2-4);
      \end{scope}
    \end{tikzpicture}
  \end{center}
  The indecomposable projectives in $\A_1$ are denoted by
  \begin{equation*}
    P_1 = \hspace*{-1.5mm} {\scriptsize \rep{1}{2\\ 3}} \, , \qquad P_2 =
    \hspace*{-1.5mm} {\scriptsize \rep{2}{3\\4}} \, , \qquad P_3 =
    \hspace*{-1.5mm} {\scriptsize \rep{3}{4}} \qquad \text{and} \qquad P_4 = 4,
  \end{equation*}
  and, similarly, the indecomposable injectives in $\A_1$ are denoted by
  \begin{equation*}
    I_1 = 1, \qquad I_2 = \hspace*{-1.5mm} {\scriptsize \rep{1}{2}} \, ,
    \qquad I_3 = P_1 \qquad \text{and} \qquad I_4 = P_2.
  \end{equation*}
  
  As $\Gamma_1$ is a finite-dimensional $\field$-algebra, the global dimension
  of $\Gamma_1$ is the supremum of the projective dimensions of the simples,
  cf.\ \cite[Corollary~7.1.14]{McConnell-Robson}, and since there are only four
  simple modules in $\A_1$, $S_1 = 1$, $S_2 = 2$, $S_3 = 3$ and $S_4 = 4$, a
  straightforward calculation gives that $\pd S_1 = 2$, $\pd S_2 = \pd S_3 = 1$
  and $\pd S_4 = 0$. Thus $\gldim \Gamma_1 = 2$. By \autoref{rem:U_description},
  it follows that we have a $4$-angulated category of the form
  \begin{equation*}
    \U_1=\add \{M_1[2i] \mid i\in\Z\}\subseteq \derived^b(\A_1).
  \end{equation*}
    
  We will consider the two following $4$-angles in $\U$:
  \begin{center}
    \begin{tikzpicture}
      \diagram{d}{1em}{2em}{
        P_4 & P_2 & P_1 & I_1 & \Sigma^2 P_4\\
        P_4 & P_3 & P_1 & I_2 & \Sigma^2 P_4.\\
      };
      
      \path[->, font = \scriptsize, auto]
      (d-1-1) edge node{$f_1$} (d-1-2)
      (d-1-2) edge node{$f_2$} (d-1-3)
      (d-1-3) edge node{$f_3$} (d-1-4)
      (d-1-4) edge node{$f_4$} (d-1-5)
      
      (d-2-1) edge node{$g_1$} (d-2-2)
      (d-2-2) edge node{$g_2$} (d-2-3)
      (d-2-3) edge node{$g_3$} (d-2-4)
      (d-2-4) edge node{$g_4$} (d-2-5);
    \end{tikzpicture}
  \end{center}
  
  We want to compute the Toda bracket $\toda{h_4,h_3,h_2,f_1}$ where
  \begin{center}
    \begin{tikzpicture}
      \diagram{d}{2em}{2em}{
        P_4 & P_2 & I_2 & \Sigma^2 P_4 & \Sigma^2 P_3\\
      };
      
      \path[->, font = \scriptsize, auto]
      (d-1-1) edge node{$f_1$} (d-1-2)
      (d-1-2) edge node{$h_2$} (d-1-3)
      (d-1-3) edge node{$h_3$} (d-1-4)
      (d-1-4) edge node{$h_4$} (d-1-5);
    \end{tikzpicture}
  \end{center}
  is a diagram in $\derived^b(\A_1)$ such that $h_2 f_1 = h_3 h_2 = h_4 h_3 =
  0$.
  
  Let $h_2 = g_3f_2$, $h_3 = g_4$ and $h_4 = \Sigma^2 g_1$. In order to compute
  $\toda{\Sigma^2 g_1,g_4,g_3f_2,f_1}$ we consider the $4$-angle extension of
  $f_1$ listed above, and use the iterated cofiber Toda bracket. That is, we
  have the following commutative diagram.
  \begin{center}
    \begin{tikzpicture}
      \diagram{d}{2em}{2em}{
        P_4 & P_2 & P_1 & I_1 & \Sigma^2 P_4\\
        P_4 & P_2 & I_2 & \Sigma^2 P_4 & \Sigma^2 P_3\\
      };
      
      \path[->, font = \scriptsize, auto]
      (d-1-1) edge node{$f_1$} (d-1-2)
      \downequal{d}{1-1}{2-1}
      (d-1-2) edge node{$f_2$} (d-1-3)
      \downequal{d}{1-2}{2-2}
      (d-1-3) edge node{$f_3$} (d-1-4)
      edge[densely dashed] node{$\phi_1$} (d-2-3)
      (d-1-4) edge node{$f_4$} (d-1-5)
      edge[densely dashed] node{$\phi_2$}(d-2-4)
      (d-1-5) edge[densely dashed] node{$\psi$}(d-2-5)
      
      (d-2-1) edge node{$f_1$} (d-2-2)
      (d-2-2) edge node{$g_3f_2$} (d-2-3)
      (d-2-3) edge node{$g_4$} (d-2-4)
      (d-2-4) edge node{$\Sigma^2 g_1$} (d-2-5);
    \end{tikzpicture}
  \end{center}

  Note that the (minimal) projective resolution of $I_1$ and $I_2$ are given by
  \begin{center}
    \begin{tikzpicture}
      \diagram{d}{2em}{1.5em}{
        0 & P_4 & P_2 & P_1 & I_1\\
      };
      
      \path[->, font = \scriptsize, auto]
      (d-1-1) edge (d-1-2)
      (d-1-2) edge (d-1-3)
      (d-1-3) edge (d-1-4)
      (d-1-4) edge[->>] (d-1-5);
    \end{tikzpicture}
    \quad
    \begin{tikzpicture}
      \diagram{d}{2em}{1em}{
        \text{and}\\
      };
    \end{tikzpicture}
    \quad
    \begin{tikzpicture}
      \diagram{d}{2em}{1.5em}{
        0 & P_4 & P_3 & P_1 & I_2\\
      };
      
      \path[->, font = \scriptsize, auto]
      (d-1-1) edge (d-1-2)
      (d-1-2) edge (d-1-3)
      (d-1-3) edge (d-1-4)
      (d-1-4) edge[->>] (d-1-5);
    \end{tikzpicture}
  \end{center}
  respectively. Consider $\phi_1 = \alpha g_3$ for $\alpha\in \field$. Using the
  projective resolution of $I_2$ and the fact that $\U_1(P_2, P_3) = 0$, we get
  that there is a chain homotopy
  \begin{equation*}
    \alpha g_3 f_2 \simeq g_3 f_2
  \end{equation*}
  if and only if $\alpha = 1$. That is, $\phi_1 = g_3$. Note that $g_4g_3 = 0$
  so that $\phi_2$ must satisfy $\phi_2 f_3 = 0$. Since
  \begin{equation*}
    \U_1(I_1,\Sigma^2P_4) = \Ext_{\Gamma_1}^2(I_1,P_4) = \Span_\field\{f_4\}
  \end{equation*}
  and $f_4f_3=0$, it follows that $\phi_2$ must be a scalar multiple of $f_4$.
  To this end, consider $\phi_2 = \beta f_4$ for $\beta\in \field$. By a similar
  argument, consider $\psi = \gamma \Sigma^2 g_1$ for $\gamma\in \field$. Using
  the projective resolution of $I_1$ and again the fact that $\U_1(P_2, P_3) =
  0$, we get that
  \begin{equation*}
    \beta (\Sigma^2 g_1) f_4 \simeq \gamma (\Sigma^2 g_1) f_4
  \end{equation*}
  if and only if $\beta = \gamma$. Hence,
  \begin{equation}\label{eq:toda_Q1}
    \toda{\Sigma^2 g_1, g_4, g_3 f_2, f_1} = \Span_\field \{\Sigma^2 g_1\} =
    \U_1(\Sigma^2 P_4, \Sigma^2 P_3)
  \end{equation}
  where we have used the fact that $\U_1(P_4,P_3) = \Span_\field \{g_1\}$.

  We can also understand what $\toda{\Sigma^2 g_1, g_4, g_3 f_2, f_1}$ is by
  appealing to the indeterminacy subgroup. By \autoref{prop:brackets_are_cosets}
  we know that $\toda{\Sigma^2 g_1, g_4, g_3 f_2, f_1}$ is a coset of the
  indeterminacy subgroup
  \begin{equation*}
    (\Sigma^2 g_1)_\ast \U_1(\Sigma^2 P_4, \Sigma^2 P_4) + (\Sigma^2 f_1)^\ast
    \U_1(\Sigma^2 P_2, \Sigma^2 P_3).
  \end{equation*}
  Clearly,
  \begin{equation*}
    \U_1(\Sigma^2 P_4, \Sigma^2 P_4) = \Span_\field \{1_{\Sigma^2 P_4}\} \qquad
    \text{and} \qquad \U_1(\Sigma^2 P_2, \Sigma^2 P_3) = 0.
  \end{equation*}
  Hence, $\toda{\Sigma^2 g_1, g_4, g_3 f_2, f_1}$ is a $1$-dimensional affine
  subspace of a $1$-dimensional space and, so, \eqref{eq:toda_Q1} holds.
  
  Finally, we may get same result using juggling formulas. To see this consider
  the diagram
  \begin{center}
    \begin{tikzpicture}
      \diagram{d}{2em}{2em}{
        P_4 & P_2 & P_1 & I_2 & \Sigma^2 P_4 & \Sigma^2 P_3\\
      };
      
      \path[->, font = \scriptsize, auto]
      (d-1-1) edge node{$f_1$} (d-1-2)
      (d-1-2) edge node{$f_2$} (d-1-3)
      (d-1-3) edge node{$g_3$} (d-1-4)
      (d-1-4) edge node{$g_4$} (d-1-5)
      (d-1-5) edge node{$\Sigma^2 g_1$} (d-1-6);
    \end{tikzpicture}
  \end{center}
  in $\derived^b(\A_1)$. Then, by \ref{prop-6-6-5}, we have for $i = 3$ that
  \begin{equation*}
    \toda{\Sigma^2 g_1, g_4 g_3, f_2, f_1} = \toda{\Sigma^2 g_1, g_4, g_3 f_2, f_1}.
  \end{equation*}
  Note that the composition of two consecutive morphisms is zero in both of
  these Toda brackets as $g_4 g_3 = 0$. A straightforward computation then shows
  that
  \begin{equation*}
    \toda{\Sigma^2 g_1, g_4 g_3, f_2, f_1} = \toda{\Sigma^2 g_1, 0, f_2, f_1} =
    \U_1(\Sigma^2 P_4, \Sigma^2 P_3),
  \end{equation*}
  and so, \eqref{eq:toda_Q1} holds.
\end{example}

%
% Example from Vaso (with 2 dimensional hom spaces)
%
\begin{example}
  Let $Q_2$ denote the quiver with vertices and arrows as shown in \autoref{fig:quiver2}.
    
  \begin{figure}[htbp]
    \centering
    \begin{tikzpicture}
      \diagram{d}{3em}{1.5em}{
        &&& 1\\
        2 && 3 && 4 && 5\\
        &&& 6\\
      };
      
      \path[->, font = \scriptsize, auto]
      (d-1-4) edge node[swap]{$a$} (d-2-1)
              edge node{$b$} (d-2-3)
              edge node[swap]{$c$} (d-2-5)
              edge node{$d$} (d-2-7)
      
      (d-2-1) edge node[swap]{$a'$} (d-3-4)
      (d-2-3) edge node{$b'$} (d-3-4)
      (d-2-5) edge node[swap]{$c'$} (d-3-4)
      (d-2-7) edge node{$d'$} (d-3-4);
    \end{tikzpicture}
    \caption{The quiver $Q_2$.}
    \label{fig:quiver2}
  \end{figure}
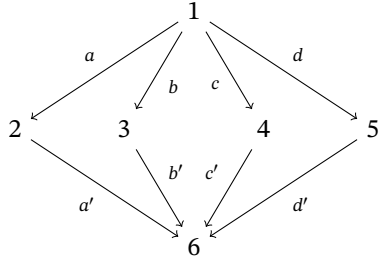
    
  Consider the ideal of the path algebra $\field Q_2$ given by the presentation
  \begin{equation*}
    I\coloneq \langle a'a + b'b + c'c , a'a + \lambda b'b + d'd\rangle,
  \end{equation*}
  for some $\lambda\in \field\setminus\{0, 1\}$, denote by $\Gamma_2 \coloneq
  \field Q_2/I$ the quotient algebra and denote by $\A_2=\mod \Gamma_2$ the
  category of finitely generated modules. We have the following indecomposable
  projective and injective modules
  \begin{equation*}
    \begin{aligned}
      P_1 &= \hspace*{-1.5mm} {\scriptsize \rep{1}{2~3~4~5\\ 6^2}} &&& P_2 &=
      \hspace*{-1.5mm} {\scriptsize \rep{2}{6}} &&& P_3 &= \hspace*{-1.5mm}
      {\scriptsize \rep{3}{6}} &&& P_4 &= \hspace*{-1.5mm} {\scriptsize \rep{4}{6}}
      &&& P_5 &= \hspace*{-1.5mm} {\scriptsize \rep{5}{6}} &&& P_6 &= 6\\
      I_1 &= 1 &&& I_2 &= \hspace*{-1.5mm} {\scriptsize \rep{1}{2}} &&& I_3
      &= \hspace*{-1.5mm} {\scriptsize \rep{1}{3}} &&& I_4 &= \hspace*{-1.5mm}
      {\scriptsize \rep{1}{4}} &&& I_5 &= \hspace*{-1.5mm} {\scriptsize \rep{1}{5}}
      &&& I_6 &= \hspace*{-1.5mm} {\scriptsize \rep{1^2}{2~3~4~5\\ 6}}
    \end{aligned}
  \end{equation*}
  where the morphisms $\field\to\field$ are the identity wherever possible in
  the quiver representation and where $P_1$ is represented by
  \autoref{fig:P1_in_tub}. Note that $I_6$ is dual to $P_1$.
    
  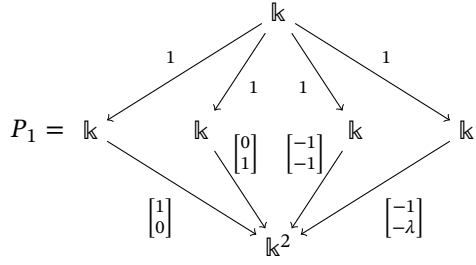
\begin{figure}[htbp]
    \centering
    \begin{tikzpicture}
      \begin{scope}[xshift=-3.2cm]
        \node at (0,0){$P_1=$};
      \end{scope}
      
      \begin{scope}
        \diagram{d}{3em}{1.5em}{
          &&& \field\\
          \field && \field && \field && \field\\
          &&& \field^2\\
        };
        
        \path[->, font = \scriptsize, auto]
        (d-1-4) edge node[swap]{$1$} (d-2-1)
                edge node{$1$} (d-2-3)
                edge node[swap]{$1$} (d-2-5)
                edge node{$1$} (d-2-7)
        
        (d-2-1) edge node[swap]{$\left[\begin{matrix}1\\0\end{matrix}\right]$} (d-3-4)
        (d-2-3) edge node[xshift=-6pt,yshift=2pt]{$\left[\begin{matrix}
              0\\1\end{matrix}\right]$} (d-3-4)
        (d-2-5) edge node[swap,xshift=6pt,yshift=2pt]{$\left[\begin{matrix}
              -1\\-1\end{matrix}\right]$} (d-3-4)
        (d-2-7) edge node{$\left[\begin{matrix}
              -1\\-\lambda\end{matrix}\right]$} (d-3-4);
      \end{scope}
    \end{tikzpicture}
    \caption{The quiver representation of the module $P_1$.}
    \label{fig:P1_in_tub}
  \end{figure}
  
  It is easily verified that $\gldim \Gamma_2 = 2$ by applying
  \cite[Corollary~7.1.14]{McConnell-Robson}. Consider the following projective
  resolutions of the simple modules:
  \begin{center}
    \begin{tikzpicture}
      \begin{scope}
        \diagram{d}{2em}{2em}{
          0 & P_6^2 & P_2\oplus P_3\oplus P_4\oplus P_5 & P_1 & S_1=I_1\\ 
        };
        
        \path[->, font = \scriptsize, auto]
        (d-1-1) edge (d-1-2)
        (d-1-2) edge node{$d_2^{I_1}$} (d-1-3)
        (d-1-3) edge node{$d_1^{I_1}$} (d-1-4)
        (d-1-4) edge[->>] (d-1-5);
      \end{scope}
      
      \begin{scope}[yshift=-.8cm]
        \diagram{d}{2em}{2em}{
          0 & P_6 & P_i & S_i\\ 
        };
        
        \path[->, font = \scriptsize, auto]
        (d-1-1) edge (d-1-2)
        (d-1-2) edge (d-1-3)
        (d-1-3) edge[->>] (d-1-4);
      \end{scope}
    \end{tikzpicture}
  \end{center}
  for all $i=2,\dots,5$, and note that $S_6=P_6$. Thus we see that $I_1$
  determines the global dimension to be 2.

  Thus we may apply \autoref{rem:U_description} and get that the full subcategory
  \begin{equation*}
    \U_2 = \add\{M_2[2i]\mid i\in\Z\}\subseteq\derived^b(\A_2)
  \end{equation*}
  has a $4$-angulated structure, where a calculation shows that
  \begin{equation*}
 	  M_2 \cong P_1\oplus \cdots\oplus P_6\oplus I_1\oplus\cdots \oplus I_6.
  \end{equation*}
  
  For this example we are only concerned with direct sums of objects of the form
  $N[2i]$ such that $N\in\add M_2,\ i\in\Z$.

  In this subcategory, we may describe all morphisms in the following way: for
  $i$ an integer, $P$ and $P'$ indecomposable projectives, and $I$ and $I'$
  indecomposable injectives, we have
  \begin{align*}
 		\U_2\bigl(P\oplus I,(P'\oplus I')[2i]\bigr)
    &\cong
      \begin{cases}
        \A_2(P,P'\oplus I')\oplus\A_2(I,I')
        & i=0,\\
        \Ext^2_{\Gamma_2}(I,P')
        & i=1,\\
        0 & i\neq 0,1.
      \end{cases}
  \end{align*}
  so the only potential non-trivial morphisms of $\U_2$ that we are interested
  in occur in the $\field$-vector spaces
  \begin{equation*}
    \A_2(P,P'),\quad\A_2(P,I'),\quad \A_2(I,I')\quad\text{and}\quad \Ext^2_{\Gamma_2}(I,P').
  \end{equation*}
    
  One may directly verify that, if we exclude the case $\U_2(I,P)$ for now, then
  $\U_2(M,N)$ is $0$- or $1$-dimensional except for three cases in which it is
  $2$-dimensional:
  \begin{align*}
    \U_2(P_6,P_1)&=\Span_{\field}\{p_{61}^1,p_{61}^2\},\\
    \U_2(P_1,I_6)&=\Span_{\field}\{\pi_{16}^1,\pi_{16}^2\},\\
    \U_2(I_6,I_1)&=\Span_{\field}\{i_{61}^1,i_{61}^2\}.
  \end{align*}
  We label the basis vectors of the $1$-dimensional spaces similarly, e.g., 
  \begin{equation*}	
    \U_2(P_2,I_2)=\Span_\field\{\pi_{22}\}
  \end{equation*}
  and for any indecomposable, either projective or injective, $M$, the
  endomorphism algebra is 1-dimensional
  \begin{equation*}
    \U_2(M,M)=\Span_\field\{1_M\}.
  \end{equation*}
  Finally we calculate that, in the case of $\U_2(I,P)$,
  \begin{equation*}
		\dim_\field\Ext^2_{\Gamma_2}(I_i,P_j)=\begin{cases}
			1 & \text{for }2\leq i=j\leq 5,\\
			1 & \text{for }i=1\text{ and }j\neq6,\\
			1 & \text{for }i\neq1\text{ and }j=6,\\
			2 & \text{for }i=1\text{ and }j=6,\\
			0 & \text{otherwise.}
		\end{cases}
  \end{equation*}
  We will label the generating extensions by
  \begin{equation*}
		\U_2(I_i,P_j)=\Span_\field\{\varepsilon_{ij}\}.
  \end{equation*}
  Note that for an injective $I$, since $\Omega^2I$ is either $P_6$ or $P_6^2$,
  all extensions may be realized as certain morphisms in $\add P_6$ which induce
  non-nullhomotopic morphisms in $\htpy^b (\proj\A_2)\simeq \derived^b(\A_2)$.
  In most cases there is an obvious choice of representative in $\add P_6$,
  however, particularly of note is $\Ext^2_{\Gamma_2}(I_1,P_1)$: recall that
  $\Omega^2I_1=P_6^2$ so we may realize the $1$-dimensional space
  $\Ext^2_{\Gamma_2}(I_1,P_1)$ as a quotient of the $4$-dimensional space
  \begin{equation*}
    \A_2(P_6^2,P_1)=\Span_\field\left\{
      \left[\begin{matrix}
          p_{61}^1 & 0
        \end{matrix}\right],
      \left[\begin{matrix}
          p_{61}^2 & 0
        \end{matrix}\right],
      \left[\begin{matrix}
          0& p_{61}^1
        \end{matrix}\right],
      \left[\begin{matrix}
          0& p_{61}^2
        \end{matrix}\right]
    \right\}
  \end{equation*}
  by the image of the morphism induced by $d_2^{I_1}$, the second differential
  of the projective resolution of $I_1$. It turns out that in this quotient, the
  morphism $\left[\begin{matrix} p_{61}^1 & 0
	   \end{matrix}\right]$
   is identified with the zero morphism, while the remaining three basis vectors
   are non-zero. To represent the morphism $\varepsilon_{11}$, the morphism
   $\left[\begin{matrix} p_{61}^2 & 0
	   \end{matrix}\right]$
   happens to be a convenient choice.

   Now let us calculate the iterated cofiber Toda bracket of the sequential
   diagram
   \begin{center}
     \begin{tikzpicture}	
       \diagram{d}{2em}{3em}{
         P_6 & P_2 & I_2 & I_1 & \Sigma^2 P_1.\\ 
       };
       
       \path[->, font = \scriptsize, auto]
       (d-1-1) edge node{$p_{62}$} (d-1-2)
       (d-1-2) edge node{$\pi_{22}$} (d-1-3)
       (d-1-3) edge node{$i_{21}$} (d-1-4)
       (d-1-4) edge node{$\varepsilon_{11}$} (d-1-5);
     \end{tikzpicture}
   \end{center}
   We need to make sure that consecutive composites are zero: One may easily
   verify that $\pi_{22}p_{62}=0$ and $i_{21}\pi_{22}=0$ already in $\mod\A_2$.
   Furthermore, we know that
    \begin{center}
      \begin{tikzpicture}
        \diagram{d}{1.5em}{1.5em}{
          \varepsilon_{11}i_{21}\colon I_2 & \Sigma^2 P_1\\
        };
        
        \path[->]
        (d-1-1) edge (d-1-2);
      \end{tikzpicture}
    \end{center}
    is trivial in $\U$ since $\Ext_{\Gamma_2}^2(I_2, P_1) = 0$. Thus, by
    \autoref{thm:main}, the iterated cofiber Toda bracket
    $\todacc{\varepsilon_{11},i_{21},\pi_{22},p_{62}}$ coincides with \emph{the}
    Toda bracket. It turns out that the diagram above is almost a $4$-angle in
    $\U_2$; consider the diagram defining the iterated cofiber Toda bracket:
    \begin{center}
      \begin{tikzpicture}	
        \diagram{d}{2em}{3em}{
          P_6 & P_2 & I_2 & I_1 & \Sigma^2 P_6\\
          P_6 & P_2 & I_2 & I_1 & \Sigma^2 P_1.\\ 
		    };
        
		    \path[->, font = \scriptsize, auto]
		    (d-1-1) edge node{$p_{62}$} (d-1-2)
		    (d-1-2) edge node{$\pi_{22}$} (d-1-3)
		    (d-1-3) edge node{$i_{21}$} (d-1-4)
		    (d-1-4) edge node{$\varepsilon_{16}^1$} (d-1-5)
        
		    (d-2-1) edge node{$p_{62}$} (d-2-2)
		    (d-2-2) edge node{$\pi_{22}$} (d-2-3)
		    (d-2-3) edge node{$i_{21}$} (d-2-4)
		    (d-2-4) edge node{$\varepsilon_{11}$} (d-2-5)
        
        \downequal{d}{1-1}{2-1}
        \downequal{d}{1-2}{2-2}
		    (d-1-3) edge[densely dashed] node{$\varphi_3$} (d-2-3)
		    (d-1-4) edge[densely dashed] node{$\varphi_4$} (d-2-4)
		    (d-1-5) edge[densely dashed] node{$\psi$} (d-2-5);
	    \end{tikzpicture}
    \end{center}
    We see that $p_{61}^2\in\todacc{\varepsilon_{11},i_{21},\pi_{22},p_{62}}$ by
    taking $\varphi_3\coloneq 1_{I_2}$ and $\varphi_4\coloneq 1_{I_1}$;
    recalling that $\varepsilon_{11}$ may be represented by
    $\left[\begin{matrix} p_{61}^2 & 0
      \end{matrix}\right]$,
    the commutativity of the last square can be witnessed by working with
    projective resolutions and chain homotopies:
    \begin{center}
	    \begin{tikzpicture}	
		    \diagram{d}{2em}{2em}{
          I_1 & \Sigma^2 P_6 & \Sigma^2 P_1 \\[-1em]
          P_6^2 & P_6  & P_1 \\
          P_2\oplus\cdots\oplus P_5 & 0 & 0\\
          P_1 & 0 & 0.\\
		    };
        
		    \path[->, font = \scriptsize, auto]
		    (d-1-1) edge node{$\varepsilon_{16}^1$} (d-1-2)
		    (d-1-2) edge node{$\Sigma^2p_{61}^2$} (d-1-3)
        
		    (d-2-1) edge node{$\left[\begin{matrix}
		          1&0
		        \end{matrix}\right]$} (d-2-2)
		    (d-2-2) edge node{$p_{61}^2$} (d-2-3)
        
		    (d-3-1) edge (d-3-2)
		    (d-3-2) edge (d-3-3)
        
		    (d-4-1) edge (d-4-2)
		    (d-4-2) edge (d-4-3)
        
		    (d-2-1) edge node[swap]{$d^{I_1}_2$} (d-3-1)
		    (d-3-1) edge node[swap]{$d^{I_1}_1$} (d-4-1)
        
		    (d-2-2) edge (d-3-2)
		    (d-3-2) edge (d-4-2)
        
		    (d-2-3) edge (d-3-3)
		    (d-3-3) edge (d-4-3);
	    \end{tikzpicture}
    \end{center}
    It is furthermore clear that $(\Sigma^2p^1_{61})\varepsilon_{16}^1=0$, since
    $\left[\begin{matrix} p_{61}^1 & 0
      \end{matrix}\right]$
    is the exact basis element that is identified with 0. Thus one may already
    deduce that
    \begin{equation*}
      p_{61}^2+\Span_\field\{p_{61}^1\}\subset
      \todacc{\varepsilon_{11},i_{21},\pi_{22},p_{62}}.
    \end{equation*}
    We see that the indeterminacy is actually exactly
    \begin{equation*}
	    (\varepsilon_{11})_*\U_2(\Sigma^2 P_6, I_1) +
      (\Sigma^2p_{62})^*\U_2(\Sigma^2 P_2, \Sigma^2P_1) = \Span_\field\{p_{61}^1\}
    \end{equation*}
    since $\U_2(\Sigma^2 P_6, I_1)=0$ and $p_{21}p_{62}=p^1_{61}$. In
    conclusion,
    \begin{equation*}
	    \toda{\varepsilon_{11},i_{21},\pi_{22},p_{62}}=p_{61}^2+\Span_\field\{p_{61}^1\}.
    \end{equation*}
    % \smallskip

    Next, let us use Toda brackets to determine whether the diagram
    \begin{center}
	    \begin{tikzpicture}	
		    \diagram{d}{2em}{3em}{
          P_6 & P_1 & I_6 & I_1 & \Sigma^2 P_6\\ 
		    };
        
		    \path[->, font = \scriptsize, auto]
		    (d-1-1) edge node{$p_{61}^1$} (d-1-2)
		    (d-1-2) edge node{$\pi_{16}^2$} (d-1-3)
		    (d-1-3) edge node{$-i_{61}^1$} (d-1-4)
		    (d-1-4) edge node{$\varepsilon_{16}^1$} (d-1-5);
	    \end{tikzpicture}
    \end{center}
    is a $4$-angle. First, we check that the morphisms compose to the zero
    morphism. We may observe that $\pi_{12}^2\pi_{61}^1 = 0$ by restricting our
    attention to vertex 6, and similarly we may observe that
    $-i_{61}^1\pi_{12}^2 = 0$ by restricting our attention to vertex 1. Finally,
    we may observe that $\varepsilon_{16}^1 (-i_{61}^1)=0$ by considering the
    morphisms in $\htpy(\proj\A_2)$:
    \begin{center}
	    \begin{tikzpicture}	
		    \diagram{d}{2em}{4.75em}{
          I_6 & I_1 & \Sigma^2 P_6 \\[-.5em]
          P_6 & P_6^2  & P_6 \\
          P_2\oplus\cdots\oplus P_5 & P_2\oplus\cdots\oplus P_5  & 0\\
          P_1 & P_1 & 0.\\
		    };
        
		    \path[->, font = \scriptsize, auto]
		    (d-1-1) edge node{$-i_{61}^1$} (d-1-2)
		    (d-1-2) edge node{$\varepsilon_{16}^1$} (d-1-3)
        
		    (d-2-1) edge node{$\left[\begin{matrix}
              0\\-\lambda
		        \end{matrix}\right]$} (d-2-2)
		    (d-2-2) edge node{$\left[\begin{matrix}
              1 & 0
		        \end{matrix}\right]$} (d-2-3)
        
		    (d-3-1) edge node{$0\oplus1\oplus1\oplus1$} (d-3-2)
		    (d-3-2) edge (d-3-3)
        
		    (d-4-1) edge (d-4-2)
		    (d-4-2) edge (d-4-3)
        
		    (d-2-1) edge node[swap]{$d^{I_6}_2$} (d-3-1)
		    (d-3-1) edge node[swap]{$d^{I_6}_1$} (d-4-1)
        
		    (d-2-2) edge node[swap]{$d^{I_1}_2$} (d-3-2)
		    (d-3-2) edge node[swap]{$d^{I_1}_1$} (d-4-2)
        
		    (d-2-3) edge (d-3-3)
		    (d-3-3) edge (d-4-3);
	    \end{tikzpicture}
    \end{center}
    We actually have something stronger: we calculated above that
    $(\Sigma^2p_{61}^1)\varepsilon_{16}^1=0$, so this sequence might be Yoneda
    exact and therefore has a real chance of being a 4-angle.

    We will use that we have a $4$-angle extension of $p_{61}^1$; this can be
    calculated iteratively by first taking the cokernel of the morphism in
    question in $\A_2$ and then finding a left $\U_2$-approximation. This yields
    the $4$-angle
    \begin{center}
	    \begin{tikzpicture}	
		    \diagram{d}{2em}{5em}{
			    P_6 & P_1 & I_2\oplus I_6 & I_1\oplus I_2 & \Sigma^2 P_6\\
		    };
        
		    \path[->, font = \scriptsize, auto]
		    (d-1-1) edge node{$p_{61}^1$} (d-1-2)
		    (d-1-2) edge node{$\left[\begin{matrix}
              \pi_{12}\\\pi_{16}^2
		        \end{matrix}\right]$} (d-1-3)
		    (d-1-3) edge node{$\left[\begin{matrix}
              i_{21} & -i^2_{61} \\ 0 & i_{62}
		        \end{matrix}\right]$} (d-1-4)
		    (d-1-4) edge node{$\left[\begin{matrix}
              \varepsilon_{16}^1 & \frac1\lambda\varepsilon_{26}
		        \end{matrix}\right]$} (d-1-5)
        ;
	    \end{tikzpicture}
    \end{center}
    which allows us to compare the diagram in question to calculate the Toda
    bracket $\todacc{\varphi_{16}^1,-i_{61}^1,\pi_{16}^2,p_{61}^1}$, i.e.,
    
    \begin{center}
	    \begin{tikzpicture}	
		    \diagram{d}{2em}{5em}{
			    P_6 & P_1 & I_2\oplus I_6 & I_1\oplus I_2 & \Sigma^2 P_6\\
          P_6 & P_1 & I_6 & I_1 & \Sigma^2 P_6.\\ 
		    };
        
		    \path[->, font = \scriptsize, auto]
		    (d-1-1) edge node{$p_{61}^1$} (d-1-2)
		    (d-1-2) edge node{$\left[\begin{matrix}
              \pi_{12}\\\pi_{16}^2
		        \end{matrix}\right]$} (d-1-3)
		    (d-1-3) edge node{$\left[\begin{matrix}
              i_{21} & -i^2_{61} \\ 0 & i_{62}
		        \end{matrix}\right]$} (d-1-4)
		    (d-1-4) edge node{$\left[\begin{matrix}
              \varepsilon_{16}^1 & \frac1\lambda\varepsilon_{26}
		        \end{matrix}\right]$} (d-1-5)
        
		    (d-2-1) edge node{$p_{61}^1$} (d-2-2)
		    (d-2-2) edge node{$\pi_{16}^2$} (d-2-3)
		    (d-2-3) edge node{$-i_{61}^1$} (d-2-4)
		    (d-2-4) edge node{$\varepsilon_{16}^1$} (d-2-5)
        
        \downequal{d}{1-1}{2-1}
        \downequal{d}{1-2}{2-2}
		    (d-1-3) edge[densely dashed] node{$\varphi_3$} (d-2-3)
		    (d-1-4) edge[densely dashed] node{$\varphi_4$} (d-2-4)
		    (d-1-5) edge[densely dashed] node{$\psi$} (d-2-5);
	    \end{tikzpicture}
    \end{center}
    It is easily checked that
    \begin{equation*}
	    \varphi_3=\left[\begin{matrix}
	        0 & 1
        \end{matrix}\right]
	    \quad\text{and}\quad
	    \varphi_3=\left[\begin{matrix}
	        0 & -i_{21}
        \end{matrix}\right]
    \end{equation*}
    are valid choices, since $i_{21}i_{62}=i_{61}^1$. We know from the previous
    example that $\varepsilon_{16}^1 i_{21}=0$, so may thus conclude that
    \begin{equation*}
	    0\in \todacc{\varphi_{16}^1,-i_{61}^1,\pi_{16}^2,p_{61}^1}.
    \end{equation*}
    Let us now consider the indeterminacy. We know that
    \begin{equation*}
	    \U_2(\Sigma^2P_6, I_1)=0 \quad\text{and}\quad \U_2(\Sigma^2P_1, \Sigma^2P_6)
      = 0
    \end{equation*}
    so the indeterminacy is trivial, and thus we have
    \begin{equation*}
	    \todacc{\varphi_{16}^1,-i_{61}^1,\pi_{16}^2,p_{61}^1}=\{0\}.
    \end{equation*}
    It follows that
    $1\not\in\todacc{\varphi_{16}^1,-i_{61}^1,\pi_{16}^2,p_{61}^1}$, and so by
    \autoref{prop:NAngleCriterion}, the diagram in question is not a $4$-angle.
\end{example}

%
% Longer vs. higher
%
\section{Long bracket versus higher order bracket}\label{sec:LongerHigher}

Assume that $\C$ is an $n$-angulated category coming from an $(n-2)$-cluster
tilting subcategory of a triangulated category $\T$ with suspension functor
$\Sigma$ as in \cite[Theorem~1]{GKO}.

In the triangulated category $\T$, one may calculate the \emph{$n$-fold Toda
  bracket} of a diagram of the shape of \autoref{eq:seq} by applying an
inductive construction to the fiber-cofiber Toda bracket
$\todafc{f_{n-2},f_{n-1},f_n}$. The question then arises of whether a relation
exists between the $n$-fold Toda bracket and the Toda bracket of the
$n$-angulated subcategory $\C$. To address this question we will use a
definition of $n$-fold Toda bracket due to Shipley \cite[Definition
A.2]{Shipley02} and developed by Sagave \cite{Sagave08}. We most closely follow
the notational conventions of \cite{CF}.

\begin{definition}\label{def:SSbracket}
  For $k\geq2$ and a diagram 
  \begin{center}
    \begin{tikzpicture}    
      \diagram{d}{2em}{3em}{
        X_1 & X_2 & X_3 & \cdots & X_k & X_{k+1}\\ 
      };
      
      \path[->, font = \scriptsize, auto]
      (d-1-1) edge node{$f_1$} (d-1-2)
      (d-1-2) edge node{$f_2$} (d-1-3)
      (d-1-3) edge node{$f_3$} (d-1-4)
      (d-1-4) edge node{$f_{k-1}$} (d-1-5)
      (d-1-5) edge node{$f_k$} (d-1-6);
    \end{tikzpicture}
  \end{center}
  we consider a $(k-1)$-filtered object $X$ based on a subdiagram consisting of
  the morphisms $(f_{k - 1},\dots,f_3,f_2)$, cf.\ \cite[Definition 5.4]{CF}
  (with $n = k - 1$ and $Y_i = X_{i + 2}$), so that we have a diagram
  \begin{center}
    \begin{tikzpicture}    
      \diagram{d}{2em}{3em}{
        0=F_0X & F_1X & F_2X & \cdots & F_{k-1}X=X\\ 
      };
      
      \path[->, font = \scriptsize, auto]
      (d-1-1) edge node{$i_0$} (d-1-2)
      (d-1-2) edge node{$i_1$} (d-1-3)
      (d-1-3) edge node{$i_2$} (d-1-4)
      (d-1-4) edge node{$i_{k-2}$} (d-1-5);
    \end{tikzpicture}
  \end{center}
  along with the following data:    
  \begin{enumerate}[label=(\arabic*), ref={Definition
      \thedefinition.(\arabic*)}]
  \item For each $0\leq j\leq k-2$ a distinguished triangle
    \begin{center}
      \begin{tikzpicture}    
        \diagram{d}{2em}{3em}{
          F_jX & F_{j+1}X & \Sigma^j X_{k-j} & \Sigma F_jX.\\ 
        };
        
        \path[->, font = \scriptsize, auto]
        (d-1-1) edge node{$i_j$} (d-1-2)
        (d-1-2) edge node{$q_{j+1}$} (d-1-3) 
        (d-1-3) edge node{$e_j$} (d-1-4);
      \end{tikzpicture}
    \end{center}
  \item \label{def:SSbracket_2} For each $1\leq j\leq k-2$ a commutative diagram
    \begin{center}
      \begin{tikzpicture}    
        \diagram{d}{2em}{1em}{
          \Sigma^j X_{k-j} & & \Sigma^j X_{k + 1 - j}\\ 
          & \Sigma F_jX.\\
        };
        
        \path[->, font = \scriptsize, auto]
        (d-1-1) edge node{$\Sigma^jf_{k - j}$} (d-1-3)
        (d-1-1) edge node[swap]{$e_j$} (d-2-2)
        (d-2-2) edge node[swap]{$\Sigma q_j$} (d-1-3);
      \end{tikzpicture}
    \end{center}
  \end{enumerate}
  The \Def{$k$-fold Toda bracket} $\todaSS{f_k,\dots,f_2,f_1}\subseteq
  \T(\Sigma^{k-2}X_1,X_{k+1})$ in the sense of Shipley--Sagave is defined to be
  all composites $\mu(\Sigma^{k-2}\nu)$ occuring in the commutative diagram
  \begin{equation}\label{eq:ss}
    \begin{aligned}
      \begin{tikzpicture}    
        \diagram{d}{2em}{3em}{
          & X_{k}\\
          \Sigma^{k-2}X_1 & X & X_{k+1}\\
          & \Sigma^{k-2}X_2 \\ 
        };
        
        \path[->, font = \scriptsize, auto]
        (d-1-2) edge node{$\sigma_X$} (d-2-2)
        (d-1-2) edge node{$f_k$} (d-2-3)
        (d-2-1) edge[densely dashed] node{$\Sigma^{k-2}\nu$} (d-2-2)
        (d-2-1) edge node[swap]{$\Sigma^{k-2}f_1$} (d-3-2)
        (d-2-2) edge node{$q_{k-1}$} (d-3-2)
        (d-2-2) edge[densely dashed] node{$\mu$} (d-2-3);
      \end{tikzpicture}
    \end{aligned}
  \end{equation}
  where $\sigma_X$ denotes the composition of the inverse isomorphism $q_1^{-1}$
  with the composite of the filtration $i_{k-2}\cdots i_2 i_1$.
\end{definition}

\begin{remark}
  Note that all objects are 1-filtered based on the empty diagram since
  \ref{def:SSbracket_2} is vacuous in that case.
\end{remark}

\begin{theorem}\label{thm:LongerHigher}
  Let $\T$ be a triangulated category with an $(n-2)$-cluster tilting
  subcategory $\C$ closed under $\Sigma^{n-2}$ and let
  \begin{center}
    \begin{tikzpicture}    
      \diagram{d}{2em}{3em}{
        X_1 & X_2 & X_3 & \cdots & X_n & X_{n+1}\\ 
      };
      
      \path[->, font = \scriptsize, auto]
      (d-1-1) edge node{$f_1$} (d-1-2)
      (d-1-2) edge node{$f_2$} (d-1-3)
      (d-1-3) edge node{$f_3$} (d-1-4)
      (d-1-4) edge node{$f_{n-1}$} (d-1-5)
      (d-1-5) edge node{$f_n$} (d-1-6)
      ;
    \end{tikzpicture}
  \end{center}
  be a diagram in $\C$ such that $f_{i+1}f_i=0$ for $1\leq i\leq n-1$. Then
  \begin{equation*}
    \toda{f_n,\dots,f_2,f_1}=(-1)^{\sum_{\ell=1}^{n-3} \ell}\todaSS{f_n,\dots,f_2,f_1}
  \end{equation*}
  where $\toda{f_n,\dots,f_2,f_1}$ denotes the Toda bracket in the $n$-angulated
  category $\C$.
\end{theorem}

\begin{remark}\label{rem:LongerHigher}
  \mbox{}
  \begin{enumerate}
  \item Note that if $f_{i+1}f_i\neq 0$ for some $i$ then the $n$-fold Toda
    bracket is empty, and so in this case the $n$-fold Toda bracket agrees with
    the fiber-cofiber Toda bracket of the $n$-angulated structure.
  \item In light of \autoref{thm:LongerHigher}, \autoref{prop:NAngleCriterion}
    provides an alternate proof of \cite[Theorem~4.1.12]{JM} for standard
    $n$-angulated categories in the sense of \cite[Definition~2.2.2]{JM}.
  \item In \cite[Theorem~5.11]{CF}, the $(n-2)!$ ways of recursively computing
    an $n$-fold Toda bracket were shown to agree up to an explicit sign. The
    standard definition $\todaSS{f_n,\dots,f_2,f_1}$, denoted $\TT_0 \TT_0
    \cdots \TT_0$ in \emph{op.\ cit.}\ starts the recursive process with the
    morphisms $f_{n-2}, f_{n-1}$ and $f_n$. The other extreme choice, which we
    denote by $\todaSS{f_n,\dots,f_2,f_1}'\subset \T(\Sigma^{n-2}X_1,X_{n+1})$,
    written $\TT_0 \TT_1 \TT_2 \cdots \TT_{n-3}$ in \emph{op.\ cit.}, starts
    with the morphisms $f_1, f_2$ and $f_3$. Combining \cite[Theorem~5.11]{CF}
    with \autoref{thm:LongerHigher}, we obtain
    \begin{equation*}
      \toda{f_n,\dots,f_2,f_1}=(-1)^{\sum_{\ell=1}^{n-3}\ell}\todaSS{f_n,\dots,f_2,f_1}
      = \todaSS{f_n,\dots,f_2,f_1}'.
    \end{equation*}
  \end{enumerate}
\end{remark}

\begin{proof}[Proof of \autoref{thm:LongerHigher}]  
  By the assumption that $\C$ is $(n-2)$-cluster tilting, \cite[Proposition
  4.10]{Sagave08} implies that the $n$-fold Toda bracket is a coset of the
  subgroup
  \begin{equation*}
    (f_n)_*\C(\Sigma^{n-2}X_1,X_{n})+(\Sigma f_1)^*\C(\Sigma^{n-2}X_2,X_{n+1})
    \subseteq \C(\Sigma^{n-2}X_1,X_{n+1}).
  \end{equation*}
  Since the diagram is in the full subcategory $\C$ and $f_{i+1}f_i=0$, the Toda
  bracket of the \mbox{$n$-angulated} category is defined and is also a coset of
  the preceding subgroup. Therefore it suffices to exhibit a common element of
  the two cosets.

  We now assume $n \geq 4$ for the remainder of the proof. The statement holds
  in the case $n = 3$: All triangulated categories are themselves 1-cluster
  tilting and $3$-angulated Toda brackets agree with triangulated Toda brackets.

  \textbf{Outline of argument.} We consider an element
  $\psi\in\toda{f_n,\dots,f_2,f_1}$ (which exists by the assumption
  $f_{i+1}f_i=0$) and modify the data this provides to construct a new morphism
  $\tilde\psi\colon\Sigma^{n-2}X_1\to X_{n+1}$ which {we then exhibit as an
    element of both the cosets we are comparing.}

  \textbf{Setup.} The element $\psi$ has a presentation as the composition
  $\Sigma^{n-2}(\beta_1^{n-1} \cdots \beta_1^2 \beta_1^1)$ sitting in a diagram
  of the shape of \autoref{figure:fcbracket}. The particular type of
  $n$-angulation provides further data: All the rows in
  \autoref{figure:fcbracket} are $n$-angles and therefore there exist diagrams
  in $\T$ of the shape
  \begin{center}
    \begin{tikzpicture}
      \diagram{d}{2em}{1.5em}{
        & Z_2^i && Z^i_3 && \cdots &&& Z^i_{n - 1} &\\
        Z_1^i && A^i_{3} && A^i_{4} & \cdots && A^i_{n-1} && Z^i_n\\
      };
      
      \path[->, font = \scriptsize, auto]
      (d-1-2) edge node{$z^i_2$} (d-1-4)
              edge node{$c^i_2$} (d-2-3)
      (d-1-4) edge node{$z^i_3$} (d-1-6)
              edge node{$c^i_3$} (d-2-5)
      (d-1-6) edge node{$z^i_{n - 2}$} (d-1-9)
      (d-1-9) edge node{$z^i_{n - 1}$} (d-2-10)
      
      (d-2-1) edge node{$z^i_1$} (d-1-2)
      (d-2-3) edge node{$b^i_{2}$} (d-1-4)
              edge[suspension,->] node{$a_3^i$} (d-2-1)
      (d-2-5) edge[suspension,->] node{$a_4^i$} (d-2-3)
      (d-2-6) edge[suspension,->] node{$a_5^i$} (d-2-5)
      (d-2-8) edge node{$b^i_{n-2}$} (d-1-9)
              edge[suspension,->] node{$a_{n-1}^i$} (d-2-6)
      (d-2-10) edge[suspension,->] node{$a_{n}^i$} (d-2-8);
    \end{tikzpicture}
  \end{center}
  {with $Z_j^i\in\C$ and} such that $z^i_i\colon Z_i^i\to Z_{i+1}^i$ denotes the
  morphism $f_i\colon X_i\to X_{i+1}$. For each morphism of $n$-angles
  $\beta^i_\bullet\colon Z^i_\bullet\to Z^{i+1}_\bullet$ in
  \autoref{figure:fcbracket} there exist compatible morphisms $\alpha_j^i\colon
  A_j^i\to A_j^{i+1}$ coming from inductively applying the morphism axiom for
  triangulated categories, (TR3), see for instance \cite{Martensen}.

  \textbf{Filtration.} We now build an $(n-1)$-filtered object $X$ based on
  $(f_{n-1},\dots,f_3,f_2)$: It follows from the existence of the trivial
  distinguished triangle $(TX_n)_\bullet[-1]$ that we may take the first
  non-trivial step of the filtration to be $F_1X=X_n$ and $q_1=1_{X_n}$. Hence,
  $F_1 X$ is a $1$-filtered object based on the empty diagram. To satisfy
  \ref{def:SSbracket_2} for $j=1$, we further define $e_1=\Sigma f_{n-1}$. From
  the trapezoid diagram presentation of the $n$-angle {$Z^{n-1}_\bullet$}, we
  are provided a distinguished triangle
  \begin{center}
    \begin{tikzpicture}    
      \diagram{d}{2em}{3em}{
        X_n & \Sigma A_{n-1}^{n-1} & \Sigma X_{n-1} & \Sigma X_{n}\\ 
      };
      
      \path[->, font = \scriptsize, auto]
      (d-1-1) edge node{$a^{n-1}_n$} (d-1-2)
      (d-1-2) edge node{$\Sigma b_{n-2}^{n-1}$} (d-1-3)
      (d-1-3) edge node{$\Sigma f_{n-1}$} (d-1-4);
    \end{tikzpicture}
  \end{center}
  and so {choosing} $F_2X=\Sigma A_{n-1}^{n-1}$, $i_1=a_n^{n-1}$ and $q_2=\Sigma
  b^{n-1}_{n-2}$ makes $F_2X$ a $2$-filtered object based on $(f_{n-1})$. We now
  build the rest of the filtration recursively. We have distinguished triangles
  \begin{center}
    \begin{tikzpicture}    
      \diagram{d}{2em}{6em}{
        \Sigma^{j-1} A^m_{m+1} & \Sigma^j A^m_m & \Sigma^j X_m & \Sigma^j
        A^m_{m+1}\\ 
      };
      
      \path[->, font = \scriptsize, auto]
      (d-1-1) edge node{$(-1)^{j-1}\Sigma^{j-1}a^m_{m+1}$} (d-1-2)
      (d-1-2) edge node{$\Sigma^jb^m_{m-1}$} (d-1-3)
      (d-1-3) edge node{$\Sigma^jc^m_{m}$} (d-1-4);
    \end{tikzpicture}
  \end{center}
  where $m=n-j$, coming from appropriate rotations of the distinguished
  triangles provided by the trapezoid diagrams for each $j$. Let $1\leq j\leq
  n-2$. We now define a sequence of diagrams, $\Dd_j$, recursively for each $j$,
  see \autoref{fig:induction}.
  \begin{figure}
    \centering
    \begin{tikzpicture}    
      \diagram{d}{3em}{6em}{
        & \Sigma^j X_{m-1} & \Sigma^j X_m\\
        & \Sigma^j Z^m_{m-1}\\
        \Sigma^{j-1} A_{m+1}^{m} & \Sigma^j A_m^m & \Sigma^j X_{m} & \Sigma^j
        A_{m+1}^{m}\\
        F_jX & F_{j+1}X & \Sigma^j X_m & \Sigma F_jX\\
        \Sigma^{j-1} A_{m+1}^{n-1} & \Sigma^j A_m^{n-1} & \Sigma^j Z^{n-1}_m &
        \Sigma^j A_{m+1}^{n-1}\\
      };
      
      \path[->, font = \scriptsize, auto]
      (d-1-2) edge node{$\Sigma^j f_{m-1}$} (d-1-3)
      (d-1-2) edge[left] node{$\Sigma^j\beta^{m-1}_{m-1}$} (d-2-2)
      \downequal{d}{1-3}{3-3}
      (d-2-2) edge[left, pos=.4] node{$\Sigma^j c_{m-1}^m$} (d-3-2)
      (d-2-2) edge node{$\Sigma^j z^m_{m-1}$} (d-3-3)
      
      (d-3-1) edge node{$(-1)^{j-1}\Sigma^{j-1}a_{m+1}^m$} (d-3-2)
      (d-3-2) edge node{$\Sigma^j b_{m-1}^m$} (d-3-3)
      (d-3-3) edge node{$\Sigma^j c^m_m$} (d-3-4)
      
      (d-4-1) edge node{$i_j$} (d-4-2)
      (d-4-2) edge node{$q_{j+1}$} (d-4-3)
      (d-4-3) edge node{$e_j$} (d-4-4)
      
      (d-5-1) edge node{$(-1)^{j-1}\Sigma^{j-1}a^{n-1}_{m+1}$} (d-5-2)
      (d-5-2) edge node{$\Sigma^jb^{n-1}_{m-1}$} (d-5-3)
      (d-5-3) edge node{$\Sigma^j c^{n-1}_m$} (d-5-4);
      
      \path[->, font = \scriptsize, midway]
      (d-3-1) edge node[inner sep=1pt,fill=white]{$\Sigma^{j-1}(\theta_{j-1}
        \alpha^{m}_{m+1})$} (d-4-1)
      (d-3-2) edge[densely dashed,left] node{$\Sigma^{j}\theta_j$} (d-4-2)
      \downequal{d}{3-3}{4-3}
      (d-3-4) edge node[inner sep=1pt, fill=white]{$\Sigma^j(\theta_{j-1}
        \alpha^m_{m+1})$} (d-4-4)
      
      (d-4-1) edge[left] node{$\Sigma^{j-1}\theta_{j-1}'$} (d-5-1)
      (d-4-2) edge[densely dashed,left] node{$\Sigma^{j}\theta_j'$} (d-5-2)
      (d-4-3) edge node[inner sep=1pt, fill=white]{$\Sigma^j(\beta^{n-2}_m\cdots
        \beta^m_m)$} (d-5-3)
      (d-4-4) edge[left] node{$\Sigma^j\theta_{j-1}'$} (d-5-4)
      
      (d-1-2) edge[line width=4pt, white,out=-50,in=50] (d-4-2)
      (d-1-2) edge[right,densely dotted,out=-50,in=50, pos=.15] node{$\Sigma^{-1}
        e_{j+1}$} (d-4-2)

      (d-3-3) -- (d-4-4) node[pos=.5,xshift=0em,font=\small] (a) {$\Sigma^j
        \square^1_j$}
      (d-4-3) -- (d-5-4) node[pos=.5,xshift=0em,font=\small] (b) {$\Sigma^j
        \square_j^2$};
    \end{tikzpicture}
    \caption{The solid part of the diagram depicts $\Dd_j$ for some $j$.}
    \label{fig:induction}
  \end{figure}
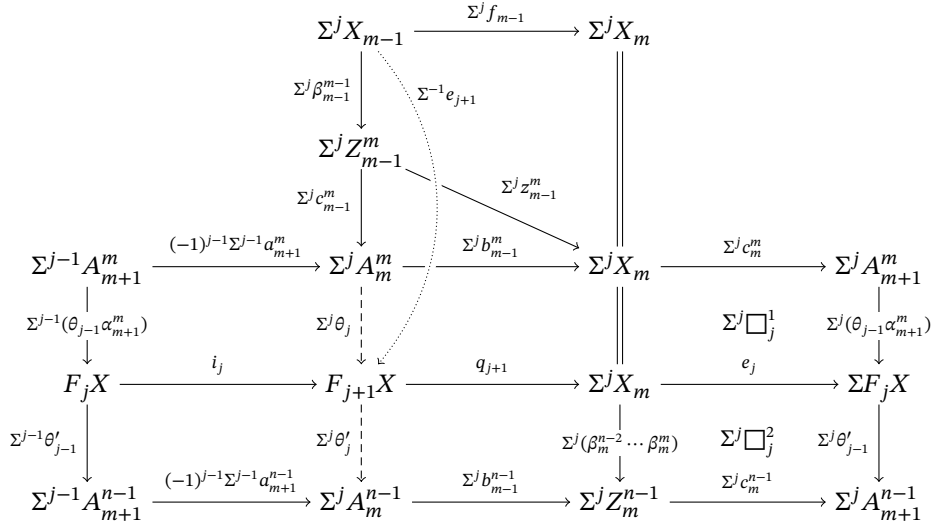
  Note that the existence of $\Dd_j$ is dependent on $\Dd_{j-1}$ being
  commutative such that one may apply the morphism axiom of triangulated
  categories, (TR3), to define $\theta_{j-1}$ and $\theta_{j-1}'$. For $\Dd_1$
  the three bottom rows are taken to be identical and all vertical morphisms
  between these rows are chosen to be identities, i.e., $\theta_1=\theta_1'=1_{
    A_{n-1}^{n-1}}$. We see that $\Dd_2$ also commutes and therefore that
  $\theta_2$ and $\theta_2'$ exist: the square $\Sigma^2\square_2^2$ commutes
  since, by construction, $\alpha^{n-2}_{n-1}$ satisfies the equation
  \begin{equation*}
    \Sigma^2(\theta_{1}'e_2)
    =\Sigma^2(\theta_1'\theta_1\alpha^{n-2}_{n-1}c^{n-2}_{n-2})
    =\Sigma^2(\alpha^{n-2}_{n-1}c^{n-2}_{n-2})
    =\Sigma^2 (c^{n-1}_{n-2}\beta^{n-2}_{n-2})
  \end{equation*}
  and $\Sigma^2\square_2^1$ commutes by the same compatibility relation. In
  fact, $\Sigma^j\square_j^1$ commutes for all $j$ by the compatibility of
  $\alpha^m_{m+1}$, i.e.,
  \begin{equation*}
    \Sigma^{j}(\theta_{j-1}\alpha^m_{m+1}c^m_m)=
    \Sigma\left(\Sigma^{j-1}(\theta_{j-1}c^{m+1}_m\beta^m_m)\right)=e_j.
  \end{equation*}
  Let $\widetilde\Dd_j$ denote the subdiagram of $\Dd_j$ that excludes the
  bottom row and the maps that target it. We have now established that all
  $\widetilde\Dd_j$ exist and are commutative. We see that for each $j$ we now
  have the desired distinguished triangles, and the commutativity of the
  vertical middle rectangle of $\widetilde\Dd_j$ provides the desired
  factorizations
  \begin{equation*}
    \Sigma^{j+1}f_{n-(j+1)}=\Sigma^{j+1}f_{m-1}=(\Sigma q_{j+1})e_{j+1}.
  \end{equation*}
  This means that setting $j=n-2$, this process produces an $(n-1)$-filtered
  object $X\coloneq F_{n-1}X$ based on $(f_{n-1},\dots,f_3,f_2)$ along with the
  important relation
  \begin{equation*}
    \Sigma^{n-2}f_{1}=q_{n-1}(\Sigma^{-1}e_{n-1})=q_{n-1}\Sigma^{n-2}(\theta_{n-2}c^{2}_1
    \beta^1_1).
  \end{equation*}

  \textbf{Commutativity of $\Dd_j$.} Now let us see that the full diagrams
  $\Dd_j$ and, therefore, the morphisms $\theta_j'$ exist for $j\geq 3$. To this
  end, we apply induction:
    
  \emph{Base case.} The morphisms $\theta_1'$ and $\theta_2'$ exist since
  $\Dd_1$ and $\Dd_2$ are defined and commutative.
    
  \emph{Induction step.} Fix $2\leq k\leq n-2$. We assume that $\Dd_j$ exists
  and commutes for $1\leq j\leq k-1$ such that we have the morphisms
  $\theta_j'$. We must prove that $\Sigma^k\square_k^2$ commutes in order to
  show that $\theta_k'$ exists. To this end, we note that $b^{n-1}_m$ factors
  through the coequalizer of $\theta_{k-1}'\theta_{k-1}$ and
  $\alpha^{n-2}_{m+1}\cdots\alpha^{m+1}_{m+1}$, i.e.,
  \begin{equation*}
    b^{n-1}_m\theta_{k-1}'\theta_{k-1}=\beta^{n-2}_{m+1}\cdots\beta^{m+1}_{m+1}b^{m+1}_m
    =b^{n-1}_m\alpha^{n-2}_{m+1}\cdots\alpha^{m+1}_{m+1}
  \end{equation*}
  by the induction hypothesis. Therefore the square $\square_k^2$, which we
  denote as an equation
  \begin{equation*}
    \gamma_0\coloneq c_{m}^{n-1}\beta^{n-2}_m\cdots\beta^{m+1}_m\beta^m_m -
    \theta_{k-1}'\Sigma^{-k}e_k,
  \end{equation*}
  commutes from the point of view of $b^{n-1}_m$, i.e.,
  \begin{align*}
    b^{n-1}_m\circ\gamma_0
    &=b^{n-1}_{m}(c_{m}^{n-1}\beta^{n-2}_m\cdots\beta^{m+1}_m\beta^m_m
      -\theta_{k-1}'\Sigma^{-k}e_k)\\
    & =b^{n-1}_{m}(\alpha^{n-2}_{m+1}\cdots\alpha^{m+1}_{m+1}
      \alpha^m_{m+1} c_m^m-\theta_{k-1}'\theta_{k-1}\alpha^m_{m+1}
      c_m^m)=0.
  \end{align*}
  Thus we may factor $\gamma_0$ through the fiber of $b^{n-1}_m$ via a morphism
  $\gamma_1\colon X_m\to \Sigma^{-1}A^{n-1}_{m+2}$, meaning that we have a
  factorization
  \begin{equation*}
    \gamma_0 = (\Sigma^{-1}a^{n-1}_{m+2})\gamma_1.
  \end{equation*}
  We note that $(\Sigma^{-1}b_{m+1}^{n-1})\gamma_1\in
  \C(X_m,\Sigma^{-1}Z^{n-1}_{m+2}) = 0$ by the $(n-2)$-cluster tilting condition
  since $X_m$ and $Z^{n-1}_{m+2}$ are objects in $\C$, so we may lift $\gamma_1$
  to the fiber of $\Sigma^{-1}b^{n-1}_{m+1}$ via a morphism $\gamma_2\colon
  X_m\to \Sigma^{-2}A_{m+3}^{n-1}$. Repeating the argument, we note that
  $(\Sigma^{-2}b_{m+2}^{n-1})\gamma_2=0$, so we may factor $\gamma_2$ through
  the fiber of $\Sigma^{-2}b^{n-1}_{m+2}$. We may continue this process to
  obtain a diagram
  \begin{center}
    \begin{tikzpicture}    
      \diagram{d}{2em}{4em}{
        \Sigma^k X_m & \\
        \Sigma^k A^{n-1}_{m+1} & \Sigma^{k-1} A^{n-1}_{m+2} & \cdots & \Sigma^2
        A^{n-1}_{n-1} & \Sigma X_n \\
        \Sigma^k Z^{n-1}_{m+1} & \Sigma^{k-1} Z^{n-1}_{m+2} && \Sigma^2 X_{n-1} \\
      };
      
      \path[->, font = \scriptsize, auto]
      (d-1-1) edge node{$\Sigma^k \gamma_0$} (d-2-1)
      (d-1-1) edge node[inner sep=1pt]{$\Sigma^k\gamma_1$} (d-2-2)
      (d-1-1) edge[out=-5, in=165] node[inner sep=1pt,fill=white,yshift=-6pt]{$\Sigma^k
        \gamma_{k-2}$} (d-2-4)
      (d-1-1) edge[out=0, in=165] node[inner sep=1pt,fill=white,yshift=-5pt]{$\Sigma^k
        \gamma_{k-1}$} (d-2-5)
      
      (d-2-2) edge node{$\Sigma^{k-1} a^{n-1}_{m+2}$} (d-2-1)
      (d-2-3) edge node{$\Sigma^{k-2} a^{n-1}_{m+3}$} (d-2-2)
      (d-2-4) edge node{$\Sigma^{2} a^{n-1}_{n-1}$} (d-2-3)
      (d-2-5) edge node{$\Sigma a^{n-1}_{n}$} (d-2-4)
      
      (d-2-1) edge node{$\Sigma^{k} b^{n-1}_{m}$} (d-3-1)
      (d-2-2) edge node{$\Sigma^{k-1} b^{n-1}_{m+1}$} (d-3-2)
      (d-2-4) edge node{$\Sigma^{2} b^{n-1}_{n-2}$} (d-3-4);
    \end{tikzpicture}
  \end{center}
  until we factor $\Sigma^k\gamma_0$ through the fiber of $\Sigma^{2}
  b^{n-1}_{n-2}$ by a morphism $\Sigma^k\gamma_{k-1}\colon \Sigma^kX_m \to\Sigma
  X_n$. However,
  \begin{equation*}
    \C(\Sigma^k X_m,\Sigma X_n)\cong \C(\Sigma^{k-1}X_m,X_n)=0
  \end{equation*}
  since $1\leq k-1\leq n-3$, and thus we conclude that $\gamma_0$ factors
  through the zero morphism, so $\gamma_0=0$ and therefore the square
  $\Sigma^k\square_k^2$, and hence the diagram $\Dd_k$, commutes. Thus we have
  proven that $\theta_j$, $\theta_j'$ and $\Dd_j$ exist for all $1\leq j\leq
  n-2$.

  \textbf{Producing a common element.} Recalling that $i_1=a^{n-1}_n$, we get
  from repeated application of the identity
  \begin{equation*}
    (\Sigma^{j}\theta_j')i_j=(-1)^{j-1}\Sigma^{j-1} (a^{n-1}_{m+1}\theta_{j-1}')
  \end{equation*}
  the commutativity relation 
  \begin{align*}
    &(\Sigma^{n-2}\beta^{n-1}_1)(\Sigma^{n-2}\theta_{n-2}')i_{n-2}i_{n-3}\cdots i_1\\
    &\quad =(\Sigma^{n-2}\beta^{n-1}_1)(-1)^{n-3}(\Sigma^{n-3}a^{n-1}_{3})
      (-1)^{n-4}(\Sigma^{n-4}a^{n-1}_4) \cdots (-1)(\Sigma a_{n-1}^{n-1})
      (\Sigma\theta_1')i_1\\
    &\quad =\varepsilon(\Sigma^{n-2}\beta^{n-1}_1)(\Sigma^{n-3}a^{n-1}_{3})
      (\Sigma^{n-4}a^{n-1}_4) \cdots (\Sigma a_{n-1}^{n-1})a^{n-1}_n\\
    &\quad =\varepsilon(\Sigma^{n-2}\beta^{n-1}_1)z^{n-1}_n =\varepsilon f_{n+1}
  \end{align*}
  where $\varepsilon=(-1)^{\sum_{\ell=1}^{n-3}\ell}$. Therefore, taking
  $\nu=\theta_{n-2}c_1^2\beta^{1}_1$ and
  $\mu=\varepsilon\Sigma^{n-2}(\beta^{n-1}_1\theta'_{n-2})$ we get a commutative
  diagram in the shape of \autoref{eq:ss} which exhibits the membership relation
  \begin{equation*}
    \tilde\psi\coloneq\varepsilon\Sigma^{n-2}(\beta^{n-1}_1\theta_{n-2}'\theta_{n-2}c_1^2
    \beta^1_1)\in\todaSS{f_n,\dots,f_2,f_1}.
  \end{equation*}
  We furthermore claim that
  \begin{equation*}
    \varepsilon\tilde\psi\in \todai{f_n,\dots,f_2,f_1}{n-1}
  \end{equation*}
  which completes the proof that
  \begin{equation*}
    \toda{f_n,\dots,f_2,f_1} = \varepsilon\todaSS{f_n,\dots,f_2,f_1}.
  \end{equation*}
  To prove the claim, note that the vertical composite of the commutative
  squares $\square_{n-2}^1$ and $\square_{n-2}^2$ is exactly the leftmost square
  in the diagram
  \begin{center}
    \begin{tikzpicture}    
      \diagram{d}{3em}{2.5em}{
        X_1 & X_2 & X_3 & \cdots & X_{n-1}\\
        Z^{n-1}_1 & Z^{n-1}_2 & Z^{n-1}_3 & \cdots & X_{n-1} & X_n & \Sigma^{n-2}
        Z^{n-1}_1\\
        &&&&&X_n & X_{n+1}\\ 
      };
      
      \path[->, font = \scriptsize, auto]
      (d-1-1) edge node{$f_1$} (d-1-2)
      (d-1-2) edge node{$f_2$} (d-1-3)
      (d-1-3) edge node{$f_3$} (d-1-4)
      (d-1-4) edge node{$f_{n-2}$} (d-1-5)
      
      (d-2-1) edge node{$z^{n-1}_1$} (d-2-2)
      (d-2-2) edge node{$z^{n-1}_2$} (d-2-3)
      (d-2-3) edge node{$z^{n-1}_3$} (d-2-4)
      (d-2-4) edge node{$z^{n-1}_{n-2}$} (d-2-5)
      (d-2-5) edge node{$f_{n-1}$} (d-2-6)
      (d-2-6) edge node{$z^{n-1}_n$} (d-2-7)
      
      (d-3-6) edge node{$f_n$} (d-3-7)
      
      (d-1-1) edge node{$\theta_{n-2}'\theta_{n-2}\beta^{1}_1$} (d-2-1)
      (d-1-2) edge node{$\beta^{n-2}_2\cdots\beta^{2}_{2}$} (d-2-2)
      (d-1-3) edge node{$\beta^{n-2}_3\cdots\beta^{3}_{3}$} (d-2-3)
      \downequal{d}{1-5}{2-5}
            
      \downequal{d}{2-6}{3-6}
      (d-2-7) edge node[inner sep=2pt]{$\Sigma^{n-2}\beta^{n-1}_1$} (d-3-7);
    \end{tikzpicture}
  \end{center}
  which exhibits the membership
  $\varepsilon\tilde\psi\in\todai{f_n,\dots,f_2,f_1}{n-1}$.
\end{proof}

\begin{corollary}\label{cor:Massey}
  Let $\field$ be a field and let $\A$ be a small DG category over $\field$ such
  that the homotopy category $H^0(\dgmod\A)$ of the DG category of right DG
  $\A$-modules is triangulated and has a $(n-2)$-cluster tilting subcategory
  $\U$ closed with respect to the shift $[n-2]$ such that $\U$ is $n$-angulated.
  For a diagram as \autoref{eq:seq} in $\U$, we have the identity
  \begin{equation*}
    \massey{f_n,\dots,f_2,f_1}[n-2] = -\toda{f_n,\dots,f_2,f_1}
  \end{equation*}
  where the left side is the Massey product of the morphisms $f_1,f_2,\dots,f_n$
  and the right side is the Toda bracket in the $n$-angulated category $\U$.
\end{corollary}

\begin{proof}
  By \cite[Theorem 4.2.6]{JM} we have the identity
  \begin{equation*}
    \massey{f_n,\dots,f_2,f_1}[n-2] = (-1)^{\sum_{\ell=1}^{n-1} \ell}
    \todaSS{f_n,\dots,f_2,f_1}
  \end{equation*}
  so since
  \begin{equation*}
    (-1)^{\sum_{\ell=1}^{n-1}\ell}\cdot(-1)^{\sum_{\ell=1}^{n-3}\ell}=(-1)^{2
      \left(\sum_{\ell=1}^{n-3}\ell\right)+2n-3}=-1
  \end{equation*}
  the result follows from applying \autoref{thm:LongerHigher}.
\end{proof}

\begin{remark}\label{rem:Massey}
  As a sanity check, for $n=3$, take $\U$ to be the full homotopy category of
  $\dgmod \A$. Then the previous result specializes to
  \begin{equation*}
    \massey{f_3,f_2,f_1}[1] = -\toda{f_3,f_2,f_1}
  \end{equation*}
  which is what we should expect.
\end{remark}

\end{document}